\documentclass[11pt,a4paper]{article}

\RequirePackage{ifthen}
\newboolean{START@even} % set true for the book first page is even
\ifthenelse{\boolean{START@even}}
{ \setlength{\oddsidemargin}{-2mm}
  \setlength{\evensidemargin}{8mm} }
{}

\usepackage{authblk}
\RequirePackage{yfonts}
\RequirePackage[active]{srcltx}
\RequirePackage{stmaryrd}
\RequirePackage[usenames,dvipsnames,svgnames,table]{xcolor}
\RequirePackage{amsmath,amsfonts,amssymb,bm,scalerel}
\RequirePackage{amsthm}
\RequirePackage{mathtools}
\RequirePackage{array}
\RequirePackage{enumerate}
\usepackage{tcolorbox}
\usepackage{booktabs}
\RequirePackage{algorithm}
\RequirePackage{algorithmic}
\usepackage{setspace}
\usepackage[font=footnotesize, margin=1.5cm]{caption}
\usepackage{subcaption}
\usepackage{tikz}
\RequirePackage[percent]{overpic}
\RequirePackage{url}
\RequirePackage{fancyhdr}
\RequirePackage{lscape}
\RequirePackage[hidelinks]{hyperref}
\RequirePackage{listings}
\usepackage{enumitem}

\hypersetup{
    colorlinks,
    linkcolor={red},
    citecolor={blue},
    urlcolor={green}
}

\newtheorem{lemma}{Lemma}
\numberwithin{lemma}{section}
\newtheorem{assumption}{Assumption}
\numberwithin{assumption}{section}
\newtheorem{corollary}{Corollary}
\numberwithin{corollary}{section}
\newtheorem{definition}{Definition}
\numberwithin{definition}{section}
\newtheorem{proposition}{Proposition}
\numberwithin{proposition}{section}
\newtheorem{remark}{Remark}
\numberwithin{remark}{section}
\newtheorem{theorem}{Theorem}
\numberwithin{theorem}{section}
\newcommand{\R}{\mathbb{R}}

\newcommand{\xx}{\boldsymbol}
\newcommand{\www}[1]{\widehat{#1}}

\newcommand{\yy}{\displaystyle}
\newcommand{\abs}[1]{\left| {#1} \right|}
\newcommand{\Qh}{Q^m_{h}}
\newcommand{\Vh}{ \mathbf{V}^\ell_h}
\newcommand{\norm}[1]{{\Vert {#1}\Vert}}
\newcommand{\normVh}[1]{{\Vert {#1}\Vert}_{\Vh}}
\newcommand{\normQh}[1]{{\Vert {#1}\Vert}_{\Qh}}

\newcommand{\normE}[1]{{\Vert {#1}\Vert}_{\textrm{\upshape{E}}}}
\newcommand{\normQ}[1]{{\Vert {#1}\Vert}_{L^2(\Omega)}}
\newcommand{\normV}[1]{{\Vert {#1}\Vert}_{\xx V}}
\newcommand{\seminormJ}[1]{\abs{#1}_{\mathsf{J}}}
\newcommand{\seminorm}[1]{| {#1} |}

\newcommand{\jump}[2][]{\left\llbracket {#2} \right\rrbracket_{#1}}
\newcommand*\circled[1]{\tikz[baseline=(char.base)]{
            \node[shape=circle,draw,inner sep=0.8pt] (char) {#1};}}

\begingroup\makeatletter\ifx\SetFigFont\undefined%
\gdef\SetFigFont#1#2#3#4#5{%
  \reset@font\fontsize{#1}{#2pt}%
  \fontfamily{#3}\fontseries{#4}\fontshape{#5}%
  \selectfont}%
\fi\endgroup%

\begin{document}

%%%%%%%%%%%%%%%%%%%%%%%%%%%%%%%%%%%%%%%

\title
{
Stability analysis of polytopic Discontinuous Galerkin approximations of the
Stokes problem with applications to fluid-structure interaction problems
}

\author
{
Paola F. Antonietti$^{\sharp}$ and Lorenzo Mascotto$^{\flat}$ and Marco
Verani$^{\sharp}$ and Stefano Zonca$^{\sharp}$
}

%%%%%%%%%%%%%%%%%%%%%%%%%%%%%%%%%%%%%%%

\maketitle
\begin{center}
{
    \small $^\sharp$ MOX -- Modelling and Scientific Computing\\
    Dipartimento di Matematica, Politecnico di Milano\\
    Piazza Leonardo da Vinci, 20133 Milano, Italy\\
    {
        \tt
        paola.antonietti@polimi.it\\
        marco.verani@polimi.it\\
        stefano.zonca@polimi.it\\[2em]
    }
    \small $^\flat$ Fakult\"at f\"ur Mathematik, Universit\"at Wien, Austria\\
    {
        \tt
        lorenzo.mascotto@univie.ac.at\\[2em]
    }
}
\end{center}

\date{}

\vspace*{0.5cm}

\vspace*{0.5cm}

\begin{abstract}
\noindent We present a stability analysis of the Discontinuous Galerkin method
on polygonal and polyhedral meshes (PolyDG) for the Stokes problem.  In
particular, we analyze the discrete \textit{inf-sup} condition for different
choices of the polynomial approximation order of the velocity and pressure
approximation spaces.  To this aim, we employ a generalized \textit{inf-sup}
condition with a pressure stabilization term. We also prove a priori
$hp$-version error estimates in suitable norms.  We numerically check the
behaviour of the \textit{inf-sup} constant and the order of convergence with
respect to the mesh configuration, the mesh-size, and the polynomial degree.
Finally, as a relevant application of our analysis, we consider the PolyDG
approximation for a fluid-structure interaction problem and we numerically
explore the stability properties of the method.
\end{abstract}

\section{Introduction}\label{sec:intro}

It is well known that a crucial aspect involving the stability of the numerical
scheme associated with the Stokes problem is the \textit{inf-sup} condition that
establishes a constraint in the choice of the velocity and pressure discrete
spaces; see, e.g, \cite{brezzi1974existence,BrezziFortin}.  This aspect, in the
context of polygonal methods, is still under investigation and only few results
are present in the literature; see, e.g.,
\cite{cockburn2002local,di2010discrete,aghili2015hybridization,di2016discontinuous,da2017divergence,burman:hal-02519896}.

The Discontinuous Galerkin (DG) method handles meshes with elements of general
shape and has proved to be suited for the approximation of fluid and structure
models, possibly involving moving domains, see, e.g.
\cite{dumbser2018staggered,tavelli2018arbitrary}.  The discrete \textit{inf-sup}
condition for DG methods has been analyzed in the following works.
In \cite{cockburn2002local}, the Local DG method for the Stokes problem is
formulated in a conservative way, by introducing the stress as unknown.  Here,
meshes with hanging nodes and elements of different shape are considered,
provided that they are affinely-equivalent to an element of a fixed set of
reference elements.  Moreover, an \textit{inf-sup} condition and optimal order
estimates are proven, when the pair of polynomials of degree $k$ and $k-1$ is
chosen for the velocity and pressure spaces.  However, the formulation requires
a stability term for both the velocity and the pressure.
In \cite{schotzau2003stabilized}, the \textit{inf-sup} condition is proven for a
pressure stabilized formulation on hexahedral meshes allowing hanging nodes,
when the pair $\mathbb{Q}_k-\mathbb{Q}_k$ is chosen.
In \cite{di2010discrete,dipietro_ern_2010}, the authors show the
\textit{inf-sup} condition for equal-order approximation $\mathbb{P}_k$ for both
the velocity and pressure in the case of a pressure stabilized formulation on
meshes consisting of elements of various shape, provided that each element is
affinely-equivalent to one in a fixed set of reference elements, and admitting
hanging nodes.
In \cite{schotzau2002mixed}, the authors propose a mixed DG formulation without
pressure stabilization for the Stokes problem and show a priori error estimates.
The \textit{inf-sup} condition is proven for the pair of spaces
$\mathbb{Q}_k-\mathbb{Q}_{k-1}$ on tensor product meshes, possibly with hanging
nodes.
In \cite{toselli2002hp}, the \textit{inf-sup} condition is proven for the pairs
of spaces $\mathbb{Q}_k-\mathbb{Q}_{k-1}$ and $\mathbb{Q}_k-\mathbb{Q}_{k-2}$
without any pressure stabilization on quadrilateral and hexahedral meshes with
hanging nodes; see also
\cite{stenberg1996mixed,schotzau1999mixed,toselli2003mixed}.
Numerical tests showing the dependence of the \textit{inf-sup} constants are
performed for the pairs of spaces $\mathbb{Q}_{k}-\mathbb{Q}_{k'}$, with $k' =
k, k-1, k-2$.
In~\cite{hansbo2002discontinuous}, the \textit{inf-sup} condition is proven on
triangular and tetrahedral meshes without any pressure stabilization term for
the pair of spaces $\mathbb{P}_k-\mathbb{P}_{k-1}$ employing the
Brezzi-Douglas-Marini spaces.
In \cite{girault2005discontinuous}, the pair of spaces
$\mathbb{P}_k-\mathbb{P}_{k-1}$ with the Crouzeix-Raviart elements is used to
prove the \textit{inf-sup} condition on triangular meshes.

In this work, we consider the Discontinuous Galerkin method on polygonal and
polyhedral grids (PolyDG) that extends the standard DG method to polytopic
meshes; see, e.g.,
\cite{Basetal12,antonietti2013hp,WiKuMiTaWeDa2013,CangianiGeorgoulisHouston_2014,antonietti2016review,cangiani2016hp,antoniettim1,antoniettiF2020}.
In this framework, we study the discrete stability and well-posedness for the
Stokes problem, by presenting an analysis that covers at once the two- and
three- dimensional cases. Under suitable assumptions, we prove that the
\textit{inf-sup} constant is independent of the mesh size.  Notwithstanding, it
is not robust with respect to the polynomial degree and this restriction
propagates to the convergence analysis, with a deterioration of the convergence
in terms of the polynomial degree.  However, we provide numerical evidence that
the discrete \emph{inf-sup} constant has a  much milder dependence on the
polynomial degree in practice.  Moreover, the mesh assumptions seem to be too
restrictive and, in fact, the method results to be \textit{inf-sup} stable also
for pathological configurations.  In the two-dimensional case, we numerically
assess the robustness of the \textit{inf-sup} constant with respect to the mesh
size and the polynomial degree for different types of mesh elements, including
elements with degenerating edges, and we numerically estimate the order of
convergence to the mesh size and the polynomial degree.

Besides, with the aim of further exploring the relevance of our stability
analysis, we consider a fluid-structure interaction (FSI) problem where both the
Stokes and the elastodynamics equations are solved based on employing the PolyDG
method. In fact, it is well known that the study of FSI problems is of paramount
importance in many engineering and biomedical applications; see, e.g.,
\cite{kamakoti2004fluid,bouaanani2014effects,xu2013large,terahara2020heart,ghosh2020numerical,picelli2020topology},
where a fluid, for instance modeled via the Stokes equations, interacts with a
structure, modeled via the elastodynamics equations. In particular, a special
class of FSI applications that requires a lot of effort from the numerical
viewpoint arises under the large deformations condition occurring in
time-dependent processes.  Indeed, to correctly model such problems, ad-hoc
techniques are mandatory to deal with the movement of the structures.  A
classical strategy to overcome this issue is the employment of the Arbitrary
Lagrangian Eulerian (ALE) approach.  It consists in deforming the fluid grid
according to the structure displacement, yet maintaining a ``honouring'' mesh at
the fluid-structure interface and generating an arbitrary deformation of the
elements in the interior of the fluid mesh; see, e.g.,
\cite{donea2,hron2006monolithic,tello2020fluid}.  Another way that preserves the
alignment of the fluid and structure grids at the interface is to use approaches
based on remeshing and mesh-adaptation techniques; see, e.g.,
\cite{tezduyar2007modelling,borker2019mesh}.  A different category of approaches
are based on employing unfitted meshes that allow to keep the fluid grid fixed
in time, while the structure mesh is free to move; see, e.g.,
\cite{glowinski2001fictitious,Mittal05,zhang2007immersed,gerstenberger2008extended,Griffith12,Borazjani13,massing2015nitsche,Court15,Alauzet16,boffi2017fictitious,fedele2017patient,ager2019poro,burmanFF20}.
Often, this requires the handling of polygonal and polyhedral elements appearing
in the fluid mesh, e.g., due to the intersection between the fluid and structure
elements, and in the solid mesh, e.g., due to the presence of hanging nodes;
see, e.g., \cite{antonietti2019numerical,da2020equilibrium}.  For this kind of
approaches, it is mandatory that the underlining discretization methods can
robustly and efficiently support meshes made of arbitrarily shaped elements. In
this respect, a deep  understanding  of the stability properties of the
numerical scheme with respect to possibly pathological meshes is of crucial
importance.

The paper is organized as follows.  Section~\ref{sec:model} introduces the
transient Stokes problem and its PolyDG approximation.  In
Section~\ref{sec:wellposed}, we prove the well-posedness of the PolyDG
approximation of the (stationary) Stokes problem, with a particular emphasis to
the discrete \textit{inf-sup} condition.  In particular, in
Section~\ref{sec:computations_inf_sup} we estimate the discrete \textit{inf-sup}
constant and numerically evaluate it for different choices of the discrete
velocity and pressure spaces and for different grids.  Section
\ref{section:abstract-convergence} is devoted to the proof of a priori error
estimates of the Stokes problem.  In Section~\ref{sec:time},  we introduce a
fluid-structure interaction problem and we present its fully-discrete PolyDG
approximation. In Section~\ref{sec:examplesFSI} we show some numerical results
for the Stokes and FSI problems.  Finally, in Section~\ref{sec:conclusions}, we
draw some conclusions.

In the sequel, the notation~$\lesssim $ and~$\gtrsim$ means that the
inequalities are valid up to multiplicative constants that are independent of
the discretization parameters, but might depend on the physical parameters of
the underlying problem.

%In \cite{cockburn2009equal}, the \textit{inf-sup} condition is demonstrated on
%triangular and quadrilateral meshes for the pair $\mathbb{P}_k-\mathbb{P}_k$
%employing the Brezzi-Douglas-Marini (BDM) spaces by means of a pressure
%stabilization term.

% In \cite{schotzau2002mixed}, dagli spazi Q_k, costruiscono uno spazio RT_k,
% tale che
% \int_e (\Pi v - v) n q = 0
% Per la dimostrazione, splittano p = \bar{p} + \tilde{p}

% In \cite{toselli2002hp}, sullo spazio discreto vale che
% \int_e (\Pi v - v) n q = 0
% Per la dimostrazione, splittano p = \bar{p} + \tilde{p}

% In \cite{hansbo2002discontinuous}, lo spazio BDM ha delle proprieta'
% aggiuntive sul bordo dell'elemento

% In \cite{girault2005discontinuous}, lo spazio CR sulle velocita' ha delle
% proprieta' aggiuntive sul bordo dell'elemento

\section{The transient Stokes problem}\label{sec:model}
Having in mind the PolyDG discretization of FSI problems as a reference
application,  in this section we consider the transient Stokes problem which
reads as follows: given a final time~$T>0$ and~$\xx f$ a (regular) forcing term,
find the velocity $\xx{u}=\xx{u}(t)$ and the pressure $p=p(t)$ such that, for
all $t\in (0,T]$,
\begin{subequations}
\label{eq:stokesPb}
    \begin{align}
    & \rho \partial_t \xx{u} - \mu \Delta \xx{u} + \nabla p = \xx f
    && \text{in} \; \Omega,  \label{eq:stokesEq1} \\
    & \nabla \cdot \xx{u} = 0
    && \text{in} \; \Omega, \label{eq:stokesEq2}\\
     & \xx{u} = 0
    && \text{on} \; \partial \Omega. \label{eq:stokesbc}
    \end{align}
\end{subequations}
Problem \eqref{eq:stokesPb} is supplemented with sufficiently regular initial conditions $\xx{u}( \xx{x} , 0 ) = \xx{u}^0(\xx{x})$ in~$\Omega$.
To guarantee the well-posedness of the problem, we prescribe that $p \in L^2_0 (\Omega)$, where~$L^2_0(\Omega)$ is the space of~$L^2(\Omega)$ functions with zero average over~$\Omega$.

We introduce the functional spaces
$$
\xx V = \lbrace \xx v \in \lbrack H^1(\Omega) \rbrack^d, d = 2,3, \textrm{ such
that } \xx v|_{\partial \Omega} = 0 \rbrace
$$
and $Q = L^2_0(\Omega)$ and endow them with the norms
\[
\Vert \xx v \Vert_{\xx V} := \Vert \mu^{\frac{1}{2}} \nabla \xx v \Vert_{L^2(\Omega)} \quad \text{and} \quad \Vert q \Vert_Q := \Vert q \Vert _{L^2(\Omega)}.
\]
The weak formulation of problem~\eqref{eq:stokesPb} reads as
follows: find $( \xx{u} , p ) \in \xx V \times Q$, such that, for all~$t\in (0,T]$,
\begin{equation}
\label{eq:stokesWeak}
\begin{aligned}
 & (\rho \partial_t \xx u, \xx v)_\Omega + a ( \xx{u}, \xx{v} )
+ b ( p , \xx{v} ) - b ( q , \xx{u} ) = (\xx f , \xx v)_\Omega
&& \forall ( \xx{v}, q ) \in \xx V \times Q,
\end{aligned}
\end{equation}
where
\[
\begin{split}
a: \xx V \times \xx V \rightarrow \R, \quad a ( \xx{u}, \xx{v} ) = \displaystyle
\int_\Omega \mu \nabla \xx{u} :     \nabla \xx{v},\\
b : Q \times \xx{V} \rightarrow \R, \quad     b ( p , \xx{v} ) = - \displaystyle
\int_\Omega p \nabla \cdot \xx{v},
\end{split}
\]
and $(\cdot , \cdot)_\Omega$ denotes the $L^2$-inner product over the domain~$\Omega$.

It is well-known that the bilinear form~$b(\cdot, \cdot)$ satisfies a continuous \textit{inf-sup} condition; see, e.g., \cite{BrezziFortin}.
More precisely, there exists a universal positive constant depending only on~$\Omega$
such that, to all~$q \in L^2_0(\Omega)$, we associate a function~$\xx v_q \in
\xx V$ satisfying $\nabla \cdot \xx v _q = q$ and
\begin{equation} \label{continuous:inf-sup}
\beta \Vert \xx v_q \Vert_{\xx V} \le \Vert q \Vert_{L^2(\Omega)}.
\end{equation}

\subsection{PolyDG semi-discrete approximation of the transient Stokes problem}\label{sec:polydg}

First, we introduce the necessary notation and key analytical results
required for the definition and analysis of PolyDG semi-discrete approximation
of the transient Stokes problem.

We introduce a mesh $\mathcal{T}_h$ composed of polytopic elements~$K$ of arbitrary shape. We indicate with $h_{K}$ the diameter of the element~$K$.
We define an interface to be either the intersection of the $(d-1)$-dimensional facets of two neighboring elements or the intersection of the $(d-1)$-dimensional facets of an element with the boundary of $\Omega$.
When  $d = 2$, interfaces coincide with faces and consist of line segments; when $d=3$, we assume that each interface consists of a general planar polygon that we assume that can be further decomposed into a set of co-planar triangles, denoted as faces.

With this notation, we collect all the $(d-1)$-dimensional faces in the set $\mathcal{F}_{h}$, i.e., any face $F \in \mathcal{F}_{h}$ is always defined as a set of $(d-1)$-dimensional simplices (line segments or triangles);
cf. \cite{CangianiGeorgoulisHouston_2014,cangiani2017hp}.
We also decompose the faces $\mathcal{F}_h$ into $\mathcal{F}_{h} = \mathcal{F}_{h}^i \cup \mathcal{F}_{h}^b$, where $\mathcal{F}_{h}^i$ denotes the set of interior faces and $\mathcal{F}_{h}^b$ denotes the set of boundary faces.
To avoid technicalities, in the following we assume that $\rho$ and $\mu$ are piecewise constant over the mesh.\\

For given integers $\ell, m \ge 1$, we introduce the DG finite element spaces
\[
\left.
\begin{array}{l}
     \Vh = \{\xx v \in [L^2(\Omega)]^d:\,\xx v|_{K}\in
    [\mathcal{P}^\ell(K)]^d \; \forall K\in\mathcal{T}_h\},\\[0.2cm]
    \Qh = \{q \in L^2_0(\Omega):\, q|_{K}\in
    \mathcal{P}^m(K) \; \forall K\in\mathcal{T}_h\},\\[0.2cm]
\end{array}
\right.
\]
where $\mathcal{P}^k(K)$, $k \ge 1$, denotes the space of polynomials defined over the element $K \in \mathcal T_h$ of total degree at most~$k$.
In practice, the shape functions and the degrees of freedom are directly generated on the physical element $K \in \mathcal T_h$ with the ``bounding box'' technique;
see, e.g., \cite{CangianiGeorgoulisHouston_2014}.

On any interior face $F \in \mathcal{F}_h^i$ and for sufficiently regular scalar, vector-valued
and symmetric tensor-valued functions $q$, $\xx v$ and $\xx T$, respectively, we
define the \emph{average} and \emph{jump} operators as
\begin{align*}
    &\left\{{ \xx v }\right\} = \frac{1}{2}\left(\xx v^+ + \xx v^-\right),
    &&\jump{q} = q^+\xx n^+ + q^-\xx n^-,\\
    &\left\{{\xx T}\right\}=\frac{1}{2}\left(\xx T^+ +  \xx T^-\right),
    &&\jump{\xx v}=\xx v^+\odot\xx n^+ + \xx v^-\odot\xx n^-,
\end{align*}
where $q^\pm$, $\xx v^\pm$ and $\xx T^\pm$ denote the  traces of $q$, $\xx v$ and $\xx T$
on~$F$ taken within the interior of~$K^\pm$ and where $\xx v \odot \xx n=(\xx v
\xx n^{T}+\xx n \xx v^{T})/2$.
The jump~$\jump{\xx v}$ is a symmetric tensor-valued function.  On a boundary
face $F \in \mathcal{F}_h^b$, we set analogously
\begin{align*}
    &\left\{\xx v\right\} = \xx v,
    &&\jump{q}=q\xx n,\\
    &\left\{\xx T\right\} = \xx T,
    &&\jump{\xx v}=\xx v\odot\xx n.
\end{align*}
We also introduce the $L^2$-inner products over a domain~$Z \subset \R^d$, $d=1,2,3$, and a face $F \in \mathcal{F}_h$ with the shorthand notation $( \cdot , \cdot)_{Z}$ and $( \cdot , \cdot)_F$, respectively.\\

Given~$s >1/2$, associated with any mesh~$\mathcal T_h$, we introduce the broken Sobolev space
\[
H^s(\mathcal T_h) := \left\{ v \in L^2(\Omega) \mid v|_K \in H^s(K) \text{ for all } K\in \mathcal T_h     \right\}.
\]
The standard Dirichlet trace operator is well defined on the skeleton of the mesh for functions in~$H^s(\mathcal T_n)$.

Define the stabilization functions $\sigma_{v}\in L^{\infty}(\mathcal F_h)$ and $\sigma_{p}\in L^{\infty}(\mathcal F_h)$ as follows.
\begin{definition}
\label{def:sigmav_sigmap}
We define the functions $\sigma_v : \mathcal F_h \rightarrow \R$ and $\sigma_p :
\mathcal F_h^i \rightarrow \R$ as
\begin{align*}
& \sigma_v|_F =
\begin{cases}
\yy \gamma_v \max_{K^+,K^-} \left\{ \frac{\ell^2 \mu}{h_K}\right\} &  F \in \mathcal F_h^i,\\
\yy \gamma_v  \frac{\ell^2 \mu}{h_K} & F \in \mathcal F_h^b,
\end{cases}
&& \sigma_p|_F =
\yy \gamma_p \min_{K^+,K^-} \left\{ \frac{h_K}{m}\right\} \quad F \in \mathcal F_h^i,
\end{align*}
where $\gamma_v$ and $\gamma_p$ are two universal positive constants.
\end{definition}
Next, we introduce three bilinear forms that are instrumental for the construction of the DG method.
More precisely, we consider $ a_h: [H^1(\mathcal{T}_h)]^d\times [H^1(\mathcal{T}_h)]^d \rightarrow \R$, $b_h : H^{\frac{1}{2} + \varepsilon}(\Omega) \times [H^1(\mathcal{T}_h)]^d \rightarrow \R$,
and~$s_h : H^{\frac{1}{2} + \varepsilon}(\Omega) \times H^{\frac{1}{2} + \varepsilon}(\Omega) \rightarrow \R$, for all~$\varepsilon>0$, defined as
\begin{subequations}
\begin{align}
& a_h ( \xx{u}, \xx{v} ) = \int_\Omega \mu \nabla_h \xx{u} :
    \nabla_h \xx{v}  - \displaystyle \sum_{F \in \mathcal F_h} \int_F \mu \{ \nabla_h \xx u \} :   \jump{\xx{v}}\\
 & \quad\quad\quad\quad\quad  -  \sum_{F \in \mathcal F_h} \int_F \mu \jump{\xx{u}} : \{ \nabla_h \xx{v} \}  + \sum_{F \in \mathcal F_h} \int_F \sigma_v \jump{\xx{u}} :   \jump{\xx{v}}, \label{eq:form_a_h} \\
& b_h ( p , \xx{v} ) = - \displaystyle \int_\Omega p \nabla_h \cdot \xx{v} + \displaystyle \sum_{F \in \mathcal F_h} \int_F \{ p \xx I \} : \jump{\xx{v}},
\label{eq:form_b_h}\\
&  s_h \left( p , q \right) = \sum_{F \in \mathcal F_h^i} \int_F \sigma_p  \jump{ p } \cdot \jump{ q }, \label{eq:form_s}
\end{align}
\end{subequations}
where $\nabla_h$ is the piecewise broken gradient operator.

Given  $\xx f \in [L^2(\Omega)]^d$, the semi-discrete PolyDG approximation of~\eqref{eq:stokesWeak} reads as follows: for any $t \in \left( 0,T \right]$, find $( \xx{u}_h, p_h ) \in \Vh \times \Qh$ such that
\begin{align}
\label{eq:semiDiscForm}
\begin{split}
&\left(\rho \partial_t \xx{u}_h , \xx{v}_h \right)_{\Omega} + a_h \left( \xx{u}_h, \xx{v}_h \right) + b_h \left( p_h , \xx{v}_h \right) - b_h \left( q_h , \xx{u}_h \right) + s_h \left( p_h , q_h \right)  = (\xx f,\xx v_h)_\Omega
\end{split}
\end{align}
for all $( \xx{v}_h, q_h ) \in \Vh \times \Qh$.

\section{Well-posedness of the stationary Stokes problem}\label{sec:wellposed}

In this section, we prove the well-posedness of problem \eqref{eq:semiDiscForm}
in the stationary case making use of the Banach-Ne\v{c}as-Babu\v{s}ka theorem.

To this aim, we first introduce
\begin{equation} \label{eq:abs_bil_form}
\mathcal{B}_h((\xx u, p),(\xx v, q)) = a_h(\xx u, \xx v) + b_h(p, \xx v) -
b_h(q, \xx u) + s_h (p,q),
\qquad F ((\xx v, q)) = (\xx f , \xx v)_{\Omega},
\end{equation}
and re-write the stationary discrete Stokes problem as follows:
find $(\xx u_h , p_h) \in \Vh \times \Qh$ such that
\begin{equation}
\label{eq:abs_stokes}
\mathcal{B}_h((\xx u_h, p_h),(\xx v_h, q_h)) = F ((\xx v_h, q_h)) \qquad \forall (\xx v_h,
q_h) \in \Vh \times \Qh.
\end{equation}
On the product space $\Vh \times \Qh$, we  define the norm
\begin{equation}
\label{eq:Enorm}
\begin{aligned}
&\normE{ (\xx v_h, q_h) }^2=\normVh{ \xx v_h }^2 + \normQh{ q_h }^2
\quad\quad \forall (\xx v_h, q_h) \in \Vh \times \Qh,
\end{aligned}
\end{equation}
where
\begin{equation}
\label{eq:DGnorm}
\begin{aligned}
\normVh{ \xx v_h }^2&=
\sum_{K\in \mathcal{T}_k}\norm{ \mu^{1/2} \nabla_h \xx v_h }^2_{L^2(K)} + \norm{
\sigma_v^{1/2} \jump{ \xx v_h } }^2_{L^2(\mathcal F_h)} \quad \quad \forall {\xx v_h} \in \Vh ,\\
\normQh{ q_h }^2 &=\normQ{q_h}^2 + \seminormJ{q_h}^2, \quad \quad
\seminormJ{q_h}^2=s_h (q_h , q_h) \quad \quad \forall  q_h \in \Qh.
\end{aligned}
\end{equation}
Before presenting the theoretical analysis, we introduce some mesh assumptions and technical results that will be needed in the forthcoming analysis.

\subsection{Mesh assumptions and preliminary results}
Following \cite{CangianiGeorgoulisHouston_2014,cangiani17Spacetime,antoniettiF2020}, we introduce the notion of a family of \emph{polytopic-regular} meshes $\mathcal{T}_h$.
To this end, we write $\tau_{K_{F}} $ to denote a $d$-dimensional simplex contained in $K\in \mathcal{T}_h$, which shares a specific face~$F\subset \partial K$, $F\in \mathcal{F}_h$.
%%%%%
\begin{definition}\label{def_mesh_polytopic_regular}
A family of polytopic meshes $\left\{\mathcal{T}_h\right\}_h$ is said to be \emph{polytopic-regular} if, for any $h$ and $K \in \mathcal{T}_h$,
there exists a set of non-overlapping (not necessarily shape-regular) $d$-dimensional simplices $\{\tau_{K_{F}} \}_{F\subset \partial K}$ contained in~$K$, such that, for all faces $F \subset \partial K$,
\[
h_K \lesssim \frac{|\tau_{K_{F}}|}{|F|}.
\]
The hidden constant is independent of the discretization parameters, the number of faces of the element, and the face measure.
\end{definition}
This definition is very general as it does not require any restriction on either the number of faces per element or their relative measure.
In particular, it allows the size of a face $F \subset \partial K$ to be arbitrarily small compared to the diameter of the element~$h_K$ it belongs to, provided that the height of the corresponding simplex $\tau_{K_{F}} $ is comparable to $h_K$;
cf. \cite{cangiani2017hp} for more details.

In order to state suitable approximation results, cf. Lemmata~\ref{lm:localapproxTPi} and~\ref{lm:approxDG} below and~\cite{CangianiGeorgoulisHouston_2014}, we introduce a \emph{shape-regular covering} $\mathcal{T}_h^{\#} = \{ T_K\}$ of $\mathcal{T}_h$ defined as a set of
shape-regular $d$-dimensional simplices $T_K$, such that, for each $K \in \mathcal{T}_h$, there exists a $T_K \in \mathcal{T}_h^{\#}$ such that $K\subsetneq T_K$.\\

We introduce the following assumption on the mesh
$\mathcal{T}_h$; cf. \cite{CangianiGeorgoulisHouston_2014,cangiani2017hp}.
\begin{assumption}\label{ass:mesh_all}
Given $\left\{\mathcal{T}_h\right\}_h$, $h>0$, we assume that the following properties are uniformly satisfied:
\begin{enumerate}[label=\textbf{A.\arabic*}]
\item \label{ass:arb_num_faces} $\mathcal{T}_h$ is uniformly \emph{polytopic-regular} in the sense of Definition \ref{def_mesh_polytopic_regular};
\item \label{ass:shape_reg} we assume that there exists a \emph{shape-regular
covering} $\mathcal T_h^{\#}$ of $\mathcal T_h$  such that, for each pair $K \in
\mathcal T_h$, $\mathcal K \in \mathcal T_h^{\#}$ with $K \subset \mathcal K$,
the following properties are fulfilled:
\emph{i)} $h_{\mathcal K} \lesssim  h_K$ and
\emph{ii)} $\max_{K \in \mathcal T_h} \textrm{\emph{card}} \lbrace K' \in \mathcal T_h : K' \cap \mathcal K \ne \emptyset, \mathcal K
\in \mathcal T_h^{\#}, K \subset
\mathcal K \rbrace \lesssim 1$;
\item  \label{ass:local_unif} for any pair of elements $K, K' \in \mathcal{T}_h$
sharing a face $F \in \mathcal{F}_h$, we have: $h_K \lesssim h_{K'}$ and $h_{K'}
\lesssim h_K$, where the hidden constants are independent of the discretization
parameters as well as  the number of faces of the two elements.
\end{enumerate}
\end{assumption}
The local bounded variation hypothesis \ref{ass:local_unif}  has been
introduced so as to avoid technicalities.\\

The following trace-inverse inequality is valid; see, e.g.,
\cite[Lemma~11]{cangiani2017hp}.
% Lemma 32, pag 51 (Discrete local inverse inequality)
\begin{lemma}[Polynomial trace inverse inequality] \label{lm:discInvIn1}
Let Assumption~\ref{ass:arb_num_faces} be valid. \newline For each $K \in \mathcal T_h$,  the
following trace-inverse inequality is valid:
\begin{equation*}
\begin{aligned}
&\norm{v}^2_{L^2(\partial K)} \lesssim \frac{r^2}{h_K} \norm{v}^2_{L^2(K)} && \forall v \in \mathcal P^{r} \left( K \right), &&  r\geq 1,
\end{aligned}
\end{equation*}
where the hidden constant is independent of $r$, $h_K$, and the number of faces of the element.
\end{lemma}

%%%%%%%%%%%%%%%%%%%%%
Let $\mathcal{E}:H^s(\Omega) \rightarrow H^s(\R^d)$, $s \geq 0$, be the Stein extension operator for Sobolev spaces on Lipschitz domains introduced in~\cite[Chapter~3]{Stein70}.
The operator~$\mathcal E$ satisfies the following property: given a domain~$\Omega$ with Lipschitz boundary, for all~$q \in H^s(\Omega)$,
\begin{equation} \label{continuity:Stein}
\mathcal{E}(q)|_\Omega=q, \quad \quad \quad \quad \norm{\mathcal{E}q}_{H^s(\R^d)} \lesssim \norm{q}_{H^s(\Omega)}.
\end{equation}
For vector-valued functions, the Stein extension operator is defined component-wise. We recall the following approximation result;
see, e.g., \cite{CangianiGeorgoulisHouston_2014,cangiani17Spacetime,cangiani2017hp} for a detailed proof, which generalizes the standard arguments for standard geometries~\cite{babuskaHP87,babuvska1987optimal}.

\begin{lemma} [Best polynomial approximation in Sobolev norms] \label{lm:localapproxTPi}
~\newline
Let Assumption~\ref{ass:shape_reg} be valid. Given the Stein extension operator $\mathcal E$ in~\eqref{continuity:Stein},
let~$v \in L^2\left( \Omega \right)$ be such that $\left(\mathcal E v\right)|_{\mathcal K} \in H^r(\mathcal K)$, for some $r \ge 0$.
Then, there exists a sequence of polynomial approximations $\Pi^{\ell}_K v \in \mathcal{P}_{\ell}(K)$ of~$v$, $K \in \mathcal T_h$ and~$\ell \in \mathbb N$ of $v$ satisfying
\begin{align*}
&\norm{v - \Pi^{\ell} _K v}_{H^q(K)} \lesssim
\frac{h_K^{\min{\lbrace \ell+1 , r \rbrace}-q}}{\ell^{r-q}} \norm{\mathcal E v}_{H^r(\mathcal K)},
&& 0 \le q \le r,
\end{align*}
where $\mathcal K \in \mathcal T_h^{\#}$ is the $d$-simplex of $\mathcal
T_h^{\#}$ such that~$K \subset \mathcal K$.  Moreover,  if $v \in H^1 ( \Omega
)$ is such that $\left(\mathcal E v\right)|_{\mathcal K} \in H^r(\mathcal K)$,
for some $r \geq 1$, then we have
\begin{align*}
& \norm{v- \Pi^{\ell} _K v}_{L^2(\partial K)} \lesssim
\frac{h_K^{\min\{ \ell+1,r\}-1/2}}{\ell^{r-1/2}}\norm{\mathcal{E}v}_{H^{r}(\mathcal
K)}, && r \geq 1. %\label{DG:stima interp su bordo}
\end{align*}
The hidden constants are independent of the discretization parameters as well as  the number of
faces of the element.
 \end{lemma}
Based on employing the above result, we define the global polynomial approximation operator $\Pi^{\ell}v $ as
\[
(\Pi^{\ell}v)|_{K}=\Pi^{\ell}_K (v|_{K}) \quad \quad \forall K \in \mathcal{T}_h.
\]
For vector-valued functions, the operators $\Pi^{\ell}_K$ and $\Pi^{\ell}$ are defined component-wise and are still denoted by $\Pi^{\ell}_K$ and $\Pi^{\ell}$, respectively.

We have the following approximation bound in the energy norm \eqref{eq:DGnorm}.
\begin{lemma} [Best polynomial approximation in the DG norm] \label{lm:approxDG}
~\newline
Let Assumption~\ref{ass:mesh_all} be valid.  Let $\xx v \in [L^2\left( \Omega \right)]^d$ be such that, for some $r \ge 1$, $\left(\mathcal E v\right)|_{\mathcal K} \in [H^r(\mathcal K)]^d$ for all $\mathcal K \in \mathcal T_h^\#$, $r\ge 1$.
Then, we have
\begin{align*}
\normVh{\xx v - \Pi^{\ell} \xx v}^2 \lesssim \sum_{K \in \mathcal{T}_h}  \frac{h_K^{2(\min{\lbrace \ell+1 , r \rbrace}-1)}}{\ell^{2(r-1)-1}} \norm{\mathcal E \xx v}_{H^r(\mathcal K)}^2.
\end{align*}
The hidden constants  are independent of the discretization parameters as well as the number of faces of each element.
\end{lemma}
The suboptimality in terms of the polynomial degree in the estimates of Lemma~\ref{lm:approxDG} is due to the presence of the stabilization term and the suboptimality of the polynomial trace inverse estimate of Lemma~\ref{lm:discInvIn1}.

Finally, we recall the following continuity and coercivity bounds for the bilinear form $a_h(\cdot,\cdot)$.
The proof is based upon employing the trace-inverse estimate in Lemma~\ref{lm:discInvIn1} and standard arguments for DG methods; see, e.g., \cite{cangiani2017hp}.
\begin{lemma}[Coercivity and continuity of $a_h(\cdot,\cdot)$]
Let Assumption~\ref{ass:mesh_all} be valid. \newline Then, we have
\[
a_h (\xx v_h, \xx v_h) \gtrsim \normVh{ \xx v_h }^2 \qquad \forall \xx v_h \in \Vh,
\]
and
\[
\abs{ a_h(\xx u_h, \xx v_h)} \lesssim \normVh{ \xx u_h } \normVh{ \xx v_h } \qquad \forall \xx u_h , \xx v_h \in \Vh.
\]
The coercivity bounds are achieved provided that the penalty parameter $\gamma_{v}$ in Definition~\ref{def:sigmav_sigmap} of the penalty function $\sigma_v$ is chosen sufficiently large.
The hidden constants are independent of the discretization parameters, the number of faces per element, and the relative size of a face compared to the diameter of the element it belongs to.
\end{lemma}

%%%%
\subsection{Generalized \textit{inf-sup} condition}
%%%%
In this section, we prove a generalized \textit{inf-sup} condition for
the discrete bilinear form~$b_h(\cdot,\cdot)$ defined in~\eqref{eq:form_b_h}.
First, we need some preliminary results.
\begin{lemma}[Boundedness of $\Pi^{\ell}$ in the energy norm \eqref{eq:DGnorm}]
\label{lm:globalBoundTPil}
Let Assumption~\ref{ass:mesh_all} \newline be valid. Then, we have
\begin{equation*}
\begin{aligned}
& \normVh{ \Pi^{\ell} \xx v } \lesssim \ell^{1/2} \norm{\xx v }_{H^1(\Omega)}
&& \forall \xx v \in \xx V.
\end{aligned}
\end{equation*}
The hidden constants are independent of the discretization parameters as well as  the number of faces of the element.
\end{lemma}
\begin{proof}
Given $\xx v \in \xx V$, using the definition~\eqref{eq:DGnorm} of the
energy norm and the fact that~$\jump{\xx v} = \xx 0$ on~$F \in \mathcal F_h$, we immediately have
\[
\normVh{ \Pi^{\ell} \xx v }^2   \leq \normVh{ \Pi^{\ell} \xx v - \xx v  }^2 + \normVh{ \xx v  }^2 \lesssim  \normVh{ \Pi^{\ell} \xx v - \xx v  }^2 + \norm{\xx v}_{H^1(\Omega)}^2.
\]
Using Lemma~\ref{lm:approxDG} with~$r=1$, Assumption~\ref{ass:shape_reg}, and the continuity of the Stein operator~$\mathcal E$ in~\eqref{continuity:Stein}, we get
\[
 \normVh{ \Pi^{\ell} \xx v }^2   \lesssim \sum_{K \in \mathcal{T}_h}  \ell \norm{\mathcal E \xx v}_{H^1(\mathcal K)}^2 + \norm{\xx v}_{H^1(\Omega)}^2 \lesssim (1+\ell) \norm{\xx v}_{H^1(\Omega)}^2.
\]
\end{proof}

Next, we introduce the~$L^2$ projector onto the space $\Vh$:
\[
 \Pi_0^{\ell}: [L^2(\Omega)]^d \longrightarrow \Vh, \quad\quad \left( \xx w_h , \xx v - \Pi_0^{\ell} \xx v \right)_{L^2(\Omega)} = 0 \quad \forall \xx w_h \in \Vh,
\]
and state  the following result, which is based on a further assumption and
a technical result; see Assumption~\ref{ass:star} and Lemma~\ref{lm:H1L2inverse}
below, respectively.
\begin{assumption}\label{ass:star}
Given~$\left\{\mathcal{T}_h\right\}_h$, $h>0$, each element~$K \in \mathcal T_h$
admits a decomposition into shape-regular simplices having size comparable to
that of~$K$.
\end{assumption}

The following inverse estimate on shape-regular polygons can be found, e.g., in~\cite[Lemma~14]{cangiani2017hp}.
It generalizes a similar result for standard geometries; see, e.g., \cite[Theorem~4.76]{SchwabpandhpFEM}.
\begin{lemma} [$H^1-L^2$ polynomial inverse estimate] \label{lm:H1L2inverse}
Let Assumption~\ref{ass:star} be valid. \newline For each $K \in \mathcal T_h$,  the following polynomial inverse inequality is valid:
\[
\begin{aligned}
& \norm{\nabla v}^2_{L^2(K)} \lesssim \frac{\ell^4}{h_K^2} \norm{v}^2_{L^2(K)} && \forall v \in \mathcal P^{\ell} \left( K \right), &&  \ell\geq 1.
\end{aligned}
\]
The hidden constant is independent of~$\ell$, $h_K$, and the number of faces of the element.
\end{lemma}

Based on  employing the above result, we prove the following bound.
\begin{lemma} [Stability properties of orthogonal projections]
\label{lm:globalBoundP0}
Let Assumptions~\ref{ass:mesh_all} and~\ref{ass:star} be valid. Then, we have
\begin{equation*}
\begin{aligned}
& \normVh{\Pi_0^{\ell} (\xx v - \Pi^{\ell} \xx v)} \lesssim \ell \normV{\xx v } && \forall \xx v \in \xx V.
\end{aligned}
\end{equation*}
The hidden constant is independent of the discretization parameters as well as the number of faces of the element.
\end{lemma}
%%%
\begin{proof}
From the definition of the energy norm~\eqref{eq:DGnorm}, the inverse estimate
in Lemma~\ref{lm:H1L2inverse}, the stability of the projector~$\Pi_0^\ell$ in
the~$L^2$ norm, the polynomial approximation properties of
Lemma~\ref{lm:localapproxTPi}, and the continuity of the Stein
operator~$\mathcal E$ in~\eqref{continuity:Stein}, we have
\begin{align*}
 \sum_{K \in \mathcal{T}_h}\seminorm{\Pi_0^{\ell}  (\xx v - \Pi^{\ell} \xx
 v)}_{H^1(K)}^2
 &\lesssim \sum_{K \in \mathcal{T}_h} \frac{\ell^4}{h_K^2} \norm{\Pi_0^\ell (\xx
 v - \Pi ^\ell \xx v)}^2_{L^2(K)} \\
 & \le \sum_{K \in \mathcal{T}_h} \frac{\ell^4}{h_K^2} \norm {\xx v - \Pi^\ell
 \xx v}^2_{L^2(K)} \lesssim \ell^2 \norm{\xx v}_{\xx V}^2 .
\end{align*}
%\footnote{ Mi sembra che debba essere $\lesssim \seminorm{\Pi_0^{\ell} \xx v}_{H^1(\Omega)}^2 + \normVh{{\color{magenta}\Pi^{\ell}}\xx v}^2\lesssim \ldots$}
Next, using the definition of~$\sigma_v$ in Definition~\ref{def:sigmav_sigmap},
the discrete trace-inverse inequality in Lemma~\ref{lm:discInvIn1}, the
continuity of the $L^2$-projector, the approximation results in
Lemma~\ref{lm:localapproxTPi}, and the continuity of the Stein
operator~$\mathcal E$ in~\eqref{continuity:Stein}, and
Assumption~\ref{ass:shape_reg}, we have
\begin{align*}
& \norm{ \sigma_v^{1/2}\jump{\Pi_0^{\ell} (\xx v - \Pi^{\ell} \xx
v)}}^2_{L^2(\mathcal F_h)} \lesssim \sum_{K \in \mathcal{T}_{h}}
\frac{\ell^2}{h_K}\norm{\Pi_0^{\ell} (\xx v - \Pi^{\ell} \xx v)}^2_{L^2(\partial
K)}\\
& \quad  \lesssim \sum_{K \in \mathcal{T}_{h}}
\frac{\ell^4}{h_K^2}\norm{\Pi_0^{\ell} (\xx v - \Pi^{\ell} \xx v)}^2_{L^2(K)}
\lesssim \sum_{K \in \mathcal{T}_{h}} \frac{\ell^4}{h_K^2}\norm{\xx v -
\Pi^{\ell} \xx v}^2_{L^2(K)}\\
& \quad \lesssim \sum_{K \in \mathcal{T}_{h}}
\frac{\ell^4}{h_K^2}\frac{h_K^2}{\ell^2}\norm{\mathcal E \xx v}_{H^1(\mathcal
K)}^2 \lesssim\ell^2\norm{\xx v}_{\xx V}^2.
\end{align*}
The assertion follows summing up the two bounds.
\end{proof}
\begin{remark}
Assumption~\ref{ass:star} is required in the proof of the polynomial inverse estimate of Lemma~\ref{lm:H1L2inverse}.
On the other hand, the suboptimality in terms of the polynomial degree in the stability properties detailed in Lemma~\ref{lm:globalBoundP0} is now due to both the inverse estimates of Lemmata~\ref{lm:discInvIn1} and~\ref{lm:H1L2inverse}.
This propagates further in the proof of the discrete \textit{inf-sup} condition, see Proposition~\ref{prop:gen_inf_sup} below, and consequently to the abstract and convergence analysis detailed in Section~\ref{section:abstract-convergence} below.
\end{remark}

\begin{remark}
Following the recent approach of~\cite{CaDoGe2019}, it is possible to prove the inverse estimates in Lemmata~\ref{lm:discInvIn1} and~\ref{lm:H1L2inverse}
using assumptions milder than Assumptions~\ref{ass:mesh_all} and~\ref{ass:star}.
Notably, the theory therein presented covers very general geometries, including $\mathcal C^1$-curved faces and possibly the presence of arbitrary number of faces.
\end{remark}

Next, we show that a generalized \textit{inf-sup} condition is valid, provided that the polynomial degrees~$\ell$ and~$m$ of the discrete velocity and pressure spaces satisfy $m-\ell \le 1$.
This condition guarantees in fact that~$\nabla \Qh \subseteq \Vh$.
\begin{proposition}[Generalized \textit{inf-sup} condition for~$b_h(\cdot,\cdot)$]
\label{prop:gen_inf_sup}
~\newline
Let Assumptions~\ref{ass:mesh_all} and~\ref{ass:star} be valid and assume that the polynomial degrees~$\ell$ and~$m$ of the discrete velocity and pressure spaces satisfy $m-\ell \le 1$.
Then, the following bound is valid:
\begin{equation*}
\begin{aligned}
& \sup_{ \xx 0\neq \xx v_h \in \Vh} \frac{b_h(q_h , \xx v_h) }{\normVh{\xx v_h}} + \seminormJ{q_h} \ge
\beta_{h} \normQ{q_h}
&& \forall q_h \in \Qh,
\end{aligned}
\end{equation*}
where the discrete \textit{inf-sup} constant behaves as
\begin{equation} \label{inf-sup:constant}
\beta_{h} = O\left(\frac{\beta}{\max\left\{ \ell^{1/2}(1+\ell^{1/2}), m^{1/2}+1\right\}}\right).
\end{equation}
\end{proposition}
\begin{proof}
Upon employing element-wise integration by parts, the bilinear form $b_h (\cdot, \cdot)$ defined as
$$
b_h (q_h , \xx{v} _h) = - \displaystyle \int_\Omega q_h \nabla_h \cdot \xx{v}_h + \displaystyle \sum_{F \in \mathcal F_h} \int_F \{ q_h \xx I \} : \jump{\xx{v}_h} \qquad \forall q_h \in  \Qh, \xx{v} _h \in \Vh,
$$
can be equivalently rewritten as
$$
b_h (q_h , \xx{v} _h) =  \displaystyle \int_\Omega \nabla_h q_h  \cdot \xx{v}_h
- \displaystyle \sum_{F \in \mathcal F_h^I} \int_F \jump{q_h} \cdot \{\xx{v}_h\}
\qquad \forall q_h \in \Qh, \xx{v} _h \in \Vh.
$$
Recall the continuous \textit{inf-sup} condition~\eqref{continuous:inf-sup}: there exists~$\beta>0$ such that, to each~$q_h \in \Qh \subset L^2_0(\Omega)$, we associate~$\xx v_{q_h} \in \xx V$ with
\begin{equation} \label{continuous:inf-sup2}
\nabla \cdot \xx v_{q_h} = q_h, \quad \quad \beta \norm{ \xx v_{q_h}}_{\xx V} \le \norm{q_h}_{L^2\left( \Omega \right)}.
\end{equation}
Then, applying element-wise integration by parts, using that $\jump{\xx
v_{q_h}} =\xx 0$ for any~$F \in \mathcal F_h$,  and observing that~$\nabla_h q_h
\in \Vh$ if~$\ell \ge m-1$, we obtain
\begin{equation} \label{bound:ABC}
\begin{split}
& \norm{ q_h }_{L^2\left( \Omega \right)}^2 = \int_{\Omega} q_h \nabla \cdot \xx  v_{q_h} = - \int_{\Omega} \nabla_h q_h \cdot \xx v_{q_h} + \displaystyle  \sum_{F \in \mathcal F_h} \int_F \jump{q_h} : \{ \xx v_{q_h} \}\\
&\quad= - \int_{\Omega} \nabla_h q_h \cdot \Pi^{\ell} \xx v_{q_h} +   \int_{\Omega} \nabla_h q_h \cdot (\Pi^{\ell} \xx v_{q_h} - \xx v_{q_h})  + \displaystyle  \sum_{F \in \mathcal F_h} \int_F \jump{q_h} \cdot \{ \xx v_{q_h} \}\\
&\quad= \int_{\Omega} q_h \nabla \cdot \Pi^{\ell} \xx v_{q_h} + \int_{\Omega} \nabla_h q_h \cdot (\Pi^{\ell} \xx v_{q_h} - \xx v_{q_h})\\
&\quad \quad \quad - \sum_{F \in \mathcal F_h} \int_F \{q_h \xx I\} : \jump{ \Pi^{\ell} \xx v_{q_h} } + \displaystyle \sum_{F \in \mathcal F_h} \int_F \jump{q_h} \cdot \{ \xx v_{q_h} - \Pi^{\ell} \xx v_{q_h} \}\\
&\quad=  \underbrace{-b_h\left( q_h , \Pi^{\ell} \xx v_{q_h} \right)}_{\circled{A}}  \underbrace{+   \int_{\Omega} \nabla_h q_h \cdot (\Pi^{\ell} \xx v_{q_h} - \xx v_{q_h})}_{\circled{B}} +  \underbrace{\sum_{F\in \mathcal F_h} \int_F \jump{q_h} \cdot \{ \xx  v_{q_h} - \Pi^{\ell} \xx v_{q_h}\}}_{\circled{C}}.
\end{split}
\end{equation}
We bound the three terms on the right-hand side of~\eqref{bound:ABC} separately.
As for the term $\circled{A}$, thanks to the boundedness of $\Pi^{\ell}$ in the energy norm, see Lemma~\ref{lm:globalBoundTPil}, and the continuous \textit{inf-sup} condition~\eqref{continuous:inf-sup2}, we get
\begin{equation} \label{bound:A}
\begin{split}
\circled{A} = - b_h
& \left( q_h , \Pi^{\ell}\xx v_{q_h} \right) \le \frac{\abs{ b_h\left( q_h ,
\Pi^{\ell} \xx v_{q_h} \right) }}{\normVh{ \Pi^{\ell} \xx  v_{q_h}}} \normVh{
\Pi^{\ell} \xx v_{q_h} }\\
&    \lesssim \ell^{1/2} \frac{\abs{ b_h\left( q_h , \Pi^{\ell} \xx v_{q_h} \right) }}{ \normVh{ \Pi^{\ell}\xx  v_{q_h} } } \norm{\xx v_{q_h} }_{H^1\left( \Omega \right)}
\le \frac{\ell^{1/2}}{\beta} \frac{\abs{ b_h\left( q_h , \Pi^{\ell}\xx v_{q_h}
\right) }}{ \normVh{ \Pi^{\ell}\xx v_{q_h} }} \norm{  q_h }_{L^2\left( \Omega
\right)} \\
&    \le \frac{\ell^{1/2}}{\beta} \norm{ q_h }_{L^2\left( \Omega \right)} \sup_{\xx v_h \in \Vh \setminus {\{\xx 0\}}} \frac{ b_h\left( q_h , \xx v_h \right) }{\normVh{\xx v_h }}.
\end{split}
\end{equation}
As for the term $\circled{C}$, using the Cauchy-Schwarz inequality, the definition of the penalty function $\sigma_p$ in Definition~\ref{def:sigmav_sigmap}, Assumption \ref{ass:local_unif},
the second bound in Lemma~\ref{lm:localapproxTPi} with $r=1$, and~$m-\ell\le 1$, we obtain
\[
\begin{split}
\circled{C}
&\lesssim \seminormJ{q_h}\left( \sum_{K \in \mathcal T_h} \frac{m}{h_K} \norm{ \xx v_{q_h} - \Pi^{\ell} \xx v_{q_h} }_{L^2(\partial K)}^2 \right)^{1/2}\\
&\lesssim \seminormJ{q_h} \left( \sum_{K \in \mathcal T_h} \frac{m}{h_K} \frac{h_K}{\ell} \norm{ \mathcal E \xx v_{q_h} }_{H^1(\mathcal K)}^2 \right)^{1/2} \lesssim \seminormJ{q_h} \left( \sum_{K \in \mathcal T_h} \norm{ \mathcal E \xx v_{q_h} }_{H^1(\mathcal K)}^2 \right)^{1/2}.
\end{split}
\]
Finally, using the continuity of the Stein extension operator $\mathcal E$ in~\eqref{continuity:Stein}, Assumption~\ref{ass:local_unif}, and the continuous \textit{inf-sup} condition~\eqref{continuous:inf-sup2}, we get
\begin{equation} \label{bound:C}
\circled{C} \lesssim  \seminormJ{q_h}\norm{ \xx v_{q_h} }_{H^1(\Omega)} \le \frac{1}{\beta} \seminormJ{q_h}\norm{ q_h}_{L^2(\Omega)}.
\end{equation}
As for the term $\circled{B}$, using the definition of $L^2$ projector, the fact that~$\nabla_h q_h \in \Vh$ ($m-\ell \leq1$), and an integration by parts, we write
\begin{equation} \label{initial-bound:B}
\begin{split}
\circled{B}
&=\int_{\Omega} \nabla_h q_h \cdot (\Pi^{\ell} \xx v_{q_h} - \Pi_0^{\ell}\xx v_{q_h} +\Pi_0^{\ell}\xx v_{q_h} - \xx v_{q_h}) = \int_{\Omega} \nabla_h q_h \cdot (\Pi^{\ell} \xx v_{q_h} - \Pi_0^{\ell}\xx v_{q_h})\\
&= \int_{\Omega} \nabla_h q_h \cdot (\Pi_0^{\ell} \Pi^{\ell} \xx v_{q_h} - \Pi_0^{\ell}\xx v_{q_h}) = \int_{\Omega} \nabla_h q_h \cdot \Pi_0^{\ell}  (\Pi^{\ell} \xx v_{q_h} - \xx v_{q_h})\\
&= \underbrace{b_h(q_h,\Pi_0^{\ell}  (\Pi^{\ell} \xx v_{q_h} - \xx v_{q_h}))}_{\circled{I}}+ \underbrace{\sum_{F \in \mathcal F_h^i} \int_F \jump{q_h} \cdot \{\Pi_0^{\ell}  (\Pi^{\ell} \xx v_{q_h} - \xx v_{q_h}))\}}_{\circled{II}}.
\end{split}
\end{equation}
We bound the two terms on the right-hand side separately. As for the term~$\circled{I}$, we proceed as above,
namely we use the continuity of the Stein extension operator $\mathcal E$ in~\eqref{continuity:Stein}, Assumption \ref{ass:shape_reg}, and the continuous \textit{inf-sup} condition~\eqref{continuous:inf-sup2}:
\begin{equation} \label{bound:I}
\begin{split}
\circled{I}
&\le \frac{\abs{ b_h\left( q_h , \Pi_0^{\ell}  (\Pi^{\ell} \xx v_{q_h} - \xx v_{q_h}) \right) }}{ \normVh{ \Pi_0^{\ell}  (\Pi^{\ell} \xx v_{q_h} - \xx v_{q_h})} } \normVh{ \Pi_0^{\ell}  (\Pi^{\ell} \xx v_{q_h} - \xx v_{q_h}) }\\
& \le \left(\sup_{\xx v_h \in \Vh \setminus {\{\xx 0\}}} \frac{ b_h\left( q_h , \xx v_h \right) }{\normVh{\xx v_h }} \right)\normVh{ \Pi_0^{\ell} (\Pi^{\ell} \xx v_{q_h} - \xx v_{q_h}) }\\
&\lesssim \ell \left(\sup_{\xx v_h \in \Vh \setminus {\{\xx 0\}}} \frac{ b_h\left( q_h , \xx v_h \right) }{\normVh{\xx v_h }} \right)\normV{\xx v_{q_h}} \le \frac{\ell}{\beta} \left(\sup_{\xx v_h \in \Vh \setminus {\{\xx 0\}}} \frac{ b_h\left( q_h , \xx v_h \right) }{\normVh{\xx v_h }} \right)\norm{q_h}_{L^2(\Omega)}.
\end{split}
\end{equation}
As for the term $\circled{II}$, we make use of the trace-inverse inequality of Lemma~\ref{lm:discInvIn1}, the stability in $L^2$ of the projector~$\Pi_0^{\ell}$, the interpolation bounds of Lemma~\ref{lm:localapproxTPi},
the continuity of the Stein extension operator $\mathcal E$ in~\eqref{continuity:Stein} together with Assumption \ref{ass:shape_reg}, and the continuous \textit{inf-sup} condition~\eqref{continuous:inf-sup2} to obtain
\begin{equation} \label{bound:II}
\begin{split}
\circled{II}
&\lesssim \seminormJ{q_h}\left( \sum_{K \in \mathcal T_h} \frac{m}{h_K} \norm{\Pi_0^{\ell}  (\Pi^{\ell} \xx v_{q_h} - \xx v_{q_h}) }_{L^2(\partial K)}^2 \right)^{1/2}\\
&\lesssim \seminormJ{q_h}\left( \sum_{K \in \mathcal T_h} \frac{m}{h_K} \frac{\ell^2}{h_K} \norm{ \Pi_0^{\ell}  (\Pi^{\ell} \xx v_{q_h} - \xx v_{q_h})}_{L^2( K)}^2 \right)^{1/2}\\
&\lesssim \seminormJ{q_h}\left( \sum_{K \in \mathcal T_h} \frac{m}{h_K} \frac{\ell^2}{h_K}\frac{h_K^2}{\ell^2} \norm{ \mathcal E \xx v_{q_h}}_{H^1(\mathcal K)}^2 \right)^{1/2}\\
&\lesssim m^{1/2} \seminormJ{q_h} \norm{\xx v_{q_h} }_{H^1(\Omega)} \le \frac{m^{1/2}}{\beta}  \seminormJ{q_h} \norm{q_h}_{L^2(\Omega)}.
\end{split}
\end{equation}
Inserting the two bounds~\eqref{bound:I} and~\eqref{bound:II} into~\eqref{initial-bound:B}, we obtain
\begin{equation} \label{bound:B}
\circled{B} \lesssim \frac{\ell}{\beta} \left(\sup_{\xx v_h \in \Vh \setminus {\{\xx 0\}}} \frac{b_h\left( q_h , \xx v_h \right) }{\normVh{\xx v_h }} \right)\norm{q_h}_{L^2(\Omega)}+ \frac{m^{1/2}}{\beta}  \seminormJ{q_h} \norm{q_h}_{L^2(\Omega)}.
\end{equation}
Collecting~\eqref{bound:A}, \eqref{bound:C}, and~\eqref{bound:B} into~\eqref{bound:ABC}, we arrive at
\begin{align*}
\norm{ q_h }_{L^2\left( \Omega \right)}
& \lesssim \left(\frac{\ell^{1/2}+\ell}{\beta}\right) \sup_{\xx v_h \in \Vh \setminus {\{\xx 0\}}} \frac{ b_h\left( q_h , \xx v_h \right) }{\normVh{\xx v_h }} + \left(\frac{m^{1/2}+1}{\beta}\right) \seminormJ{q_h}\\
& \lesssim \frac{1}{\beta}\max\left\{ \ell^{1/2}(1+\ell^{1/2}), m^{1/2}+1\right\} \left( \sup_{\xx v_h \in \Vh \setminus {\{\xx 0\}}} \frac{ b_h\left( q_h , \xx v_h \right) }{\normVh{\xx v_h }} + \seminormJ{q_h}\right).
\end{align*}
The assertion follows with the discrete generalized \textit{inf-sup} constant having the behaviour in~\eqref{inf-sup:constant}.
\end{proof}
\begin{remark}
\label{rem:betah}
The constant of the generalized \textit{inf-sup} condition stated in Proposition~\ref{prop:gen_inf_sup} is uniform with respect to the mesh size but depends on the polynomial approximation degrees $\ell$ and $m$; see~\eqref{inf-sup:constant}.
This implies that $\beta_{h} \searrow 0$ as $\ell,m \nearrow + \infty$.

In Section \ref{sec:computations_inf_sup}, we will present some computations to assess numerically the sharpness of the \textit{inf-sup} constant~$\beta_h$, for different mesh configurations and polynomial orders.
We will find out that Assumption~\ref{ass:star} does not seem necessary in the proof of Proposition~\ref{prop:gen_inf_sup}. The analysis with milder assumptions is under investigation.
\end{remark}

\subsection{Well-posedness of the discret Stokes problem via the Banach-Ne\v{c}as-Babu\v{s}ka theorem.}

To prove that the discrete problem~\eqref{eq:abs_stokes} is well-posed,  we
first recall the following abstract result; see, e.g, \cite{ern2013theory}.
\begin{theorem}[Banach-Ne\v{c}as-Babu\v{s}ka] \label{th:BNB}
Let $\mathcal W$ be a Banach space and~$\mathcal V$ a reflexive Banach space.
Let $B \in \mathcal L (\mathcal W \times \mathcal V ; \R)$ and $f \in
\mathcal{V}'$, where  $\mathcal{V}'$ is the dual space of $V$.
Then, the problem
\begin{equation*}
\text{find } u \in \mathcal{W} \quad \text{ such that } \quad B(u,v) = f(v) \quad \forall v \in \mathcal{V}
\end{equation*}
is well-posed if and only if
\begin{equation}
\tag*{[BNB(i)]} \exists \alpha > 0 \quad \inf_{w \in \mathcal W} \sup_{v \in \mathcal V} \frac{B(w,v)}{\norm{w}_{\mathcal W} \norm{v}_{\mathcal V}} \ge \alpha;\label{eq:bnb1}
\end{equation}
\begin{equation}
\tag*{[BNB(ii)]} \forall v \in \mathcal V \quad \left( \forall w \in \mathcal W \quad B(w,v) = 0 \right) \quad \Rightarrow \quad ( v = 0 ). \label{eq:bnb2}
\end{equation}
\end{theorem}

In the following, we set $B_h:\Vh \times \Qh \longrightarrow \Vh \times \Qh$ defined as
$$
(B_h  (\xx v_h, q_h), (\xx w_h, z_h))_{L^2(\Omega)}=\mathcal{B}_h((\xx v_h, q_h);(\xx w_h, z_h))
$$
$\forall (\xx v_h, q_h), (\xx w_h, z_h) \in \Vh \times \Qh$. Under the
hypotheses of Proposition~\ref{prop:gen_inf_sup}, we show \newline that~\ref{eq:bnb1} and~\ref{eq:bnb2} are valid for the choice~$B={B}_h$, $\mathcal W=\mathcal V=
\Vh \times \Qh$ endowed with the energy norm~$\normE{(\cdot, \cdot)}$. This implies that problem~\eqref{eq:abs_stokes} is well-posed.

To this aim, we notice that
\begin{equation}\label{eq:weakcoerc}
\begin{aligned}
\mathcal{B}_h((\xx u_h, p_h);(\xx u_h, p_h)) &= a_h (\xx u_h , \xx u_h) + b_h (p_h , \xx u_h) - b_h (p_h, \xx u_h) + s_h (p_h, p_h) \\
& \gtrsim \normVh{\xx u_h}^2 + \seminormJ{p_h}^2 \quad \quad \forall \left( \xx u_h, p_h\right) \in \Vh \times \Qh,
\end{aligned}
\end{equation}
provided that  Assumption~\ref{ass:mesh_all} is valid and the stabilization constant~$\gamma_{v}$ appearing in Definition~\ref{def:sigmav_sigmap} is chosen sufficiently large.

\begin{proof}[Proof of \ref{eq:bnb1}]
Given~$\left( \xx u_h, p_h\right) \in \Vh \times\Qh$, we have
\begin{align}
\label{eq:bnb_a}
\mathcal{B}_h((\xx u_h, p_h);(\xx u_h, p_h))
&\leq  \mathbb{M} \normE{(\xx u_h, p_h)},
\end{align}
where $\normE{\cdot}$ is defined as in \eqref{eq:Enorm} and
$$
\mathbb{M} = \sup_{\substack{(\xx v_h, q_h) \in \Vh \times \Qh \\
(\xx v_h, q_h) \neq  (\xx 0, 0)}} \frac{\mathcal{B}_h((\xx u_h, p_h);(\xx v_h,
q_h))}{\normE{(\xx v_h, q_h)}}.
$$
Using~\eqref{eq:weakcoerc} and~\eqref{eq:bnb_a}, we get
\begin{equation}\label{eq:proofBNB1}
\normVh{\xx u_h}^2 + \seminormJ{p_h}^2 \lesssim \mathcal{B}_h((\xx u_h, p_h);(\xx u_h, p_h))  \leq \mathbb{M} \normE{(\xx u_h, p_h)}.
\end{equation}
Thanks to Proposition~\ref{prop:gen_inf_sup} and to the fact that
$$
b_h ( p_h , \xx v_h ) = \mathcal B_h((\xx u_h, p_h),(\xx v_h, 0)) - a_h(\xx u_h, \xx v_h) \quad \forall \xx v_h \in \Vh,
$$
we have
\begin{align*}
\beta_{h} \normQ{p_h} & \leq \sup_{ \xx 0 \neq\xx v_h\in \Vh} \frac{b_{h}(p_h ,
\xx v_h) }{\normVh{\xx v_h}} + \seminormJ{p_h} \\
& = \sup_{ \xx 0 \neq\xx v_h\in \Vh} \frac{\mathcal{B}_h((\xx u_h, p_h),(\xx v_h, 0)) - a_h(\xx u_h, \xx v_h) }{\normVh{\xx v_h}} + \seminormJ{p_h} \\
&\leq \sup_{ \xx 0 \neq\xx v_h\in \Vh} \frac{\abs{\mathcal{B}_h((\xx u_h, p_h),(\xx v_h, 0))}}{\normE{(\xx v_h, 0)}}
+ \sup_{ \xx 0 \neq\xx v_h\in \Vh} \frac{\abs{a_h(\xx u_h, \xx v_h)}}{\normVh{\xx v_h}} + \seminormJ{p_h} \\
& \lesssim \sup_{ \xx 0 \neq\xx v_h\in \Vh} \frac{\abs{\mathcal{B}_h((\xx u_h,
p_h),(\xx v_h, 0))}}{\normE{(\xx v_h, 0)}} + \normVh{\xx u_h} + \seminormJ{p_h}
\\
& = \mathbb{M} + \normVh{\xx u_h} + \seminormJ{p_h}.
\end{align*}
Using~\eqref{eq:proofBNB1}, we deduce
\begin{align*}
\beta_{h} ^2 \normQ{p_h}^2
& \lesssim \mathbb{M}^2 + \normVh{\xx u_h}^2 + \seminormJ{p_h}^2
\leq \mathbb{M}^2 + \mathbb{M} \normE{(\xx u_h, p_h)}.
\end{align*}
From the definition of $\normE{(\cdot, \cdot)}$, using again
\eqref{eq:proofBNB1}, the above bound, and the Young's inequality with a positive
parameter $\gamma$, we have
\begin{align*}
\beta_{h} ^{2}  & \normE{(\xx u_h, p_h)}^2 =\beta_{h} ^{2}(\normVh{\xx u_h}^2 +
    \normQ{p_h}^2 + \seminormJ{p_h}^2) \\ &\lesssim \beta_{h} ^{2} \mathbb{M}
    \normE{(\xx u_h, p_h)}  + \beta_{h} ^{2}\normQ{p_h}^2 \lesssim \beta_{h} ^{2}
    \mathbb{M} \normE{(\xx u_h, p_h)} + \mathbb{M}^2 + \mathbb{M} \normE{(\xx u_h,
    p_h)}\\ & =(1+\beta_{h} ^{2})\mathbb{M} \normE{(\xx u_h, p_h)} +  \mathbb{M}^2
    \leq \gamma \normE{(\xx u_h, p_h)}^2 +
    \left(1+\frac{(1+\beta_{h} ^{2})^2}{\gamma} \right)
    \mathbb{M}^2.
\end{align*}
Thus, we write
\begin{align*}
(\beta_{h} ^{2}- \gamma) \normE{(\xx u_h, p_h)}^2
& \lesssim  \left(1+\frac{(1+\beta_{h} ^{2})^2}{\gamma} \right) \mathbb{M}^2.
\end{align*}
Choosing $\gamma$ to be equal to $\beta_{h} ^{2}/2C$, being~$C$ the hidden
constant in the above inequality, we finally arrive at
\begin{align*}
\normE{(\xx u_h, p_h)}^2
 \lesssim \left(\frac{1}{\beta_{h} ^{2}}+\frac{(1+\beta_{h} ^{2})^2}{\beta_{h} ^{4}} \right) \mathbb{M}^2
 \lesssim \frac{1}{\beta_{h} ^{2}} \mathbb{M}^2,
\end{align*}
i.e., $\normE{(\xx u_h, p_h)} \lesssim  \alpha \mathbb{M}$ with~$\alpha =O(\beta_{h})$. The assertion follows from the definition of~$\mathbb M$.
\end{proof}
%%%%%%%%%%%%%%%%%%%

Next, we show that \ref{eq:bnb2} is valid, provided that the stabilization constant~$\gamma_v$ appearing in Definition \ref{def:sigmav_sigmap} is chosen sufficiently large.

\begin{proof}[Proof of \ref{eq:bnb2}]
Let $(\xx u_h, p_h)\in \Vh \times \Qh$ be such that
$$
\mathcal{B}_h((\xx v_h, q_h);(\xx u_h, p_h)) = 0 \qquad \forall (\xx v_h, q_h) \in \Vh
\times \Qh.
$$
By taking $(\xx v_h, q_h)=(\xx u_h, p_h)$ and using \eqref{eq:weakcoerc}, the couple $(\xx {u_h}, p_h)$ satisfies
$$
0 = a_h (\xx u_h , \xx u_h) +  b_h (p_h , \xx u_h) - b_h (p_h, \xx u_h) + s_h (p_h, p_h).
$$
Provided that the stabilization constant $\gamma_v$ in Definition~\ref{def:sigmav_sigmap} is chosen sufficiently large, this implies
$$
\normVh{ \xx u_h }^2 + s_h(p_h,p_h) \lesssim 0,
$$
whence $\xx u_h = \xx 0$ follows.

Next, we prove that $p_h = 0$.
Thanks to the continuous \textit{inf-sup} condition~\eqref{continuous:inf-sup2}, we have
\begin{equation*}
\normQ{p_h}^2=-b(p_h, \xx v_{p_h})= \mathcal{B}_h((\xx v_{p_h}, 0);(\xx 0, p_h))=0,
\end{equation*}
whence the assertion follows.
\end{proof}
%%%%%%%%%%%%%%%%%

Summarizing the above computations, we eventually state the main result of the
section, namely the well-posedness of problem~\eqref{eq:abs_stokes}.
\begin{theorem}
Under the hypotheses of Proposition~\ref{prop:gen_inf_sup}, \ref{eq:bnb1} and~\ref{eq:bnb2} are valid.
Notably, the constant~$\alpha$ in~\eqref{eq:bnb1} satisfies $\alpha =O(\beta_{h})$, where $\beta_{h}$ is defined as in Proposition~\ref{prop:gen_inf_sup}.
Therefore, thanks to Theorem~\ref{th:BNB}, the discrete problem~\eqref{eq:abs_stokes} is
well-posed.
\end{theorem}

The constant~$\alpha$ in~\eqref{eq:bnb1} deteriorates as the polynomial
degree grows. This is due to the use of the polynomial inverse estimates, which
yield a discrete \textit{inf-sup} constant depending on the polynomial degree.

%% MESH
\begin{figure}[!htbp]
    \centering
    \includegraphics[width=0.35\textwidth]{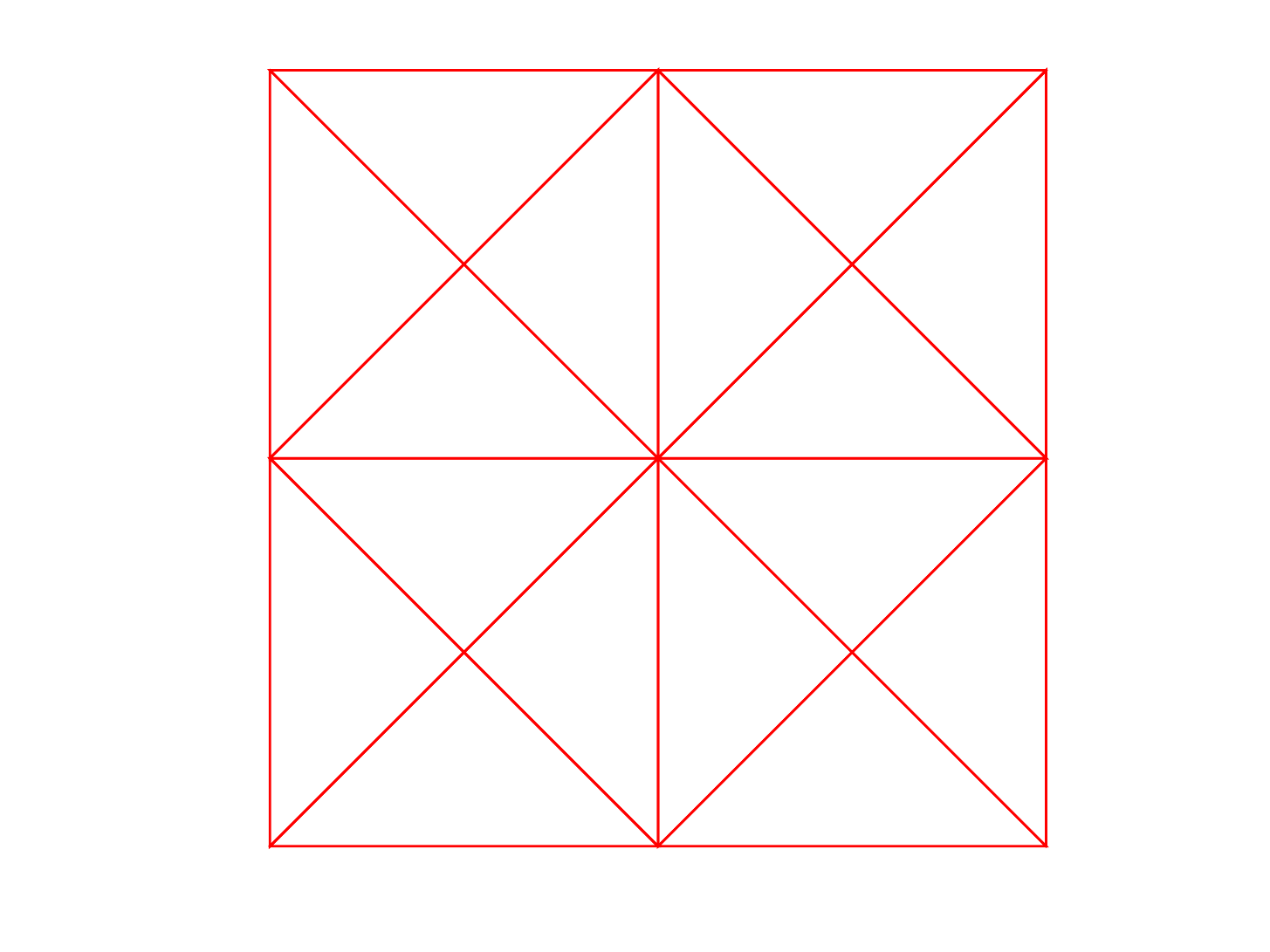}%
    \includegraphics[width=0.35\textwidth]{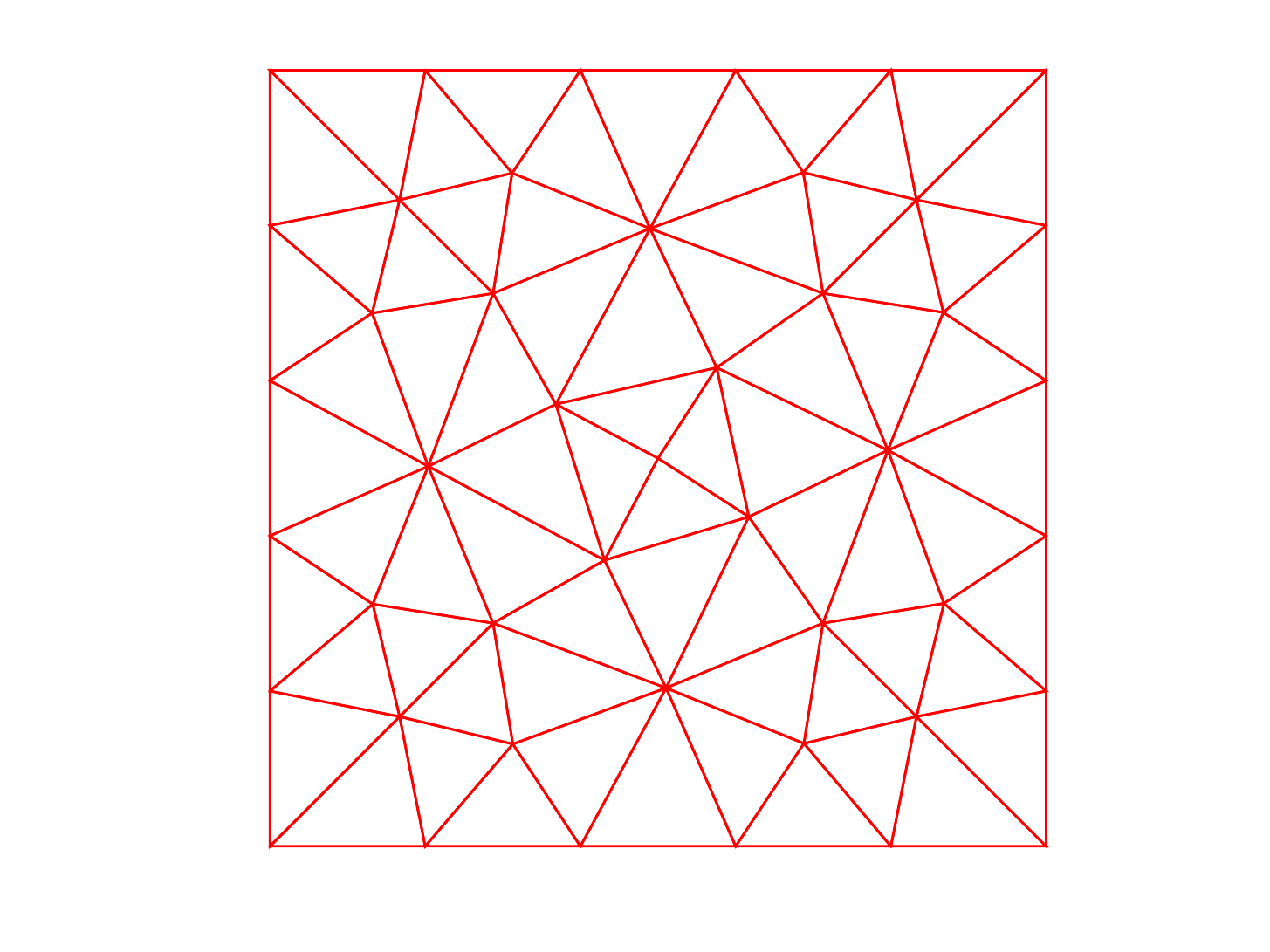}%
    \includegraphics[width=0.35\textwidth]{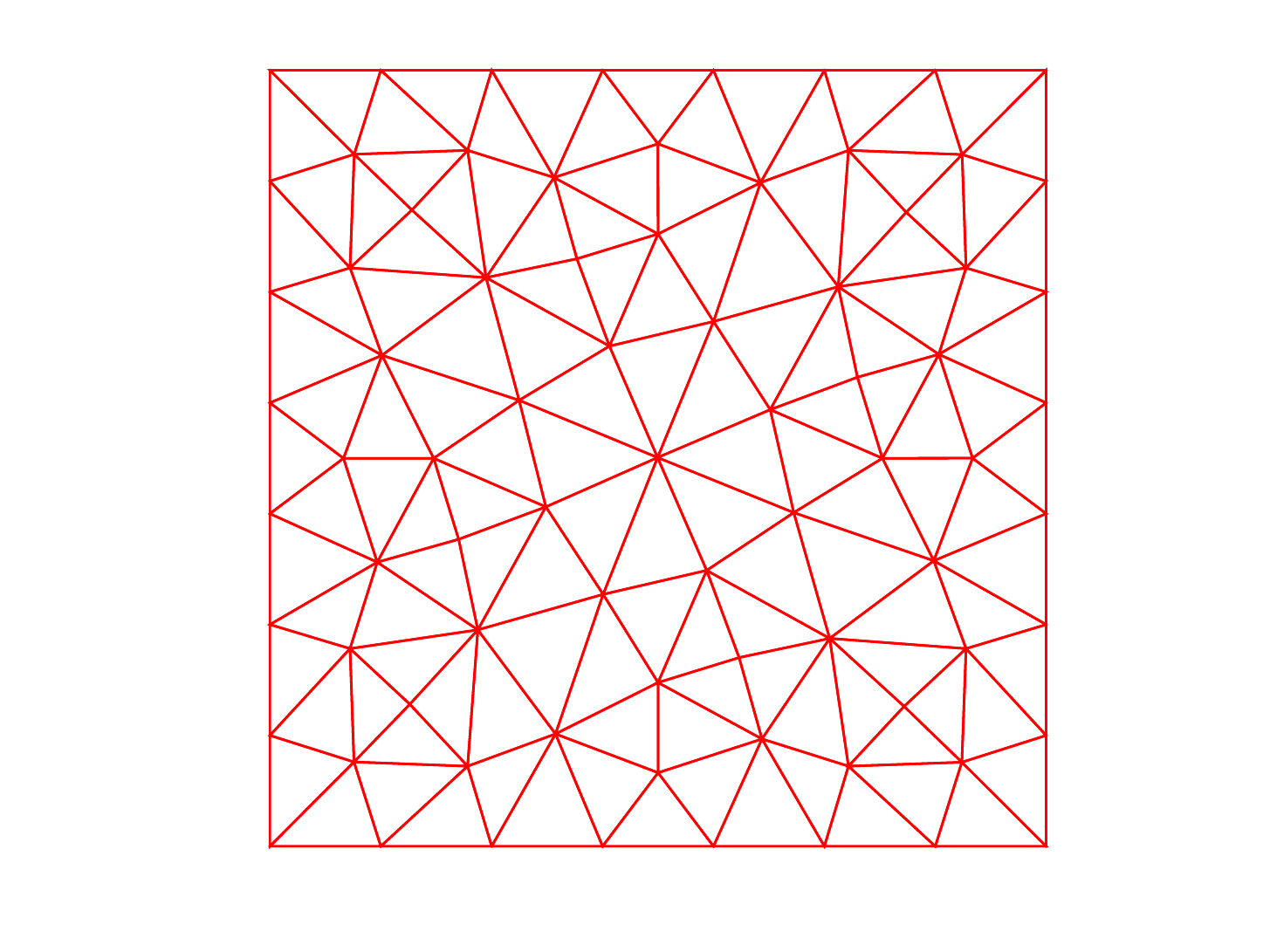}\\%
    \includegraphics[width=0.35\textwidth]{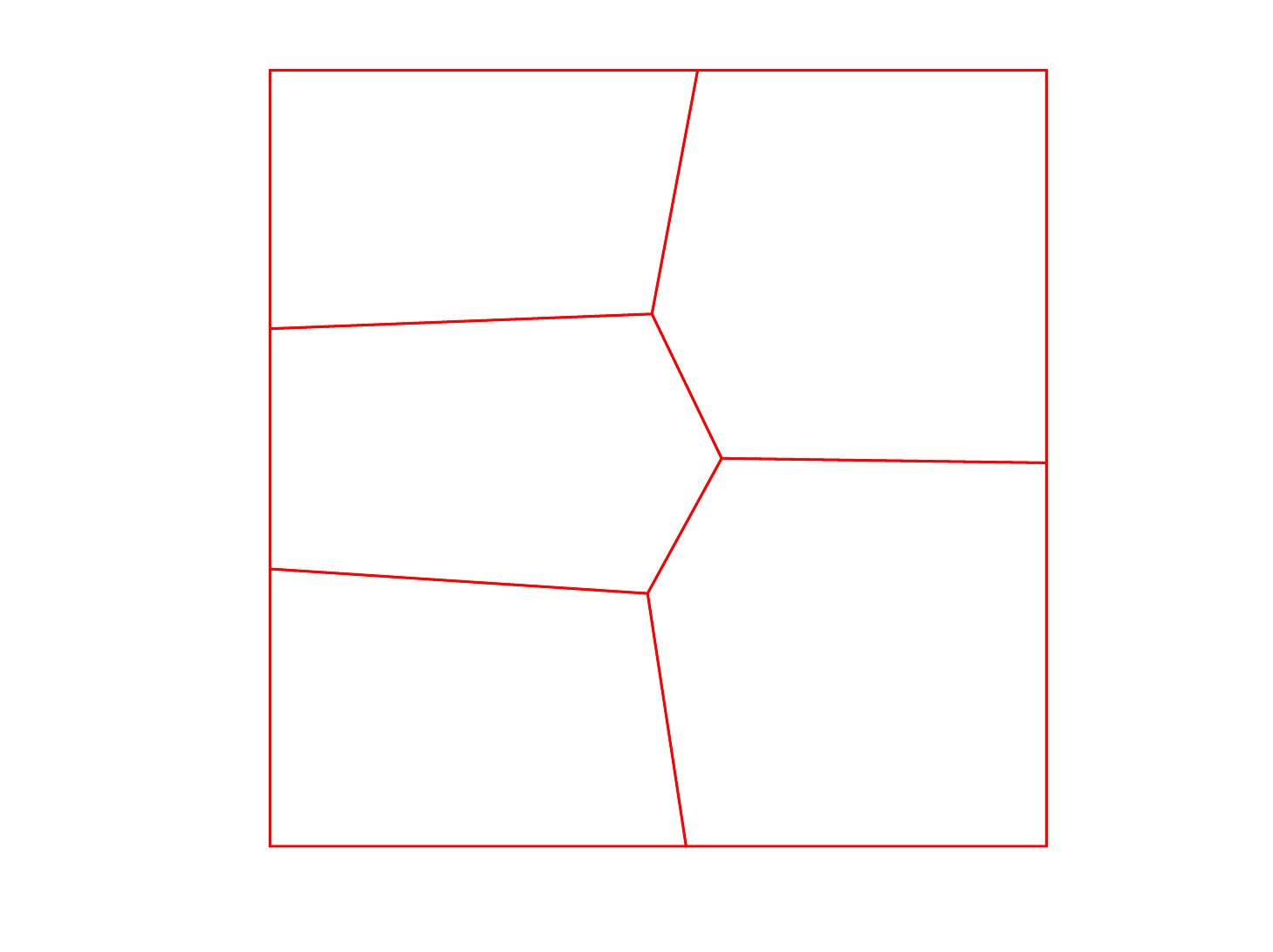}%
    \includegraphics[width=0.35\textwidth]{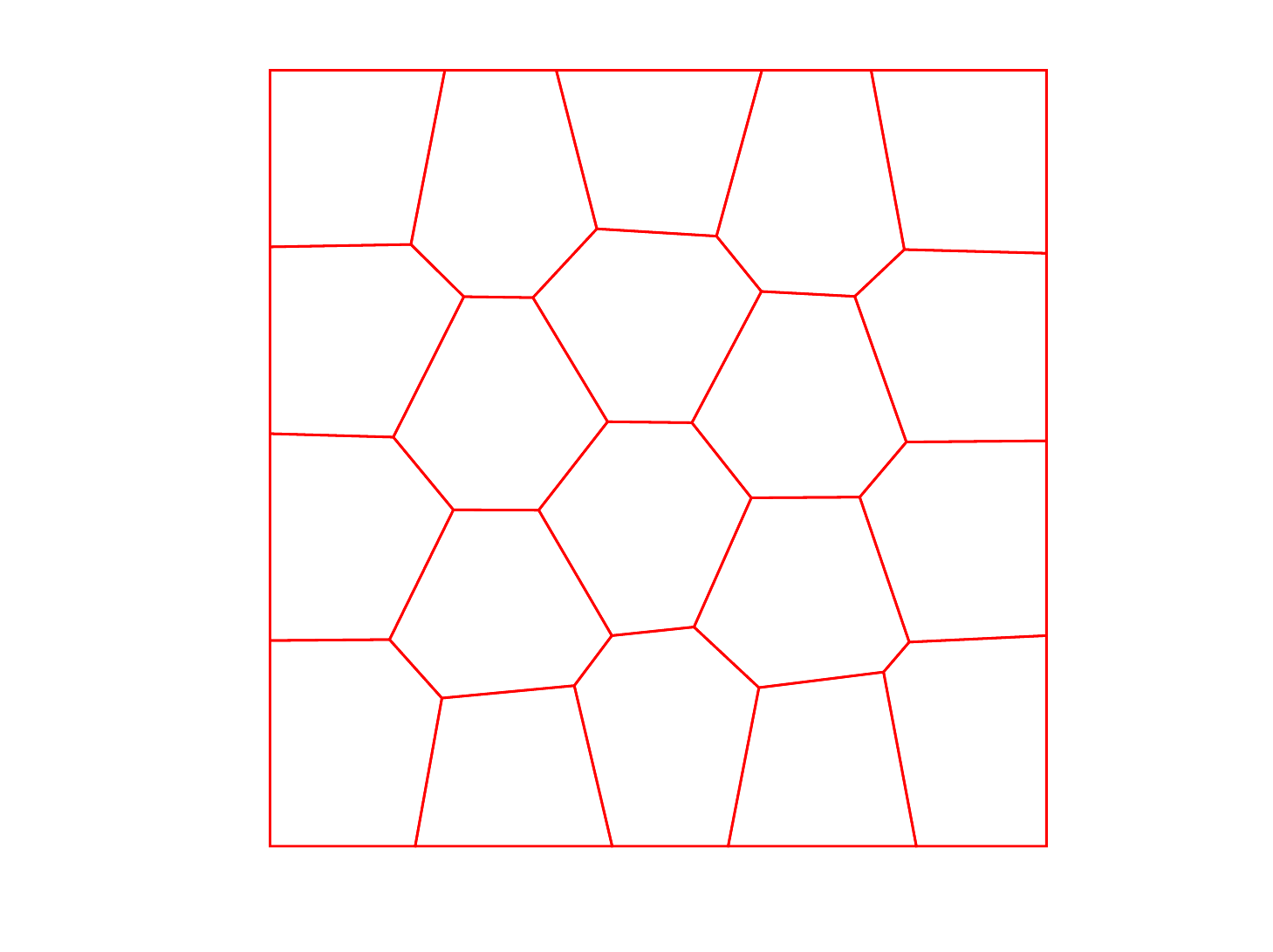}%
    \includegraphics[width=0.35\textwidth]{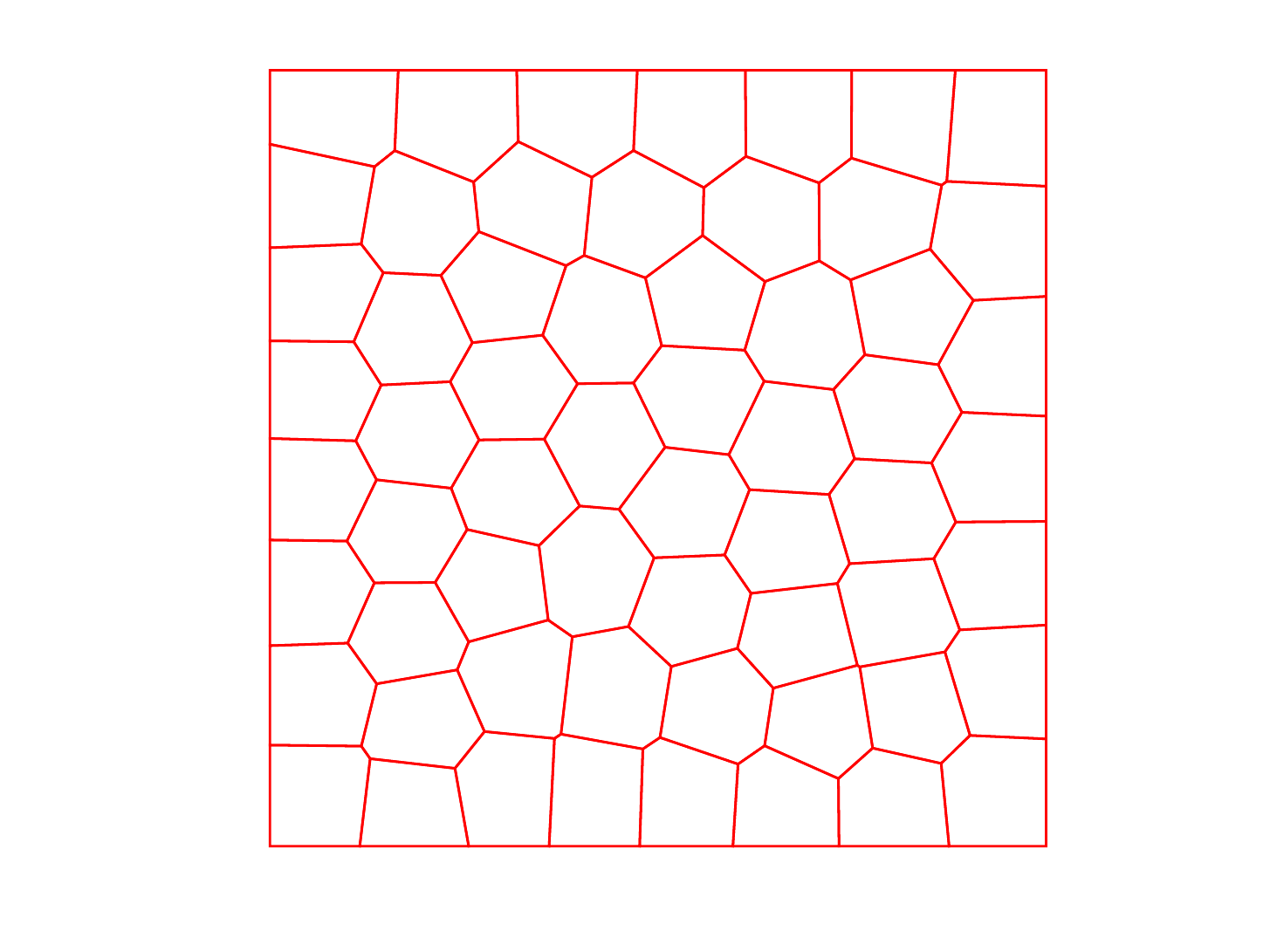}\\%
    \includegraphics[width=0.35\textwidth]{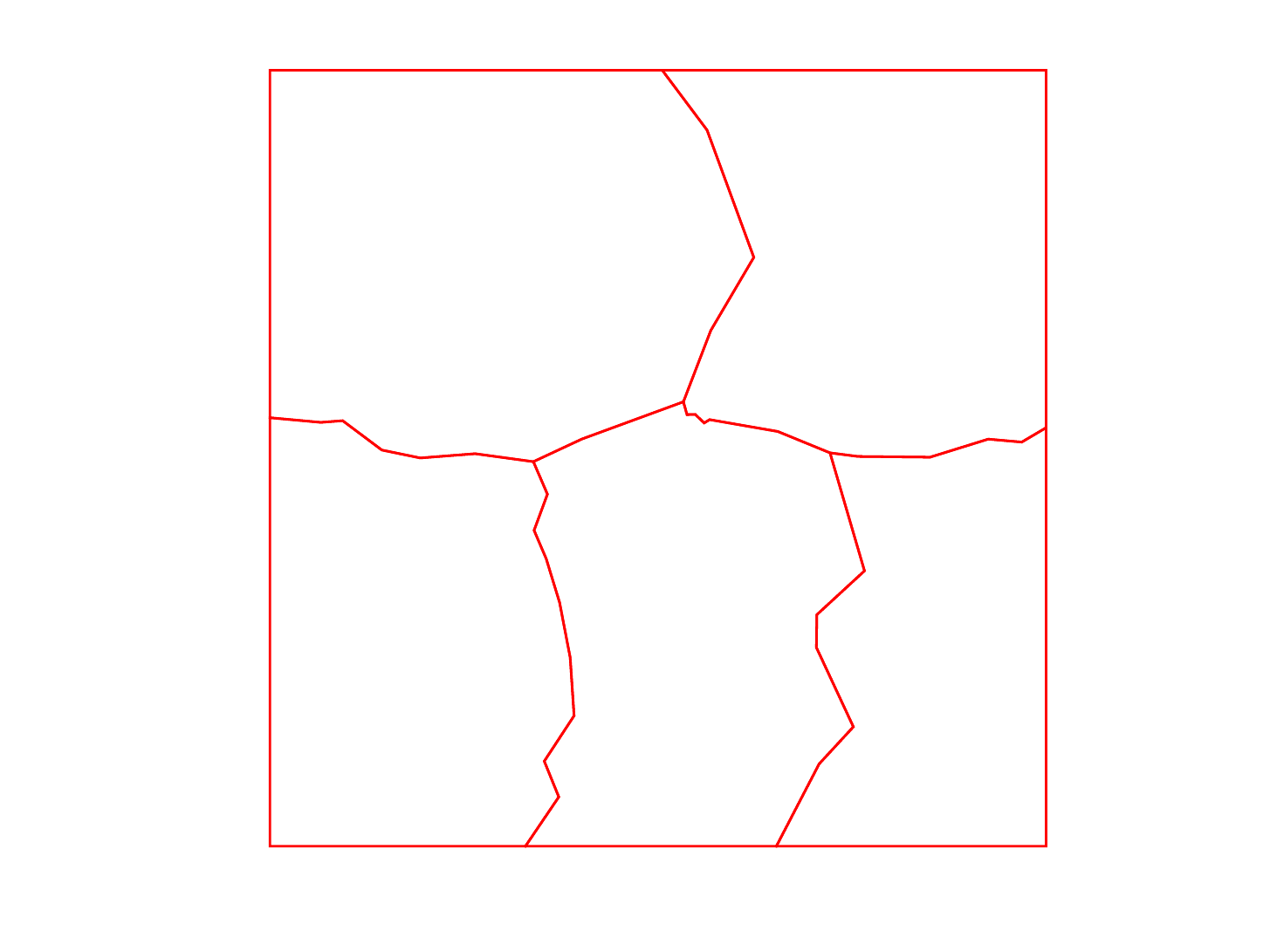}%
    \includegraphics[width=0.35\textwidth]{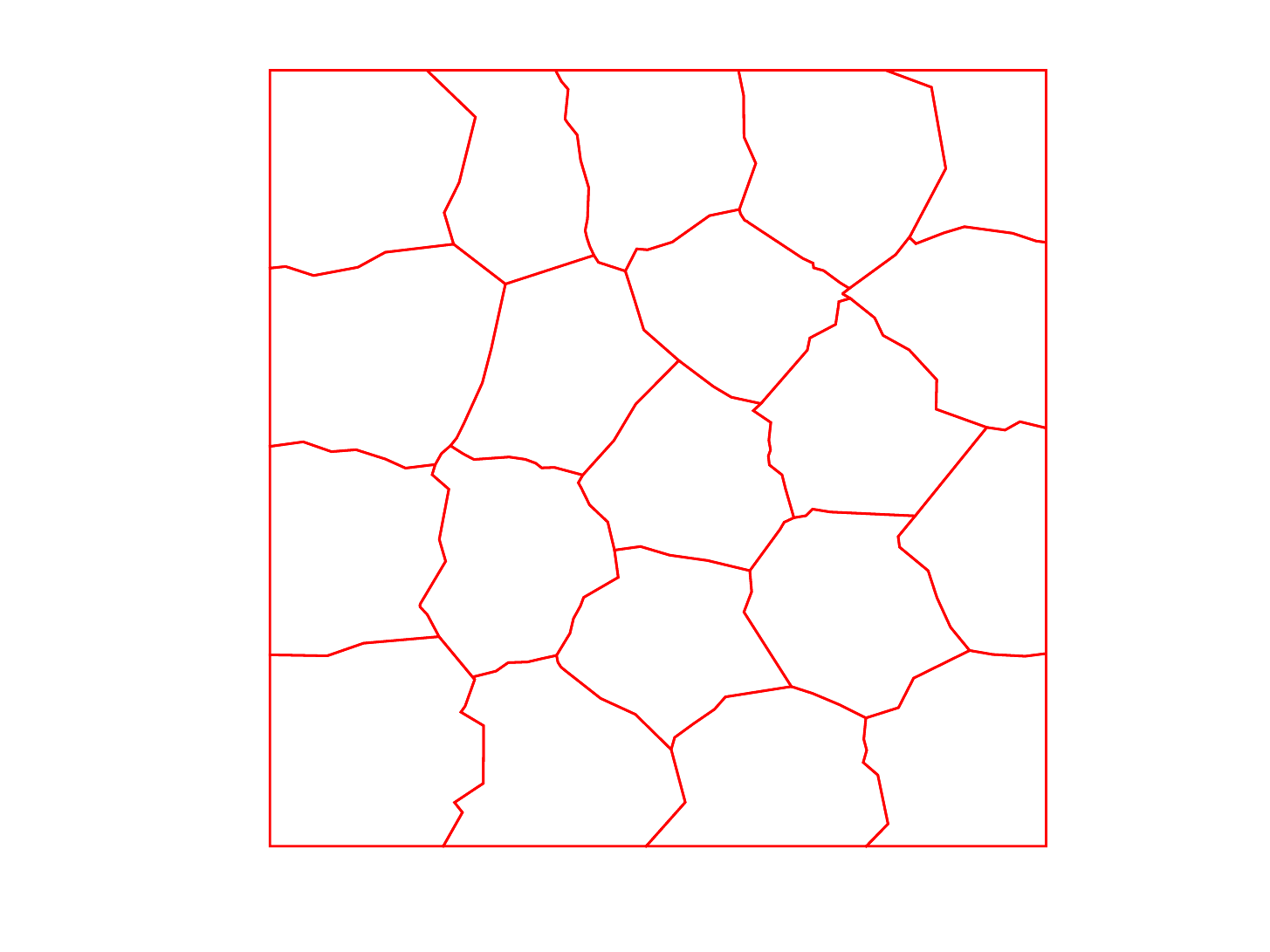}%
    \includegraphics[width=0.35\textwidth]{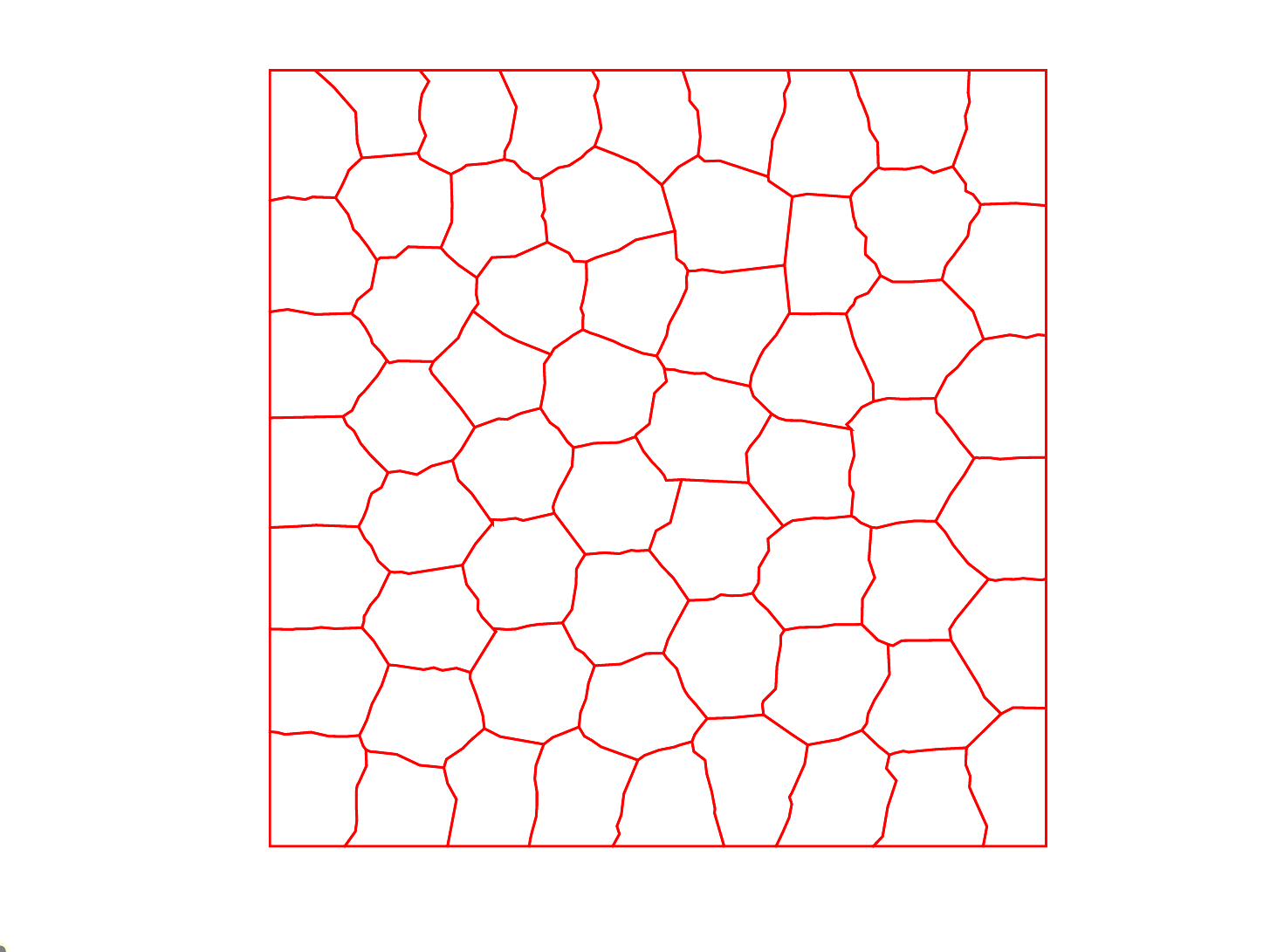}\\%
    \includegraphics[width=0.35\textwidth]{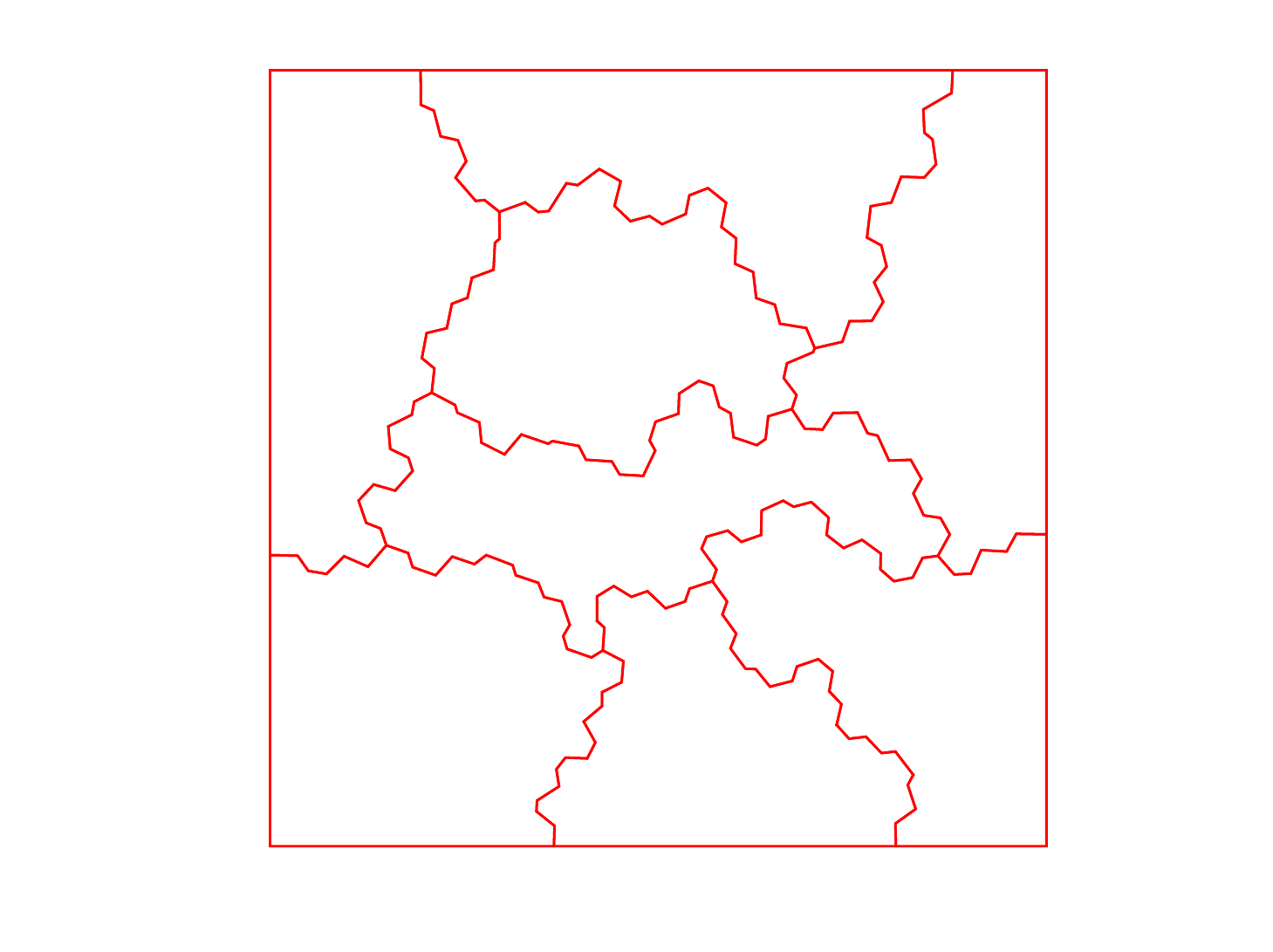}%
    \includegraphics[width=0.35\textwidth]{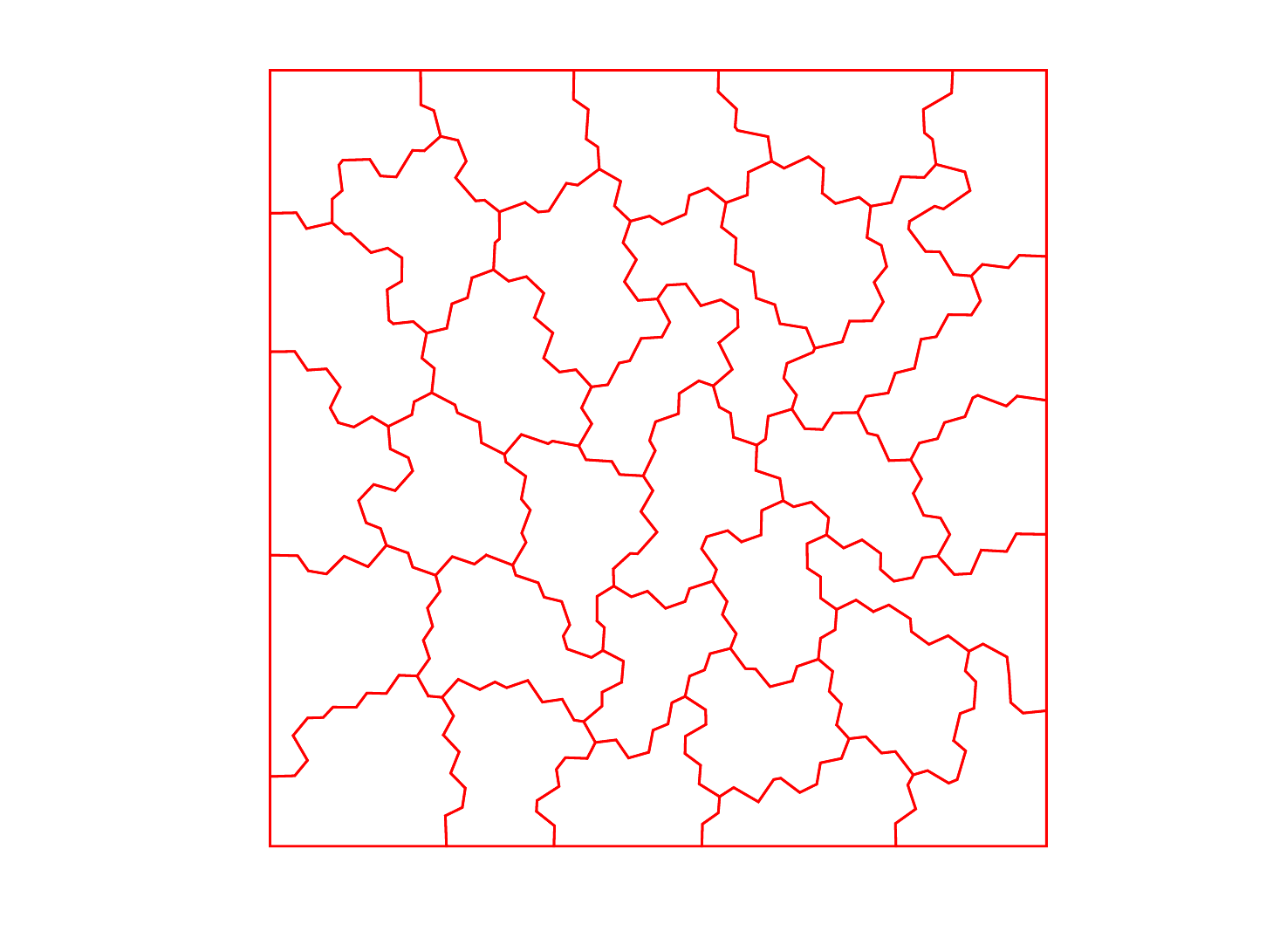}%
    \includegraphics[width=0.35\textwidth]{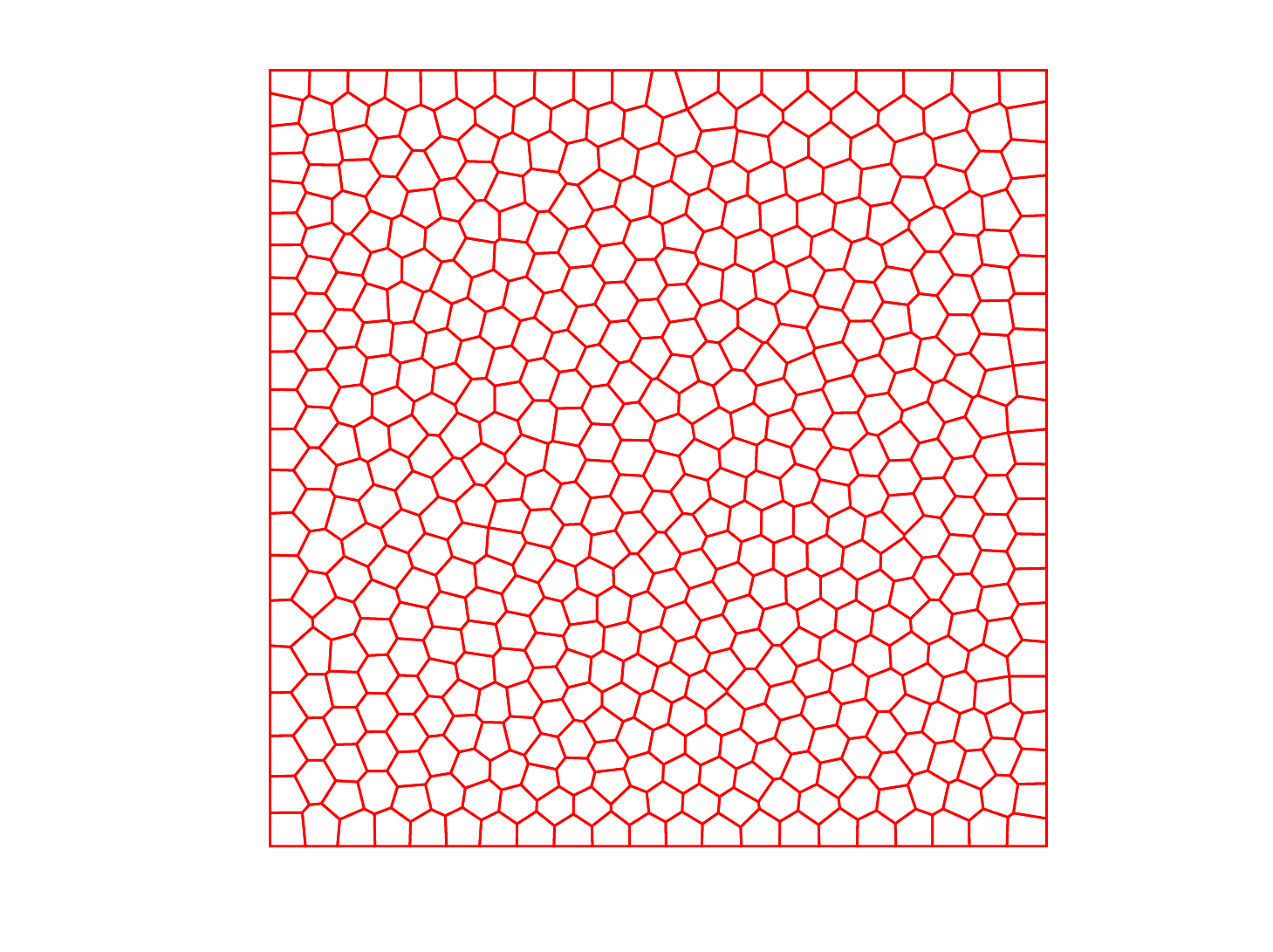}%
    \caption{From top to bottom: triangular  meshes with $N_{el} = 16, 68, 124$
    (from left to right); regular polygonal meshes with $N_{el} = 5, 20,
    60$ (from left to right); distorted polygonal  meshes with $N_{el} = 5, 20, 60$
    (from left to right); agglomerated polygonal meshes  with
    $N_{el} = 8, 32, 512$ (from left to right).}
    \label{fig:meshes}
\end{figure}

\section{Numerical evaluation of the generalized \textit{inf-sup} constant}
\label{sec:computations_inf_sup}

Denote the shape functions of~$\xx V_h^{\ell}$ and $Q_h^m$ by~$\left\{ \xx \varphi_i \right\}_{i=1}^{N_{\xx u}}$ and~$\left\{ \psi_j \right\}_{j=1}^{N_p}$,
and the corresponding number of degrees of freedom by~$N_{\xx u}$ and $N_p$, respectively.
We write~$\xx u_h = \sum_{j=1}^{N_{\xx u}} u_j \xx \varphi_i$ and $p_h = \sum_{j=1}^{N_p} p_j \psi_i$.
The algebraic form of the stationary Stokes problem corresponding to the problem in~\eqref{eq:semiDiscForm} reads
\begin{equation*}
    \left[
    \arraycolsep=8pt\def\arraystretch{2}
    \begin{array}{cc}
        A_h & B_h^T \\
        B_h & -S_h \\
    \end{array} \right]
    \left[
    \arraycolsep=8pt\def\arraystretch{2}
    \begin{array}{c}
        \xx U \\
        \xx P \\
    \end{array} \right]
    = \left[
    \arraycolsep=8pt\def\arraystretch{2}
    \begin{array}{c}
        \xx {F_h}\\
        \xx 0\\
    \end{array} \right],
\end{equation*}
where $\xx U\in \mathbb{R}^{N_{\xx u}}$ and $\xx P\in \mathbb{R}^{N_{\xx p}}$ are the vectors collecting the expansion coefficients $\{u_j\}_j$ and $\{p_j\}_j$, respectively,
whereas $A_h$, $B_h$, $S_h$, and~$\xx{f_h}$ denote the matrix representations of the discrete bilinear forms in~\eqref{eq:form_a_h}--\eqref{eq:form_s} and right-hand side in~\eqref{eq:semiDiscForm}.

We recall that the discrete \textit{inf-sup} condition given in Proposition~\ref{prop:gen_inf_sup} is given by
\begin{equation*}
\begin{aligned}
& \sup_{ \xx 0\neq \xx v_h \in \Vh} \frac{b_h(q_h , \xx v_h) }{\normVh{\xx v_h}} + \eta \seminormJ{q_h} \ge
\beta_{h} \normQ{q_h}
&& \forall q_h \in \Qh.
\end{aligned}
\end{equation*}
Here, we have added a parameter $\eta=\{0,1\}$, so as to address numerically the case where no pressure stabilization is added in the discrete formulation.
Introduce  the generalized eigenvalue problem
\begin{equation}
\label{eq:GEP}
G_h \xx x = \lambda T_h \xx x,
\end{equation}
with~$G_h = B_h A_h^{-1} B_h^T + \eta S_h$,  $T_h = M$, where $M$ is the mass matrix.
We distinguish two cases:
\begin{itemize}
\item If $\eta = 1$ (pressure stabilization), the discrete \textit{inf-sup} constant satisfies
$$
\beta_h \langle M \xx q , \xx q \rangle^{1/2} \le \langle B_h A_h^{-1} B_h^T \xx q , \xx q \rangle^{1/2} + \langle S_h \xx q , \xx q \rangle^{1/2} \quad \forall \xx q \in \R^{N_p}, \xx q \neq \xx 1{}.
$$
By noting that~$ a + b \le \left( \sqrt{a} + \sqrt{b} \right)^2 \le 2a + 2b$, we have
\begin{align*}
\beta_h^2 &\langle M \xx q , \xx q \rangle \le \langle B_h A_h^{-1} B_h^T \xx q
, \xx q \rangle + \langle S_h \xx q , \xx q \rangle \\
&= \langle \left( B_h A_h^{-1} B_h^T + S_h \right) \xx q, \xx q \rangle \qquad
\forall \xx q \in \R^{N_p}, \xx q \neq 1
\end{align*}
and hence
$$
\beta_h^2 = \min_{\substack{\xx q \in \R^{N_p} \\ \xx q \neq \xx 1}} \frac{
\langle \left( B_h A_h^{-1} B_h^T + S_h \right) \xx q, \xx q \rangle}
{\langle M \xx q , \xx q \rangle} = \min_{\substack{\xx q \in \R^{N_p} \\ \xx q
\neq \xx 1}} \frac{ \langle G_h \xx q , \xx q \rangle}{\langle T_h \xx q ,
\xx q \rangle}.
$$
\item If $\eta = 0$ (no pressure stabilization), the discrete \textit{inf-sup} constant satisfies
\begin{align*}
\begin{split}
\beta_h &= \min_{\substack{\xx q \in \R^{N_p} \\ \xx q \neq \xx 1}}
\max_{\substack{ \xx v \in \R^{N_{\xx u}} \\ \xx v \neq \xx 0}} \frac{\abs{
\langle \xx q , B_h \xx v \rangle}} { \langle A_h \xx v , \xx v \rangle^{1/2}
\langle M \xx q , \xx q \rangle^{1/2}}\\
&= \min_{\substack{\xx q \in \R^{N_p} \\ \xx q \neq \xx 1}} \frac{1}{\langle M \xx q , \xx q \rangle^{1/2}} \max_{\substack{ \xx w \in \R^{N_{\xx u}} \\ \xx w = A_h^{1/2} \xx v \neq \xx 0}} \frac{\abs{ \langle \xx q , B_h A_h^{-1/2} \xx w \rangle} }{ \langle \xx w , \xx w \rangle^{1/2} }\\
&= \min_{\substack{\xx q \in \R^{N_p} \\ \xx q \neq \xx 1}} \frac{1}{\langle M \xx q , \xx q \rangle^{1/2}} \max_{\substack{ \xx w \in \R^{N_{\xx u}} \\ \xx w \neq \xx 0}} \frac{\abs{ \langle A_h^{-1/2} B_h^T \xx q , \xx w \rangle} }{ \langle \xx w , \xx w \rangle^{1/2} }.
\end{split}
\end{align*}
By noting that the maximum is realized for $\xx w = A_h^{-1/2} B_h^T \xx q$, we have
\begin{align*}
\begin{split}
\beta_h^2 &=  \min_{\substack{\xx q \in \R^{N_p} \\ \xx q \neq \xx 1}} \frac{
\langle A_h^{-1/2} B_h^T \xx q , A_h^{-1/2} B_h^T \xx q \rangle}{\langle M \xx q
, \xx q \rangle} = \min_{\substack{\xx q \in \R^{N_p} \\
\xx q \neq \xx 1}} \frac{ \langle B_h A_h^{-1} B_h^T \xx q , \xx q
\rangle}{\langle M \xx q , \xx q \rangle} \\
& = \min_{\substack{\xx q \in \R^{N_p} \\ \xx q \neq \xx 1}} \frac{ \langle G_h
\xx q , \xx q \rangle}{\langle T_h \xx q , \xx q \rangle}.
\end{split}
\end{align*}
\end{itemize}
By solving the discrete eigenvalue problem in equation~\eqref{eq:GEP}, we have that
$$
\beta_h \approx \min_{\lambda_i > 0} \sqrt{ \lambda_i }.
$$
To estimate numerically $\beta_h$, we consider the Stokes problem on the unit square
domain $\Omega = (0,1)^2$. We computed
$\beta_h$ on several sequences of meshes, namely triangular, regular, distorted,
and agglomerated polygonal meshes; see Figure~\ref{fig:meshes} for an
illustrative example of the considered grids.
The regular polygonal meshes have been generated via \texttt{PolyMesher}
\cite{talischi2012polymesher}, while the distorted polygonal ones are generated
starting from a regular grid and  randomly  adding grid nodes on the edges to obtain elements with a large number of possibly degenerating edges.
The resulting elements may be non-convex.
The sequence of agglomerated polygonal meshes are generated by agglomerating elements starting from an initial Voronoi tessellation; see the last row of Figure~\ref{fig:meshes}.
To solve the generalized eigenvalue problem~\eqref{eq:GEP}, we employ the \texttt{eigs} command of Matlab.\\

%%%%%%%%%%%%%%%%%%%%%%%%%%%%%%%%%%%%%%%%%%%%%%%%%%%%%%%%%%%%%%%%%%%%
We first investigate the behavior of~$\beta_h$ for fixed polynomial approximation orders for the velocity and the pressure and varying the mesh size.
In Figures~\ref{fig:infsup_h_all}  and~\ref{fig:infsup_h2_all}, we report the computed values of~$\beta_h$ as a function of the mesh size $h$ for different mesh configurations
and different choices of the discrete velocity and pressure spaces $\mathcal P^{m+k} - \mathcal P^m$, $k=0,1,2,3,4$.
In the stabilized cases $\eta = 1$, the constant $\beta_h$ is uniformly bounded from~$0$ independently of the mesh size.
This is in agreement with the result shown in Proposition~\ref{prop:gen_inf_sup}.
Furthermore, as predicted in Proposition~\ref{prop:gen_inf_sup}, $\beta_h$ depends on~$m$ for all the considered mesh configurations, at least when~$m=\ell$.
From the numerical computations obtained in the no pressure stabilization cases $\eta=0$ with~$k=1,2,3,4$, we draw the following conclusions:
\emph{(i)} $\beta_h$ is independent of~$h$ for all the considered mesh configurations except for agglomerated meshes, where we can detect a mild dependence;
\emph{(ii)} the dependence of~$\beta_h$ on the velocity and pressure polynomial approximation degrees is stronger than in the stabilized case.

%%%%%%%%%%%%%%%%%%%%%%%%%%%%%%%%%%%%%%%%%%%%%%%%%%%%%%%%%%%%%%%%%%%%
\begin{figure}[!htbp]
\centering
\begin{subfigure}{1.0\textwidth}
\centering
\includegraphics[width=0.343\textwidth]{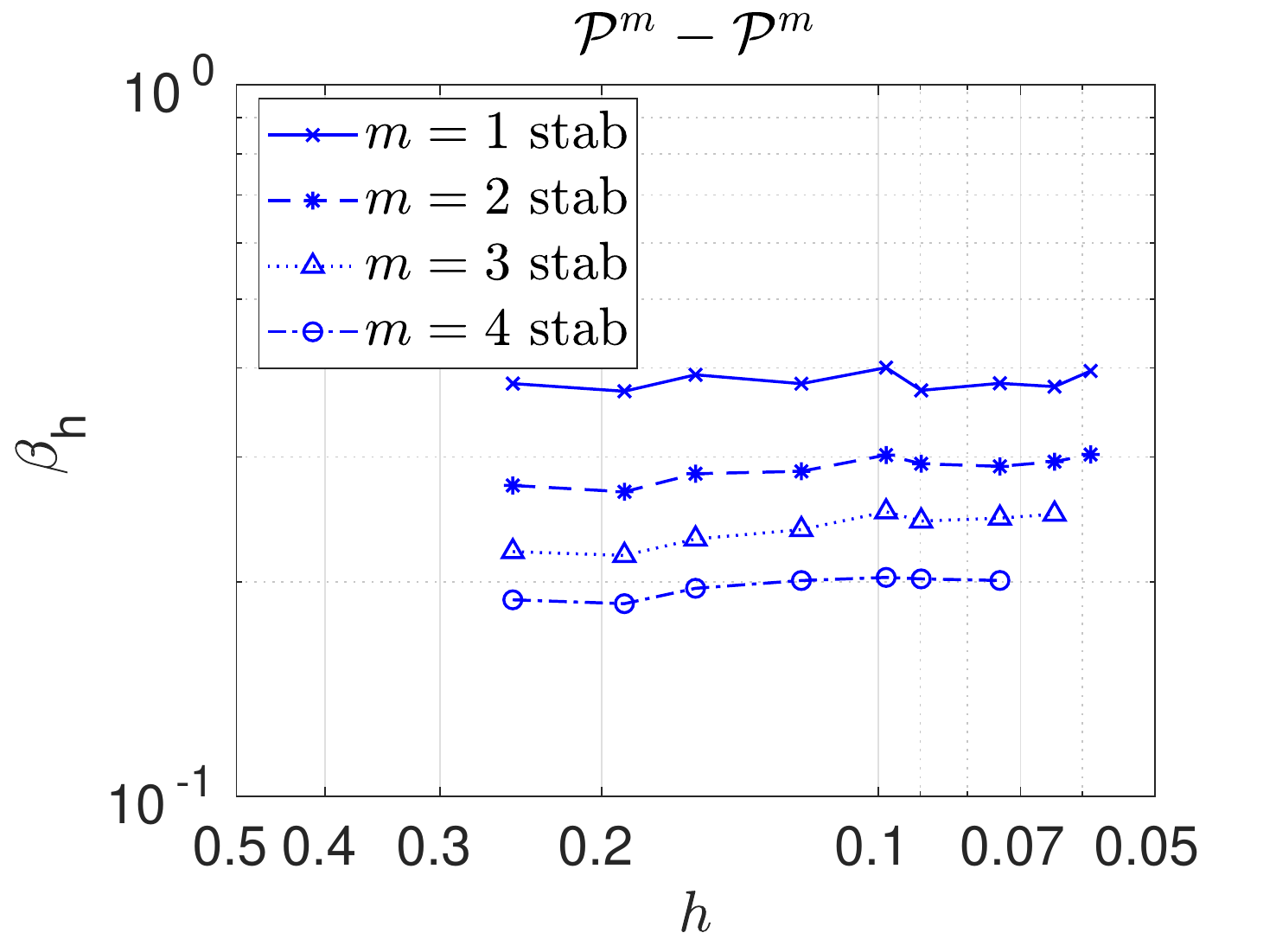}
\hspace{-0.48cm}
\includegraphics[width=0.343\textwidth]{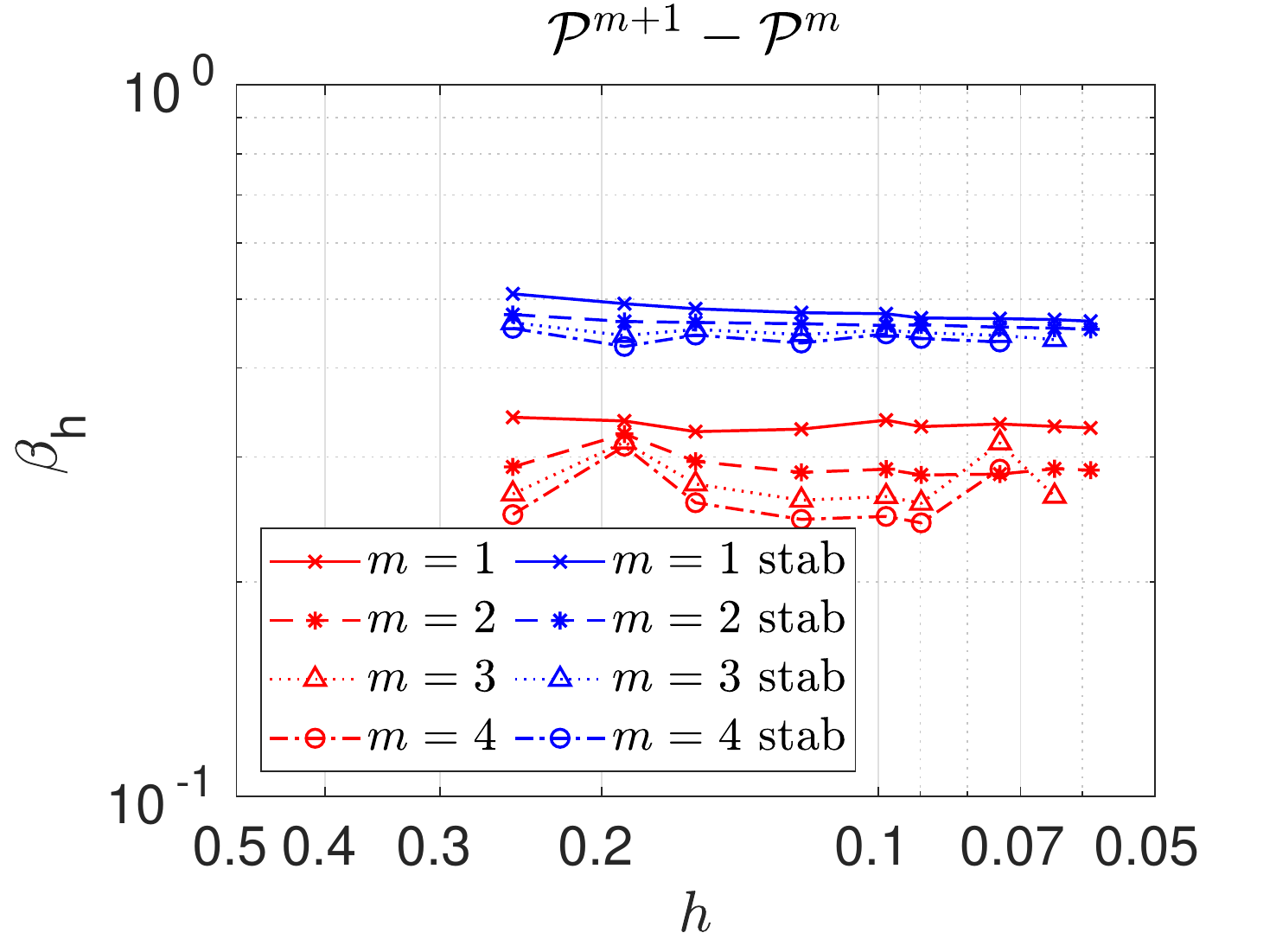}
\hspace{-0.48cm}
\includegraphics[width=0.343\textwidth]{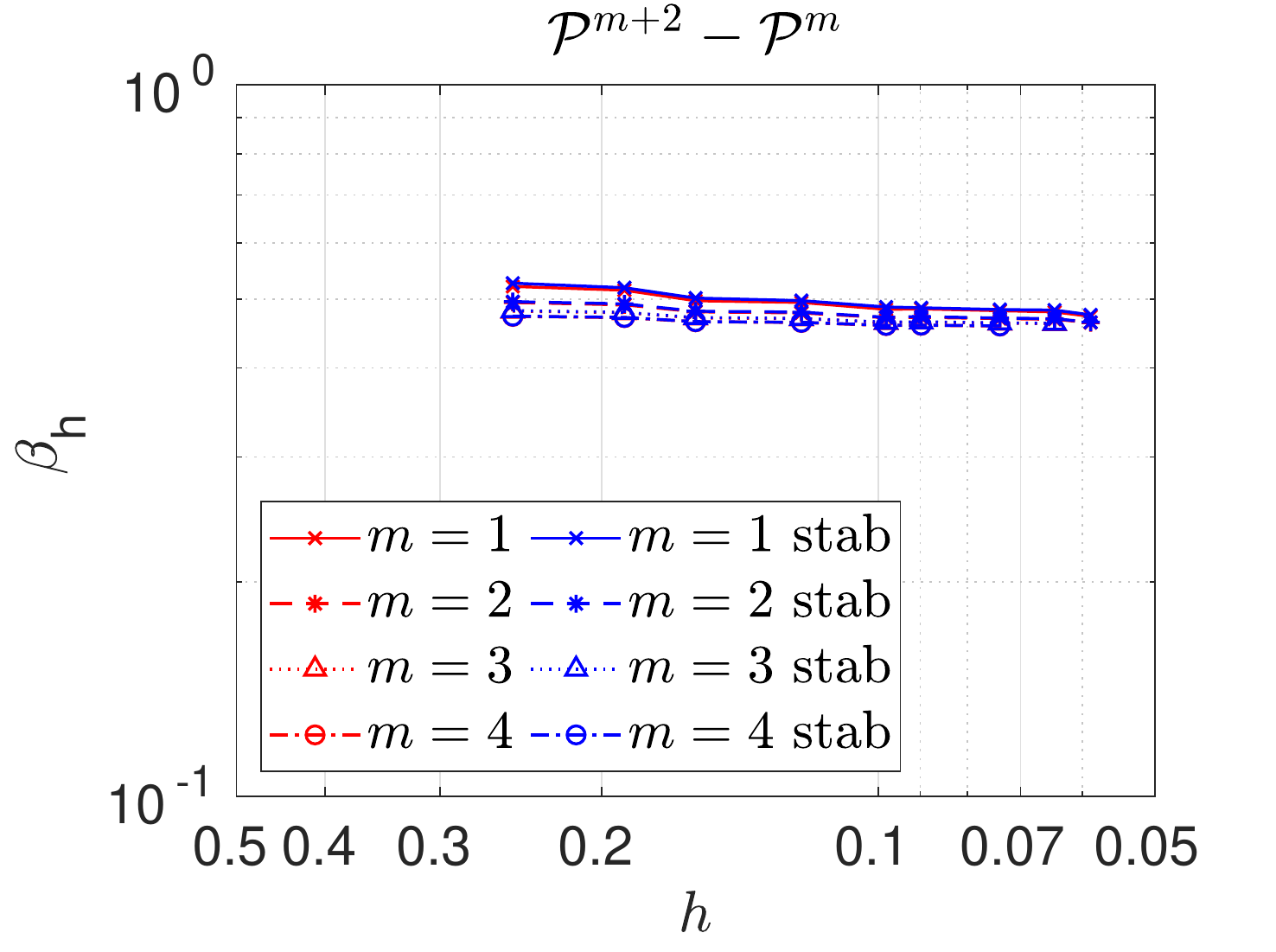}
\caption{Triangular grids}
\label{fig:infsup_h_tria}
\end{subfigure}
%%%%%%%%%%%%%%%%%%%%
\begin{subfigure}{1.0\textwidth}
\centering
\includegraphics[width=0.343\textwidth]{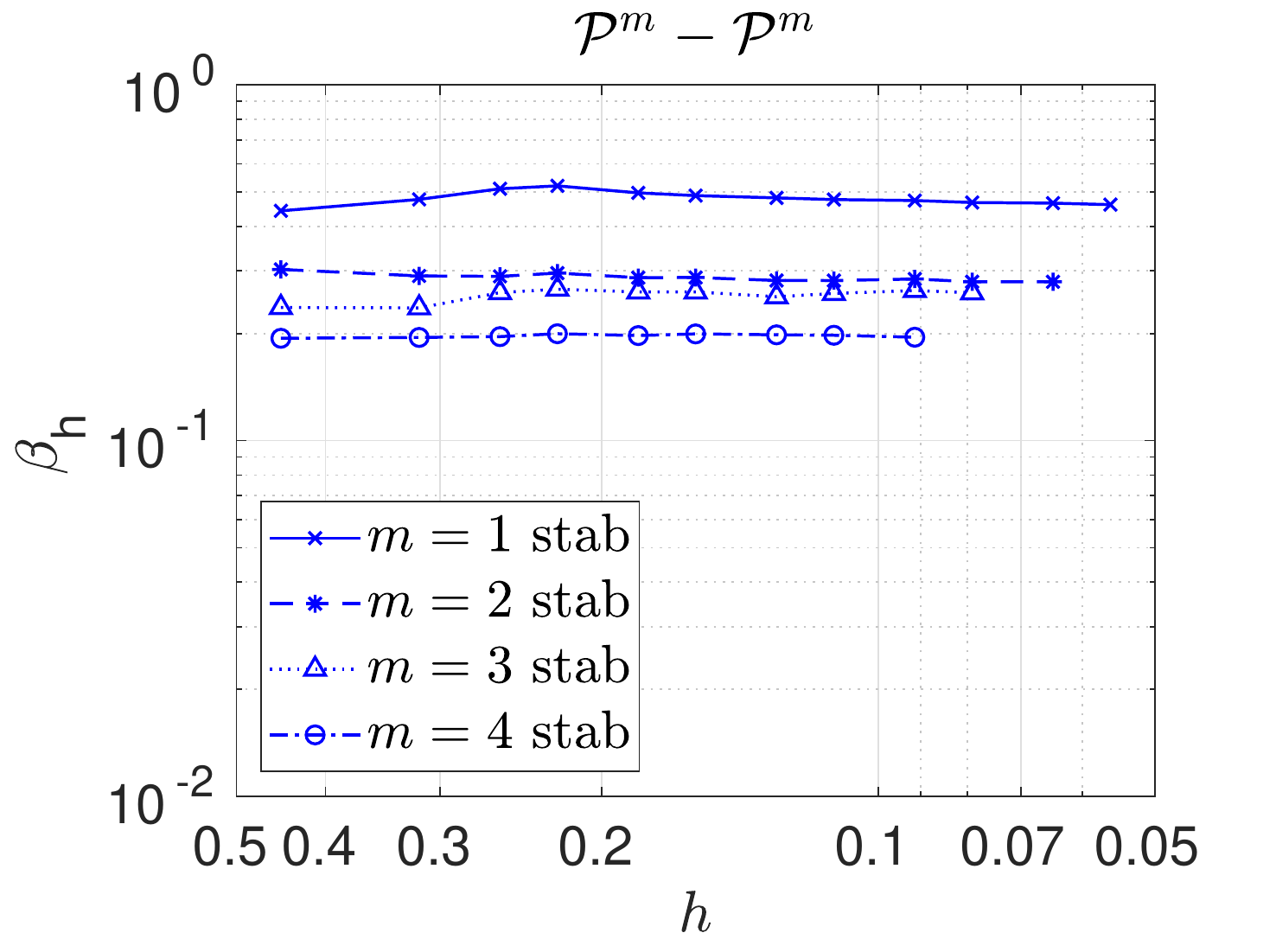}
\hspace{-0.48cm}
\includegraphics[width=0.343\textwidth]{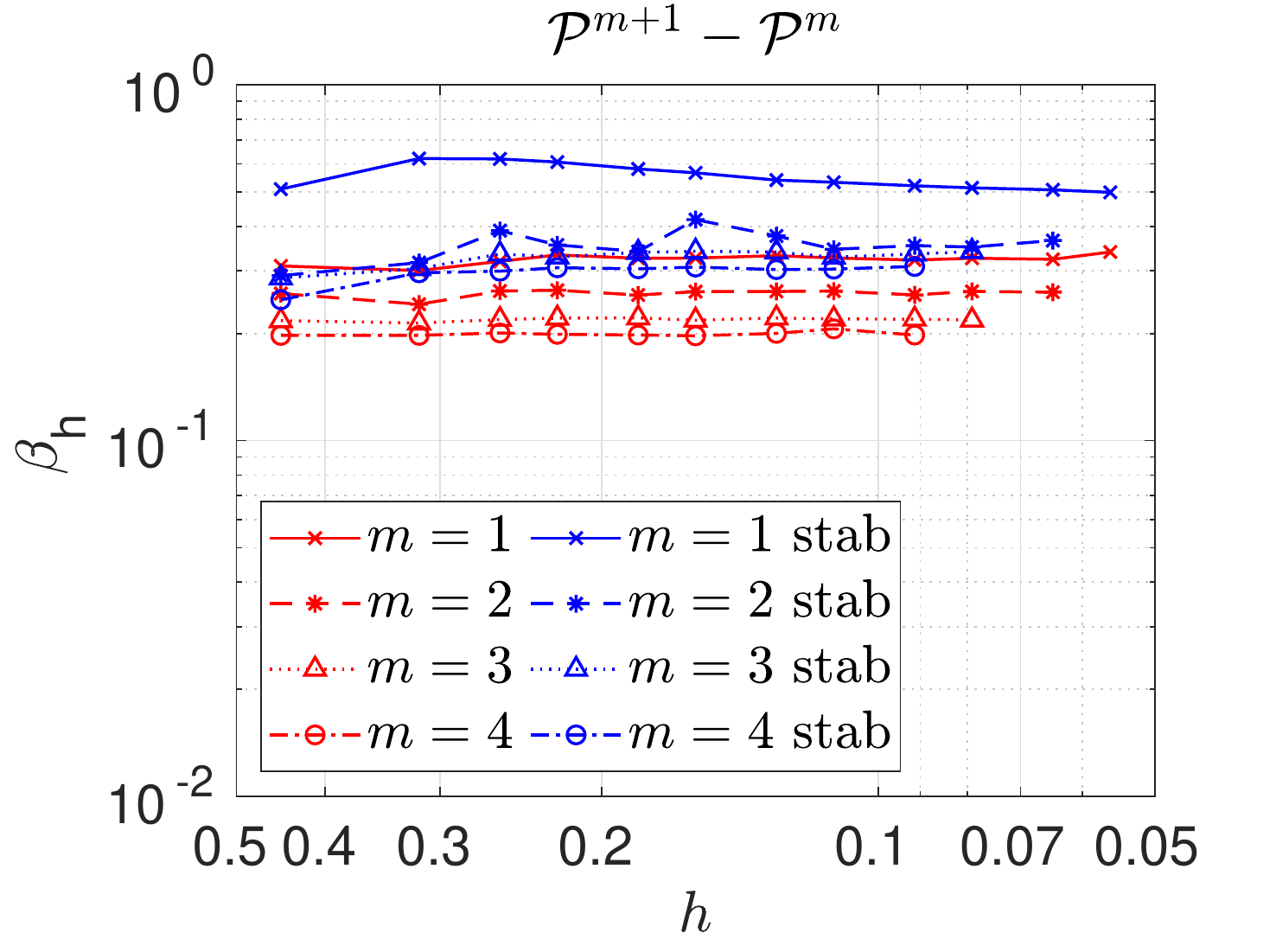}
\hspace{-0.48cm}
\includegraphics[width=0.343\textwidth]{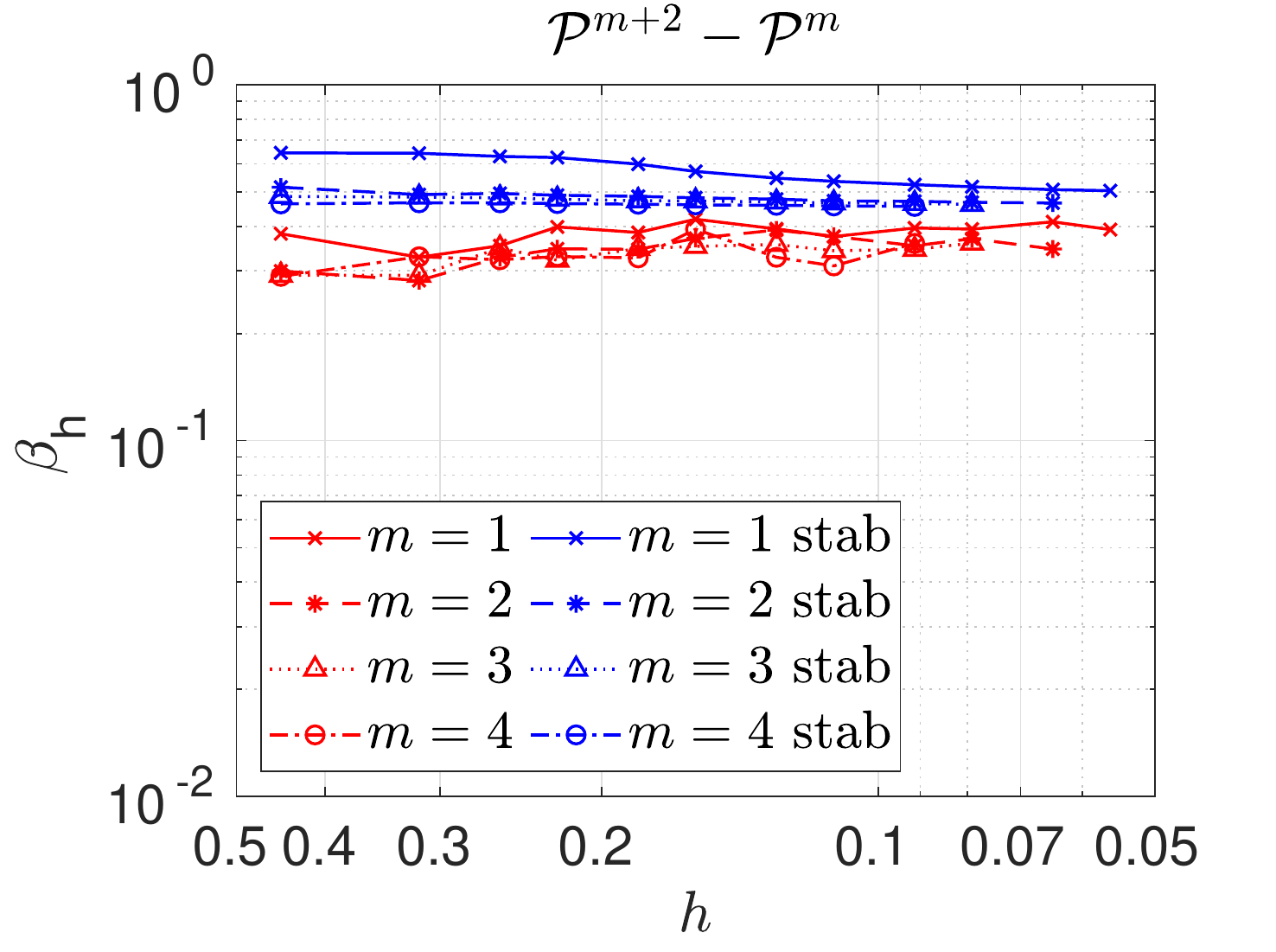}
\caption{Regular polygonal meshes}
\label{fig:infsup_h_polyreg}
\end{subfigure}
%%%%%%%%%%%%%%%%%
\begin{subfigure}{1.0\textwidth}
\centering
\includegraphics[width=0.343\textwidth]{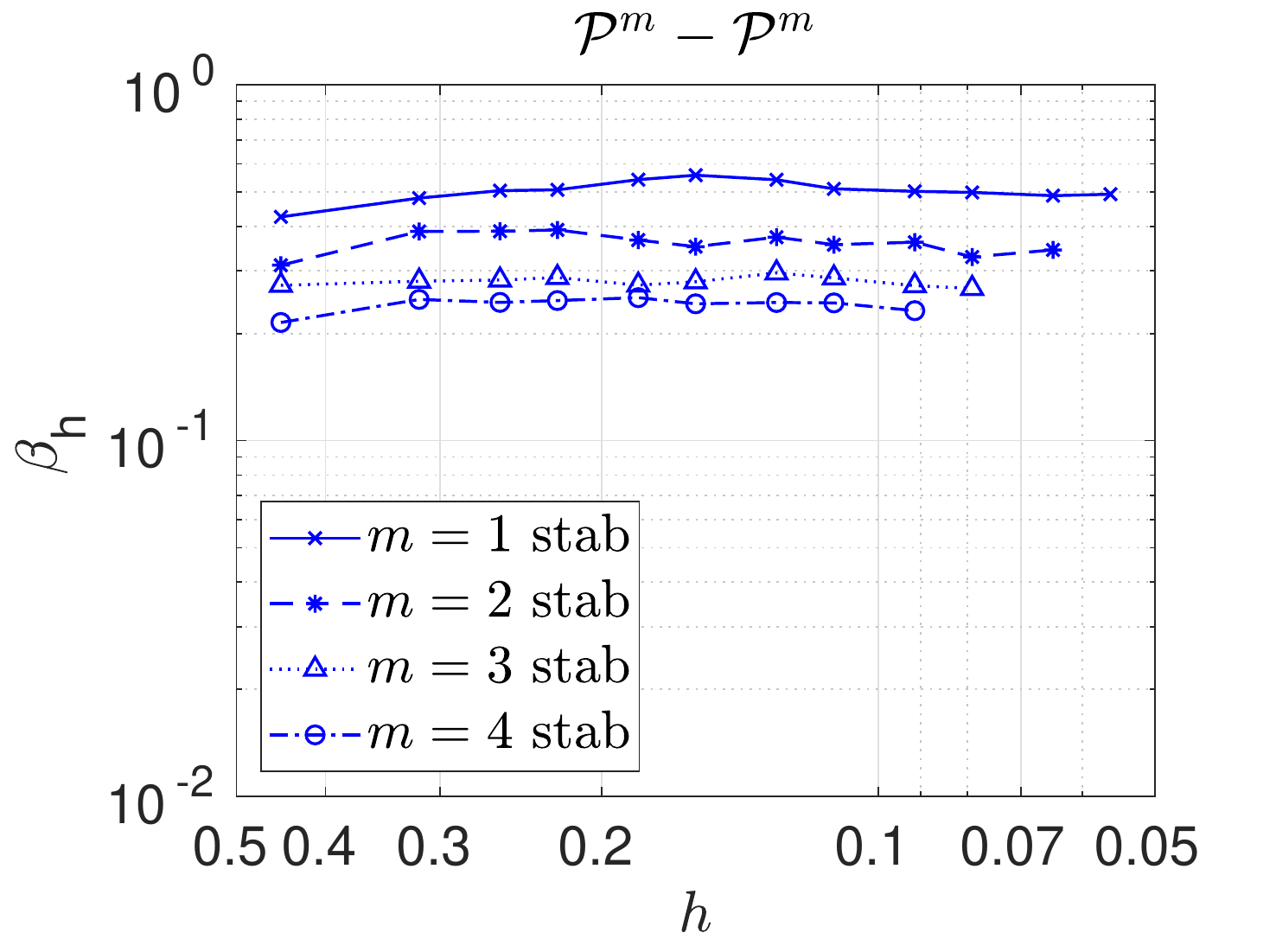}
\hspace{-0.48cm}
\includegraphics[width=0.343\textwidth]{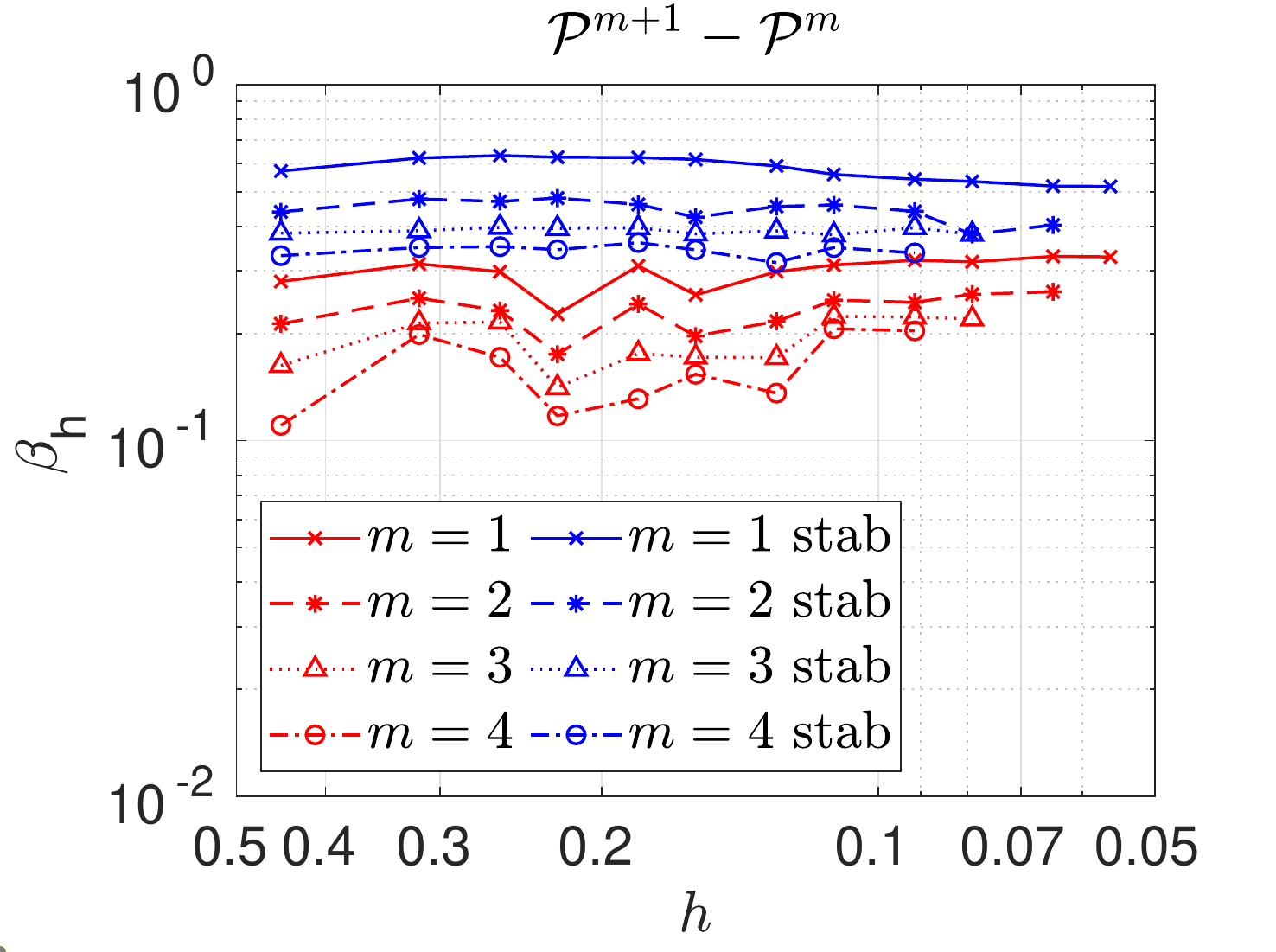}
\hspace{-0.48cm}
\includegraphics[width=0.343\textwidth]{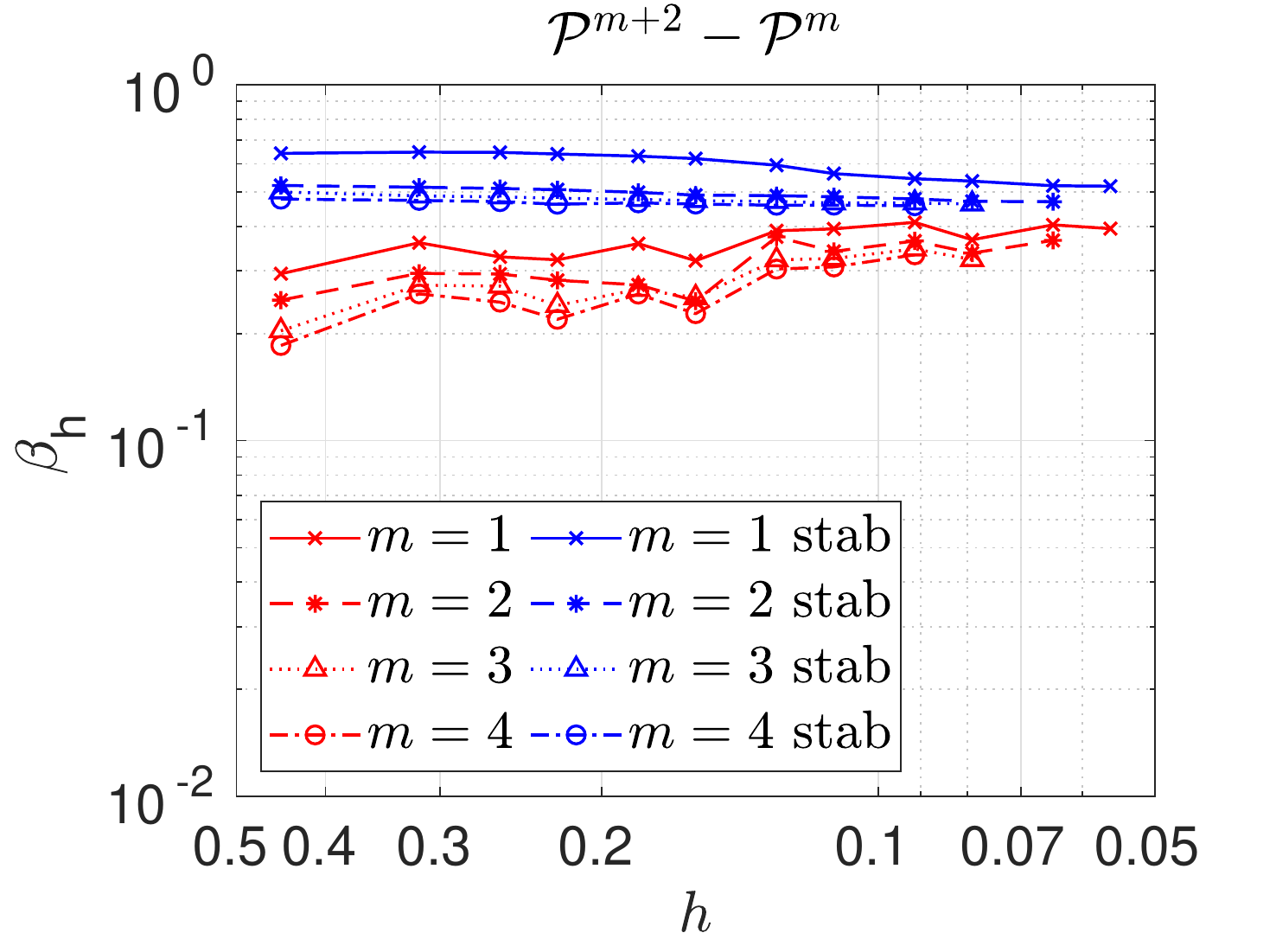}
\caption{Distorted polygonal meshes}
\label{fig:infsup_h_polybad}
\end{subfigure}
%%%%%%%%%%%%%%%%%
\begin{subfigure}{1.0\textwidth}
\centering
\includegraphics[width=0.343\textwidth]{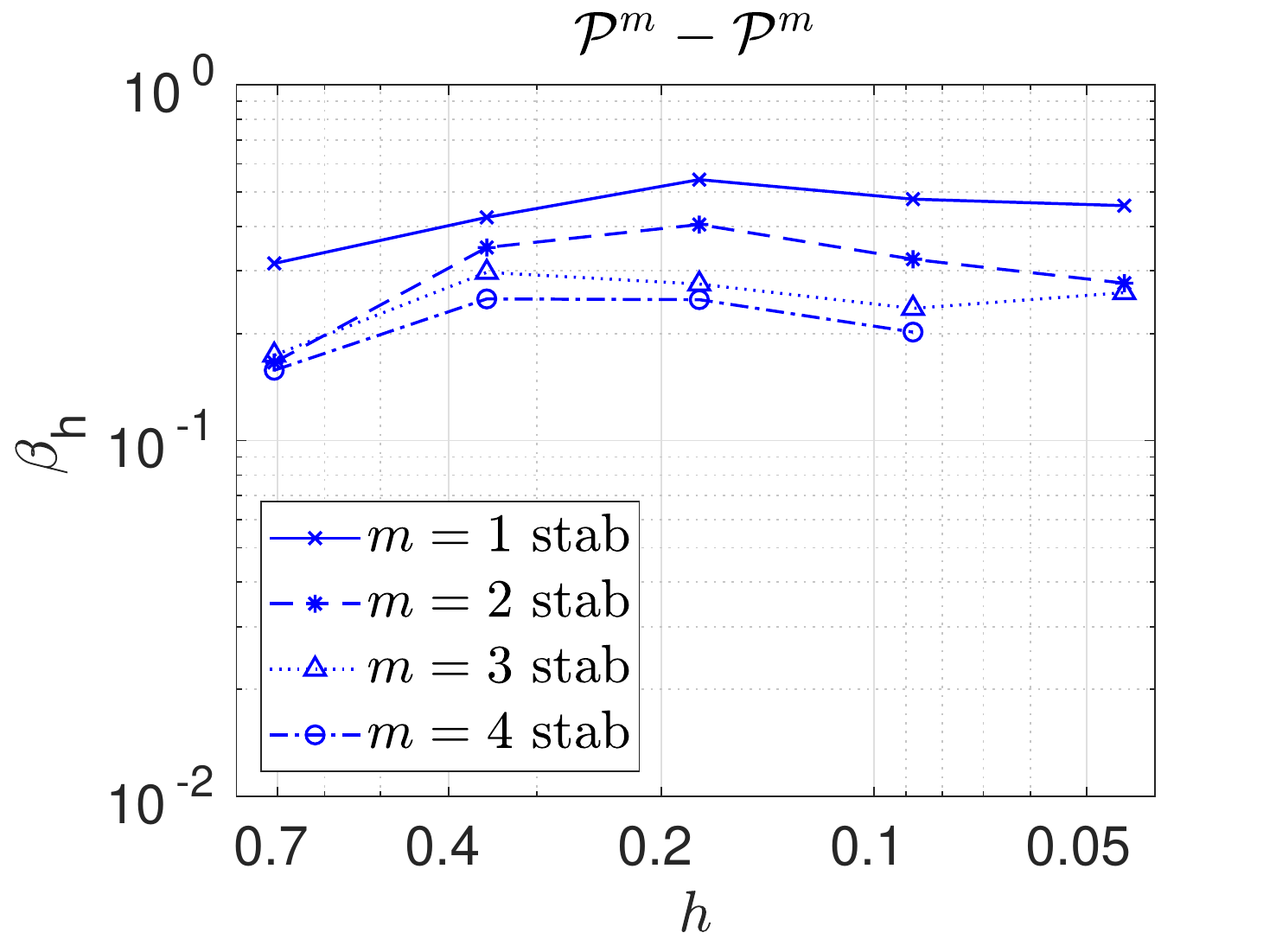}
\hspace{-0.48cm}
\includegraphics[width=0.343\textwidth]{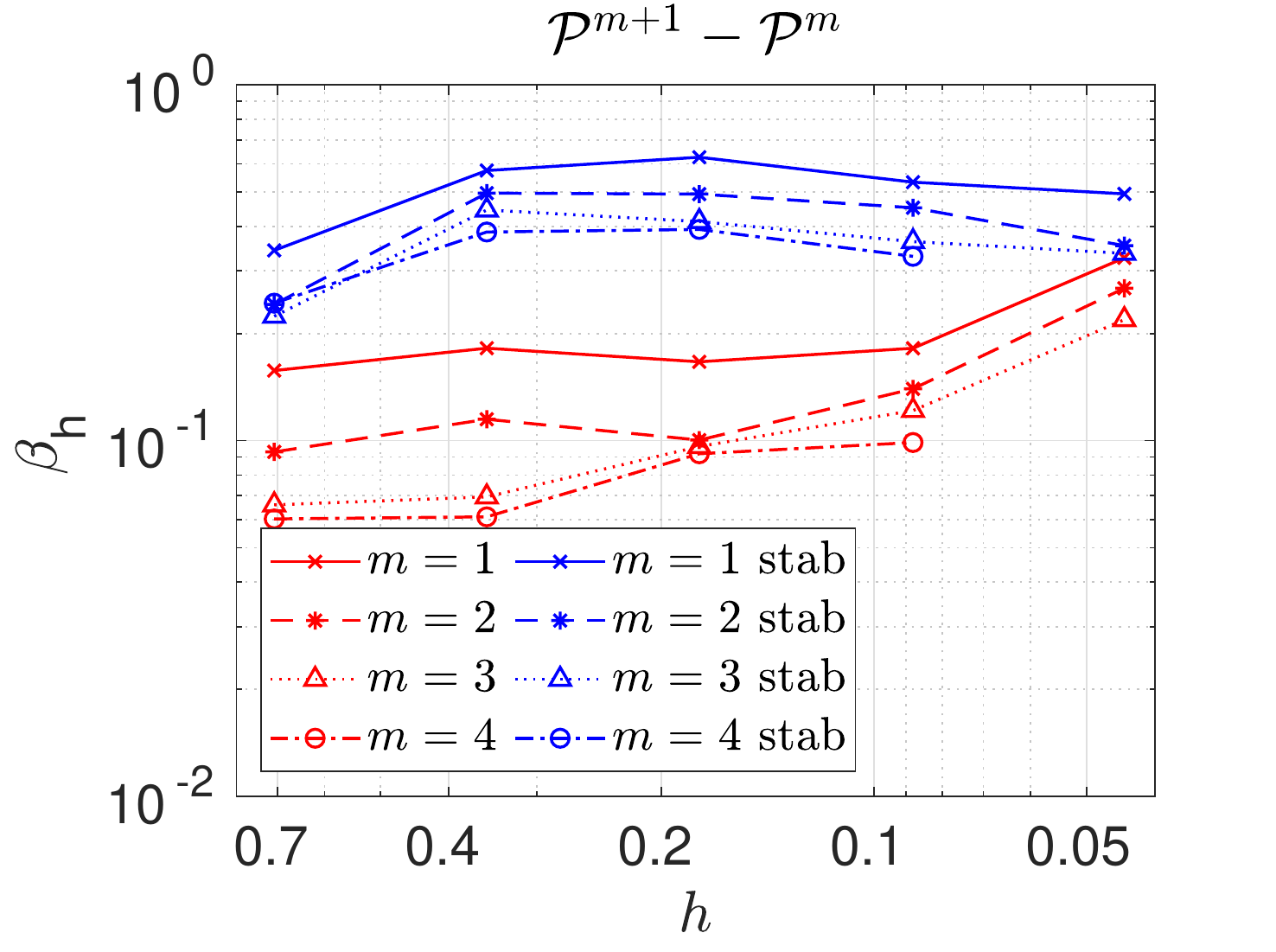}
\hspace{-0.48cm}
\includegraphics[width=0.343\textwidth]{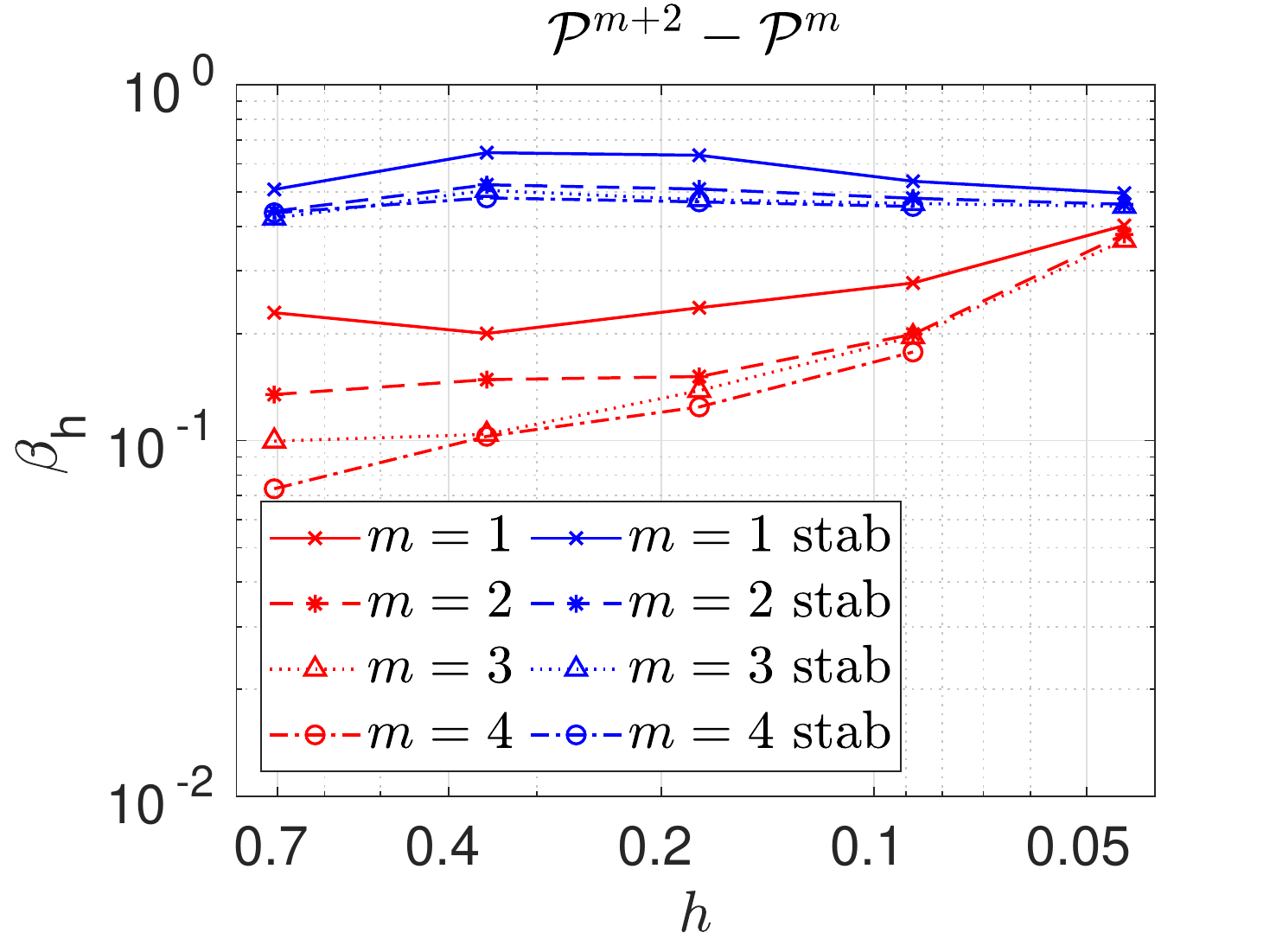}
\caption{Agglomerated polygonal meshes}
\label{fig:infsup_h_polyMGG}
\end{subfigure}
%%%%%%%%%%%%%%%%%
\caption{Values of~$\beta_h$ as function of the mesh size~$h$ for different
choices of the polynomial degree for the discrete velocity and pressure spaces
$\mathcal P^{m+k} - \mathcal P^m$, computed solving the generalized eigenvalue
problem~\eqref{eq:GEP}.  From left to right: $\mathcal P^{m} - \mathcal P^m$,
$\mathcal P^{m+1} - \mathcal P^m$, $\mathcal P^{m+2} - \mathcal P^m$.  The
parameter $\eta=1$ for $k=0,1,2$ (blue lines), and  $\eta=0$ for $k=1,2$ (red
lines).}
\label{fig:infsup_h_all}
\end{figure}
%%%%%%%%%%%%%%%%%%%%%%%%%%%%%%%%%%%%%%%%%%%%%%%%%%%%%%%%%%%%%%%%%%%%

%%%%%%%%%%%%%%%%%%%%%%%%%%%%%%%%%%%%%%%%%%%%%%%%%%%%%%%%%%%%%%%%%%%%
\begin{figure}[!htbp]
\centering
\begin{subfigure}{1.0\textwidth}
\centering
\includegraphics[width=0.343\textwidth]{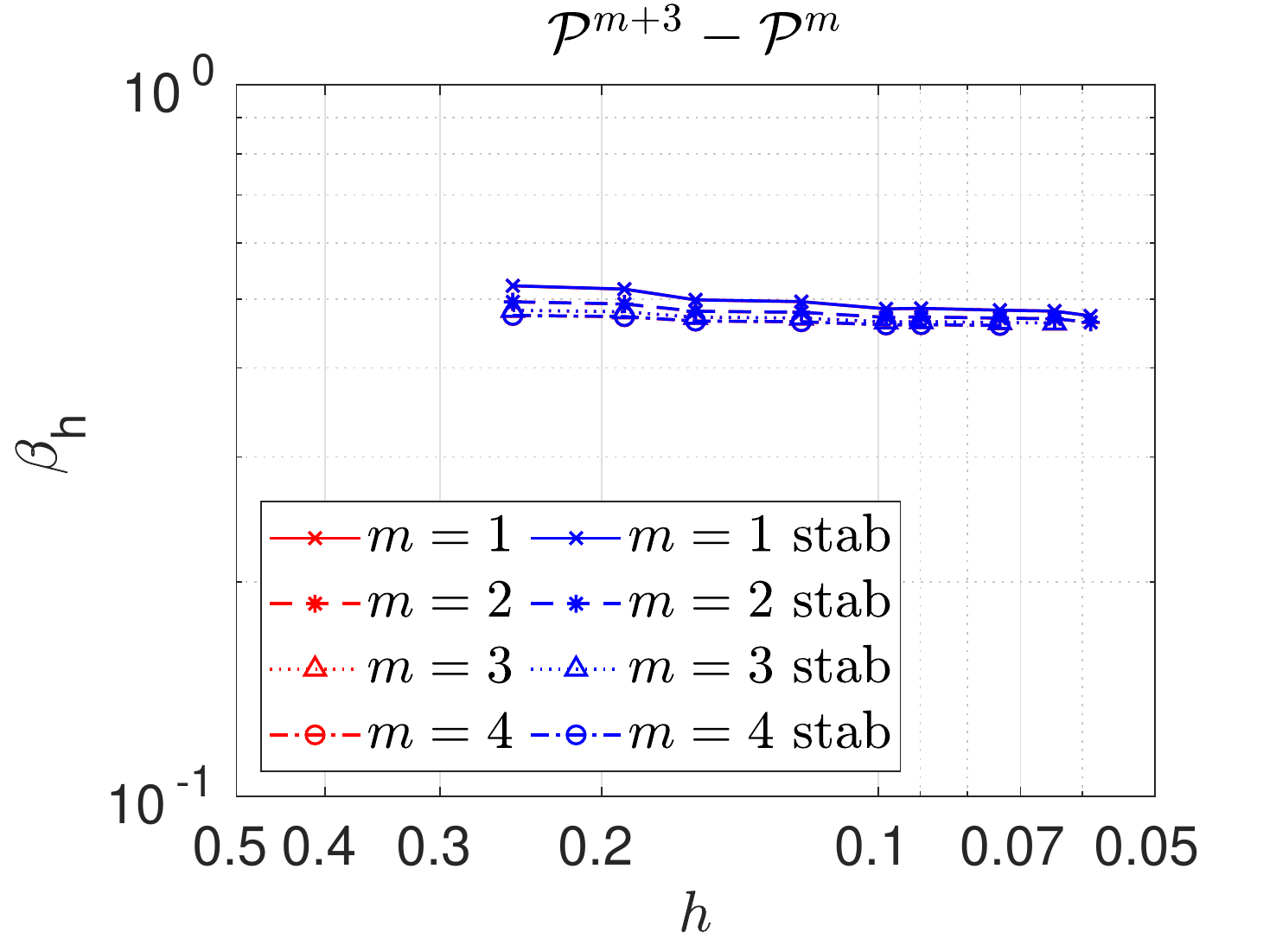}
\includegraphics[width=0.343\textwidth]{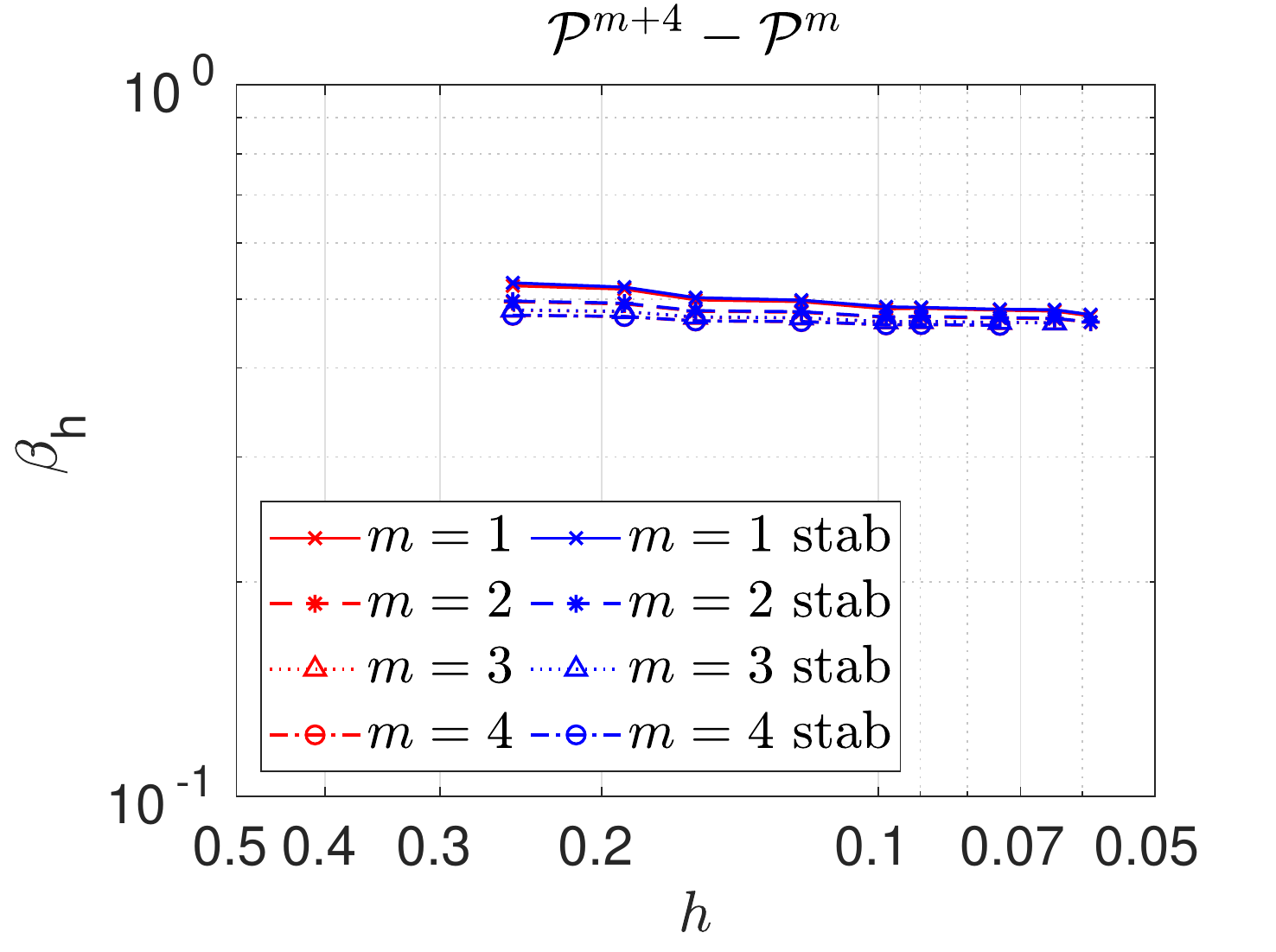}
\caption{Triangular grids}
\label{fig:infsup_h2_tria}
\end{subfigure}
%%%%%%%%%%%%%%%%%%%%
\begin{subfigure}{1.0\textwidth}
\centering
\includegraphics[width=0.343\textwidth]{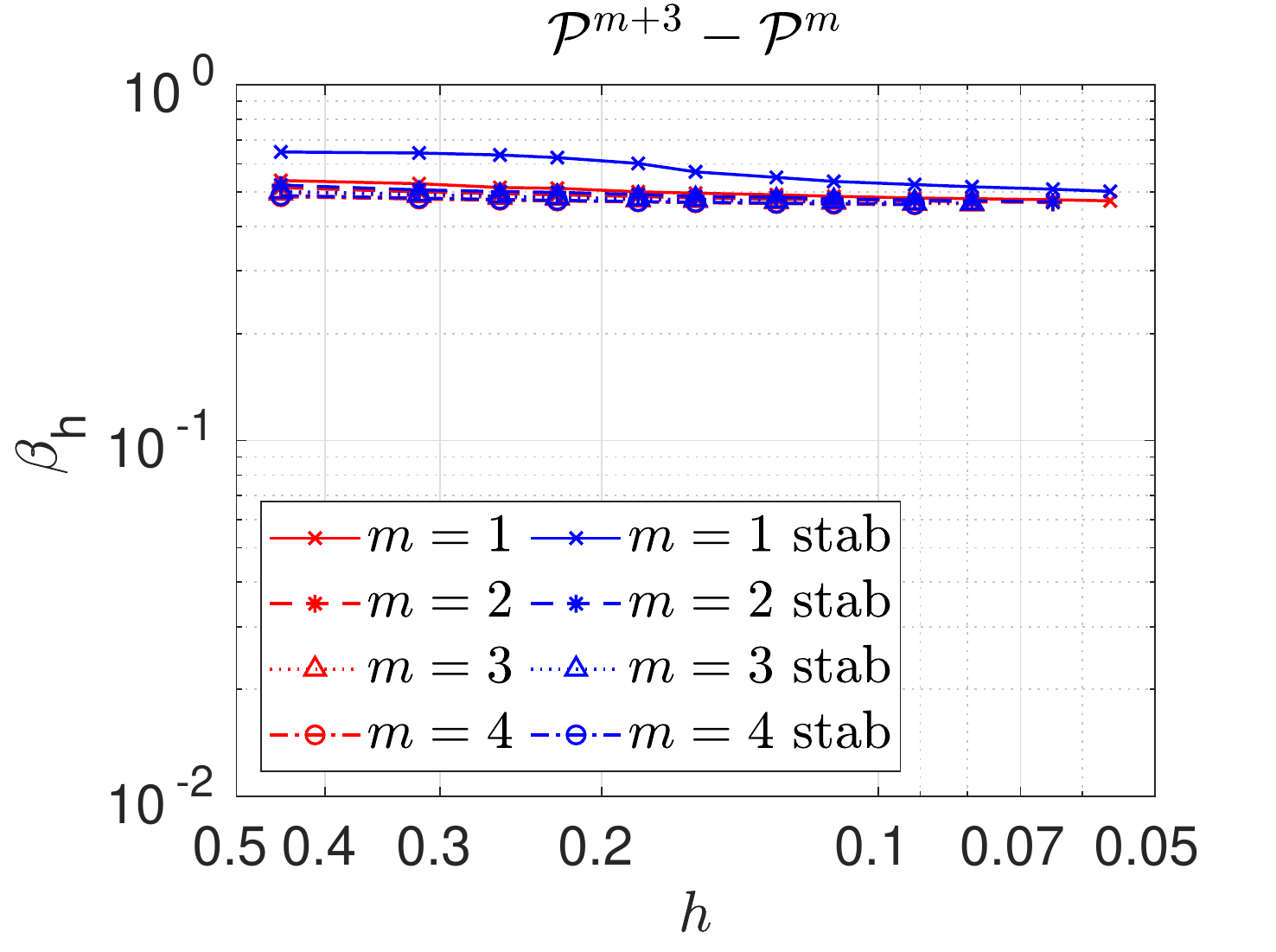}
\includegraphics[width=0.343\textwidth]{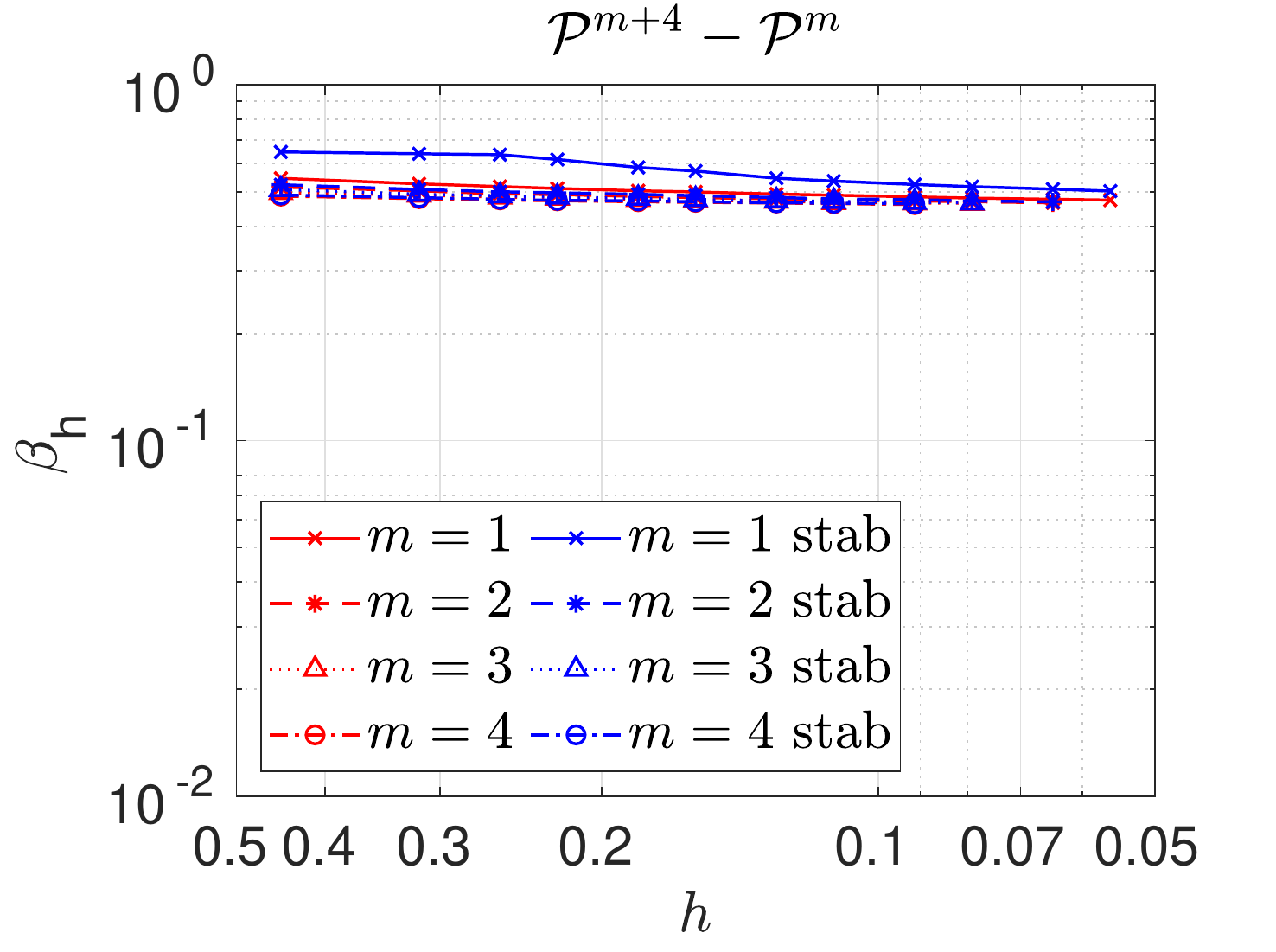}
\caption{Regular polygonal meshes}
\label{fig:infsup_h2_polyreg}
\end{subfigure}
%%%%%%%%%%%%%%%%%
\begin{subfigure}{1.0\textwidth}
\centering
\includegraphics[width=0.343\textwidth]{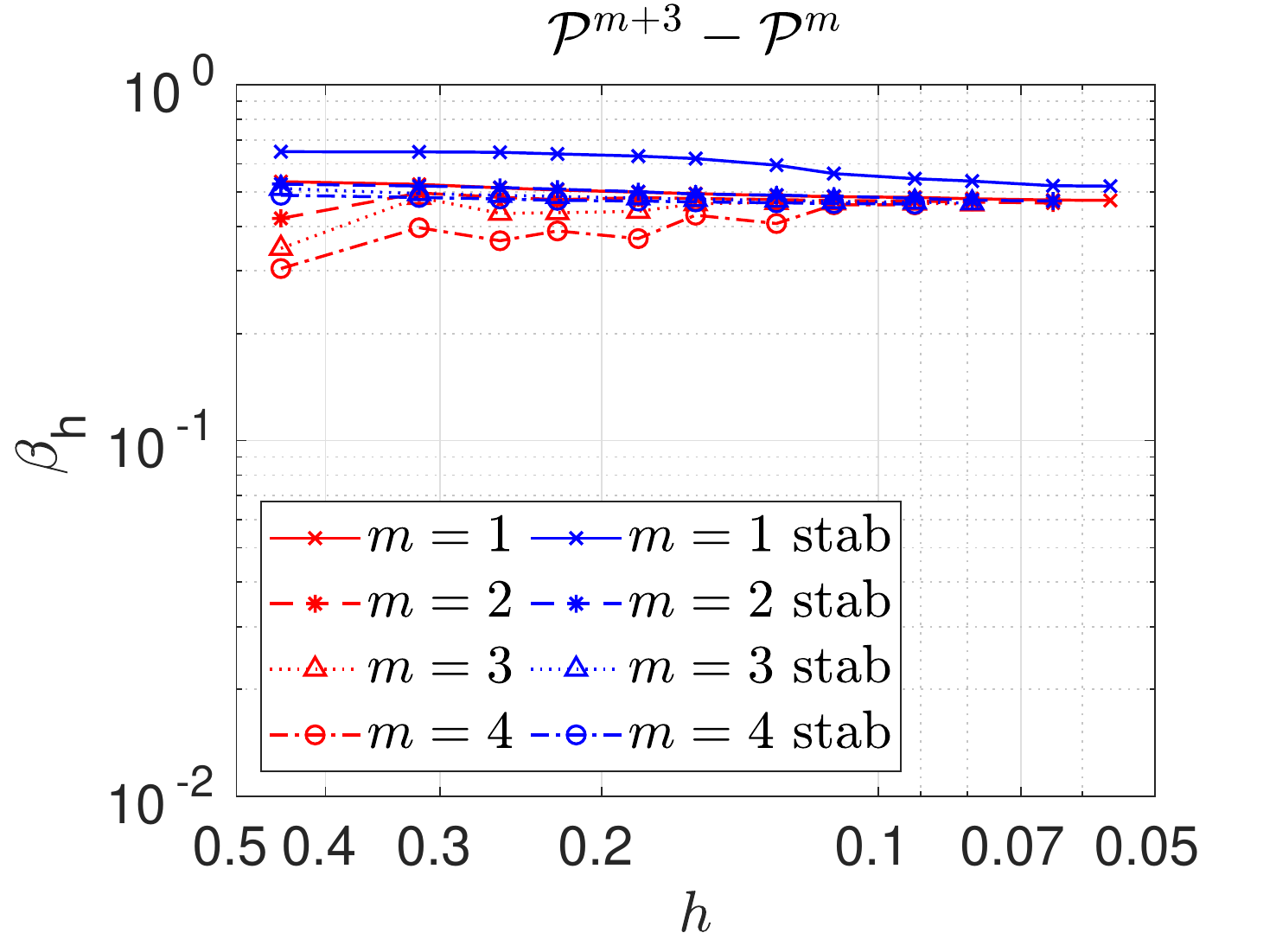}
\includegraphics[width=0.343\textwidth]{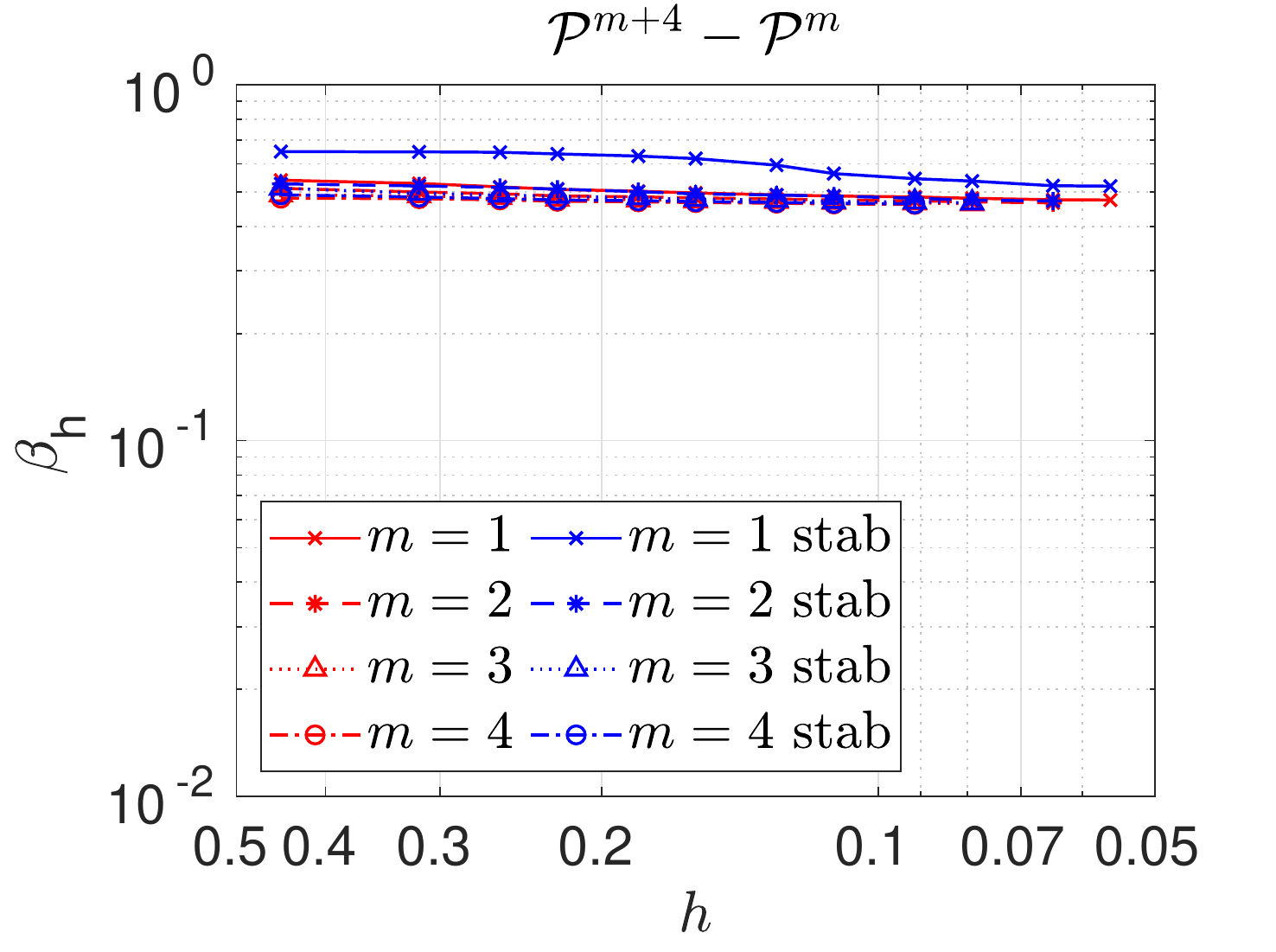}
\caption{Distorted polygonal meshes}
\label{fig:infsup_h2_polybad}
\end{subfigure}
%%%%%%%%%%%%%%%%%
\begin{subfigure}{1.0\textwidth}
\centering
\includegraphics[width=0.343\textwidth]{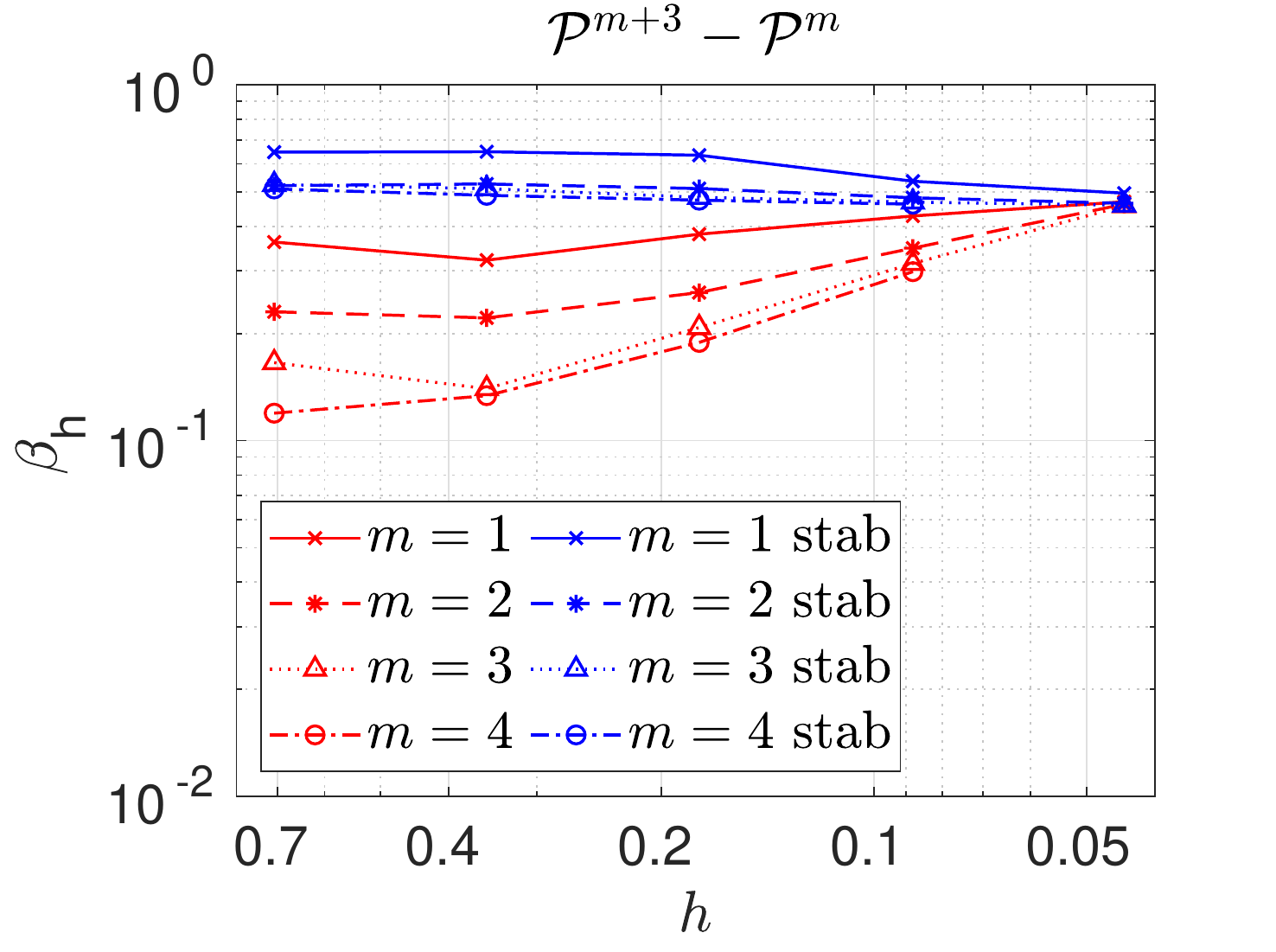}
\includegraphics[width=0.343\textwidth]{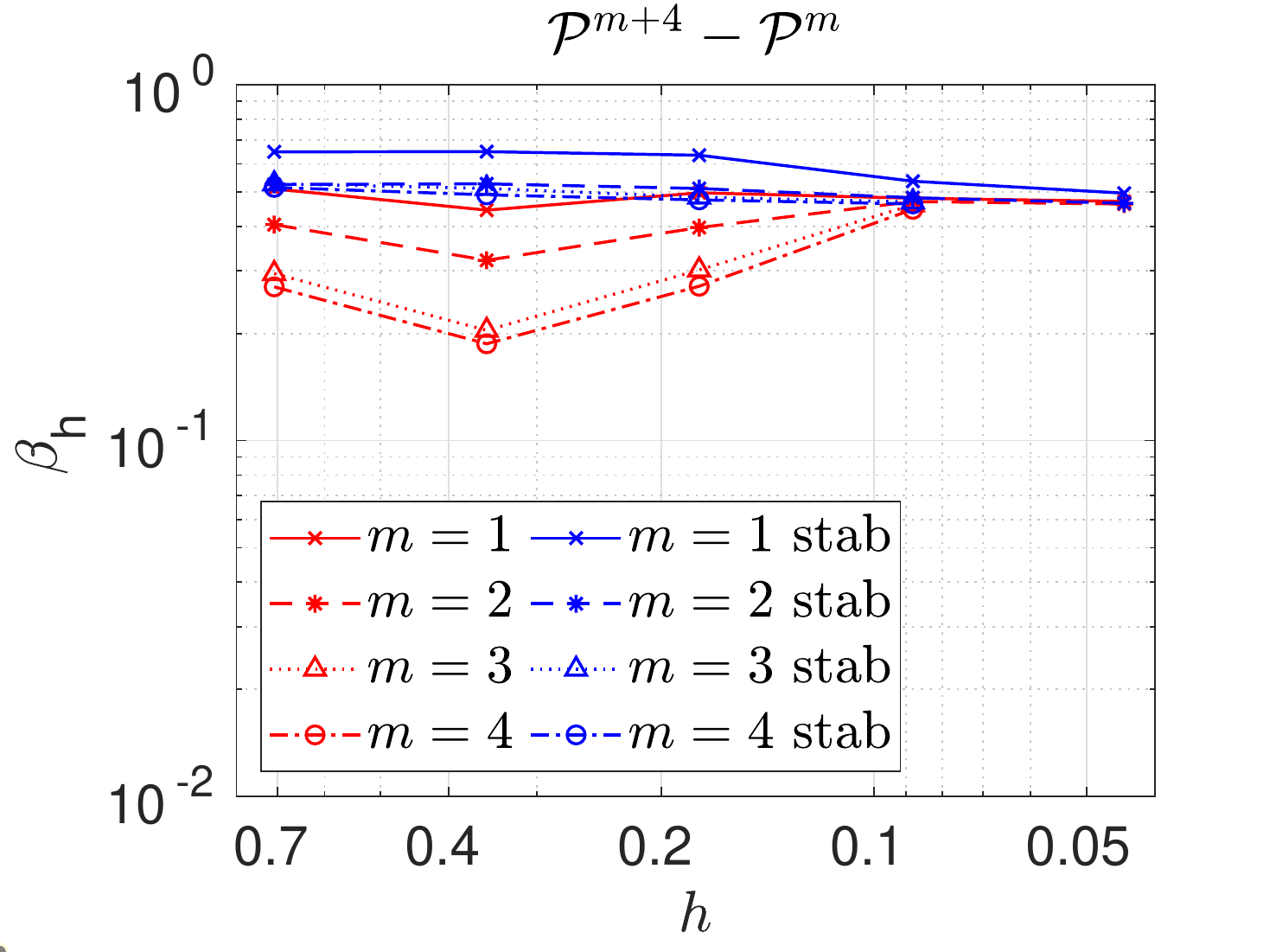}
\caption{Agglomerated polygonal meshes}
\label{fig:infsup_h2_polyMGG}
\end{subfigure}
%%%%%%%%%%%%%%%%%
\caption{Values of~$\beta_h$ as function of the mesh size~$h$ for different
choices of the polynomial degree for the discrete velocity and pressure spaces
$\mathcal P^{m+k} - \mathcal P^m$, computed solving the generalized eigenvalue
problem~\eqref{eq:GEP}.  From left to right: $\mathcal P^{m+3} - \mathcal P^m$,
$\mathcal P^{m+4} - \mathcal P^m$.  The parameter $\eta =1$ (blue lines) and
$\eta=0$ (red lines).}
\label{fig:infsup_h2_all}
\end{figure}
%%%%%%%%%%%%%%%%%%%%%%%%%%%%%%%%%%%%%%%%%%%%%%%%%%%%%%%%%%%%%%%%%%%%

%POLYNOMIAL DEGREE
Next, we investigate the behavior of~$\beta_h$ by varying the polynomial approximation orders for the velocity and the pressure spaces, and fixing the computational mesh.
In Figure \ref{fig:infsup_p_tria_all}, we report the computed value of $\beta_h$  as a function of the polynomial approximation degree~$m$ for different choices of
the velocity and pressure spaces  $\mathcal P^{m+k} - \mathcal P^m$, $k=0,1,2,3,4$.
We set the parameter $\eta=1$ for $k=0,1,2$, and~$\eta=0$ for $k=1,2,3,4$.
When~$\eta=1$, we obtain the following results.
For $k=0$, i.e., $\ell=m$, the dependence of $\beta_h$ is in agreement with Proposition~\ref{prop:gen_inf_sup}: the constant~$\beta_h$ deteriorates as~$m$ grows.
Nevertheless, the estimate in Proposition~\ref{prop:gen_inf_sup} is slightly suboptimal by a factor of~$m^{-1/2}$, as our numerical computations suggest that $\beta_h= O(m^{-1/2})$ for all the considered mesh configurations.
For $k=1,2$, on triangular and regular polygonal meshes, $\beta_h$ looks independent of~$m$, while mildly depends on~$m$ on irregular and agglomerated polygonal meshes.
From the numerical computations obtained in the no pressure stabilization cases $\eta=0$, we draw the following conclusions:
\emph{(i)} on regular polygonal meshes, $\beta_h$ is independent of~$m$ for all the considered velocity-pressure pairs $\mathcal P^{m+k} - \mathcal P^m$, $k=1,2,3,4$;
\emph{(ii)} on irregular and agglomerated polygonal meshes, the behavior of $\beta_h$ is less clear and we detect a mild dependence on~$m$.
By comparing the cases $\mathcal P^{m+k} - \mathcal P^m$, $k=1,2$ with and without pressure stabilization,
at least for the agglomerated polygonal meshes, the dependence of $\beta_h$ on $m$ is milder for the case~$\eta=1$.

%\textcolor{red}{finire di scrivere conclusioni a valle dei nuovi risultati.
%Stefano, conrolla in letteratura cosa si veder sui triangoli. I nostri
%risultati sono consistenti??}\\
%%%%%%%%%%%%%%%%%%%%%%%%%%%%%%%%%%%%%%%%%%%%%%%%%%%%%%%%%%%%%%%%%%%%
\begin{figure}[!htbp]
\centering
\begin{subfigure}{0.495\textwidth}
\centering
\includegraphics[width=\textwidth]{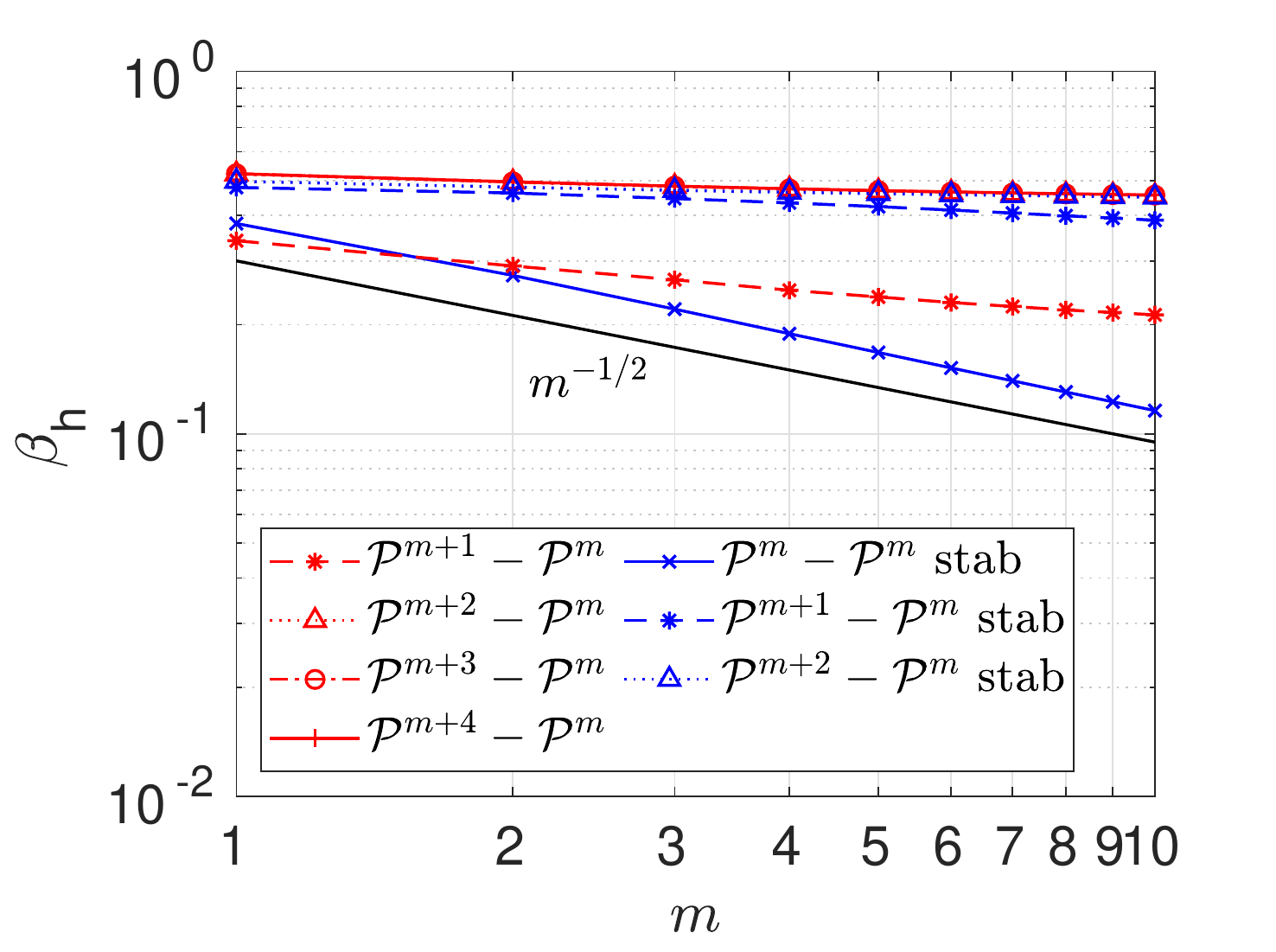}
\caption{Triangular grids ($N_{el} = 16$)}
\label{fig:infsup_p_tria}
\end{subfigure}
%%%%%%%%%%%%%%%%%%%%%
\begin{subfigure}{0.495\textwidth}
\centering
\includegraphics[width=\textwidth]{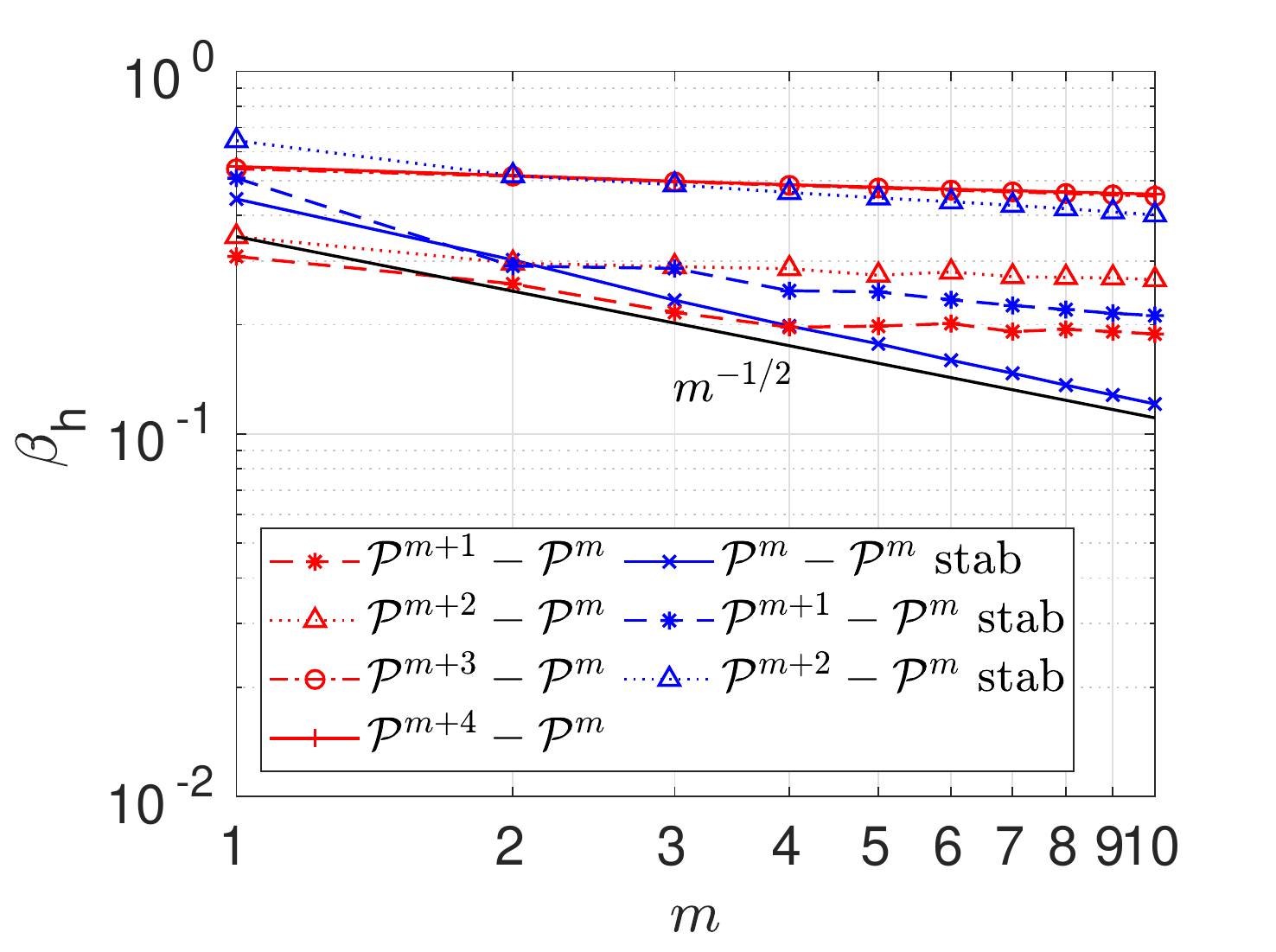}
\caption{Regular polygonal meshes ($N_{el} = 5$)}
\label{fig:infsup_p_polyreg}
\end{subfigure}
%%%%%%%%%%%%%%%%%%
\begin{subfigure}{0.495\textwidth}
\centering
\includegraphics[width=\textwidth]{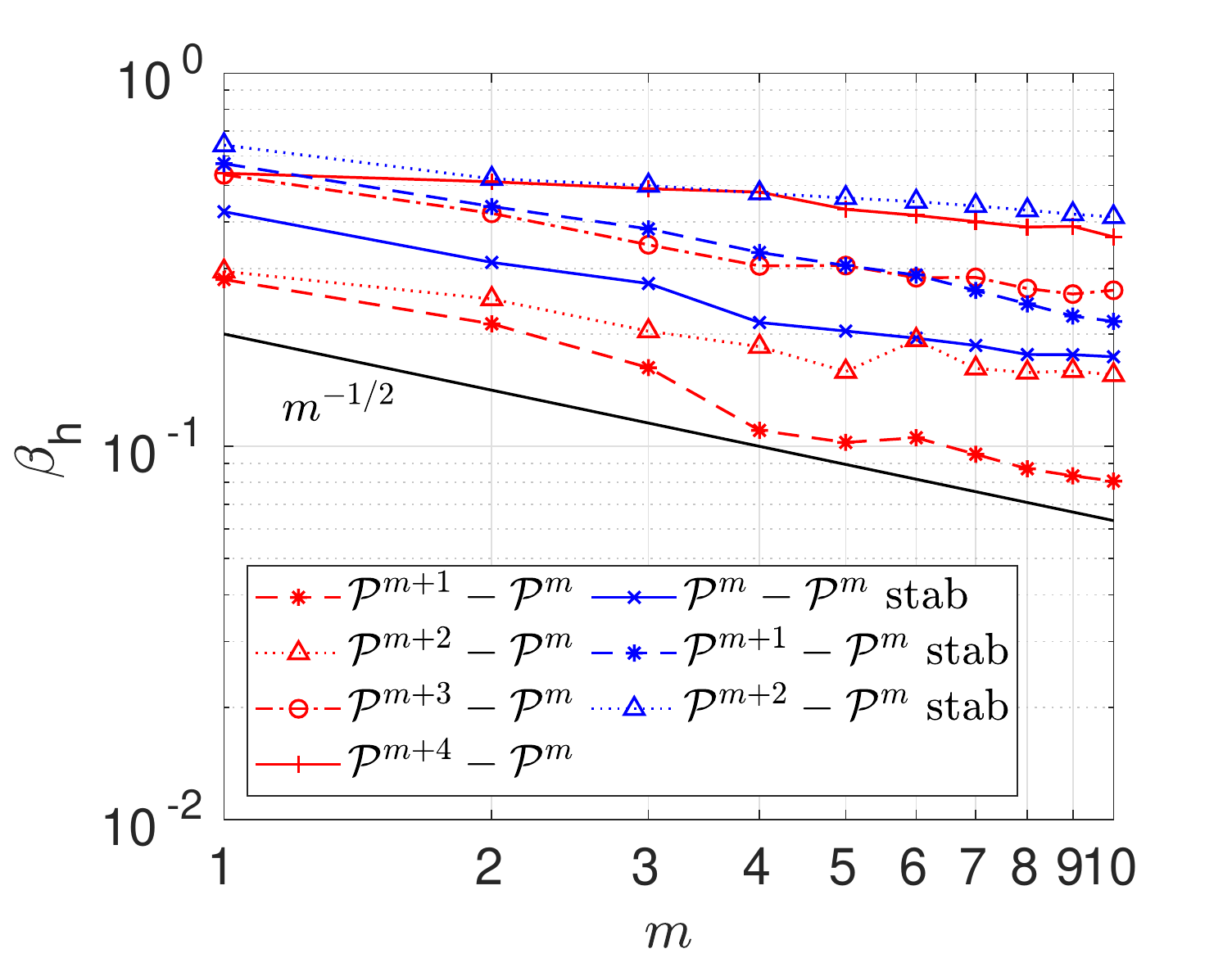}
\caption{Distorted polygonal meshes ($N_{el} = 5$)}
\label{fig:infsup_p_polybad}
\end{subfigure}
%%%%%%%%%%%%%%%%%%
\begin{subfigure}{0.495\textwidth}
\centering
\includegraphics[width=\textwidth]{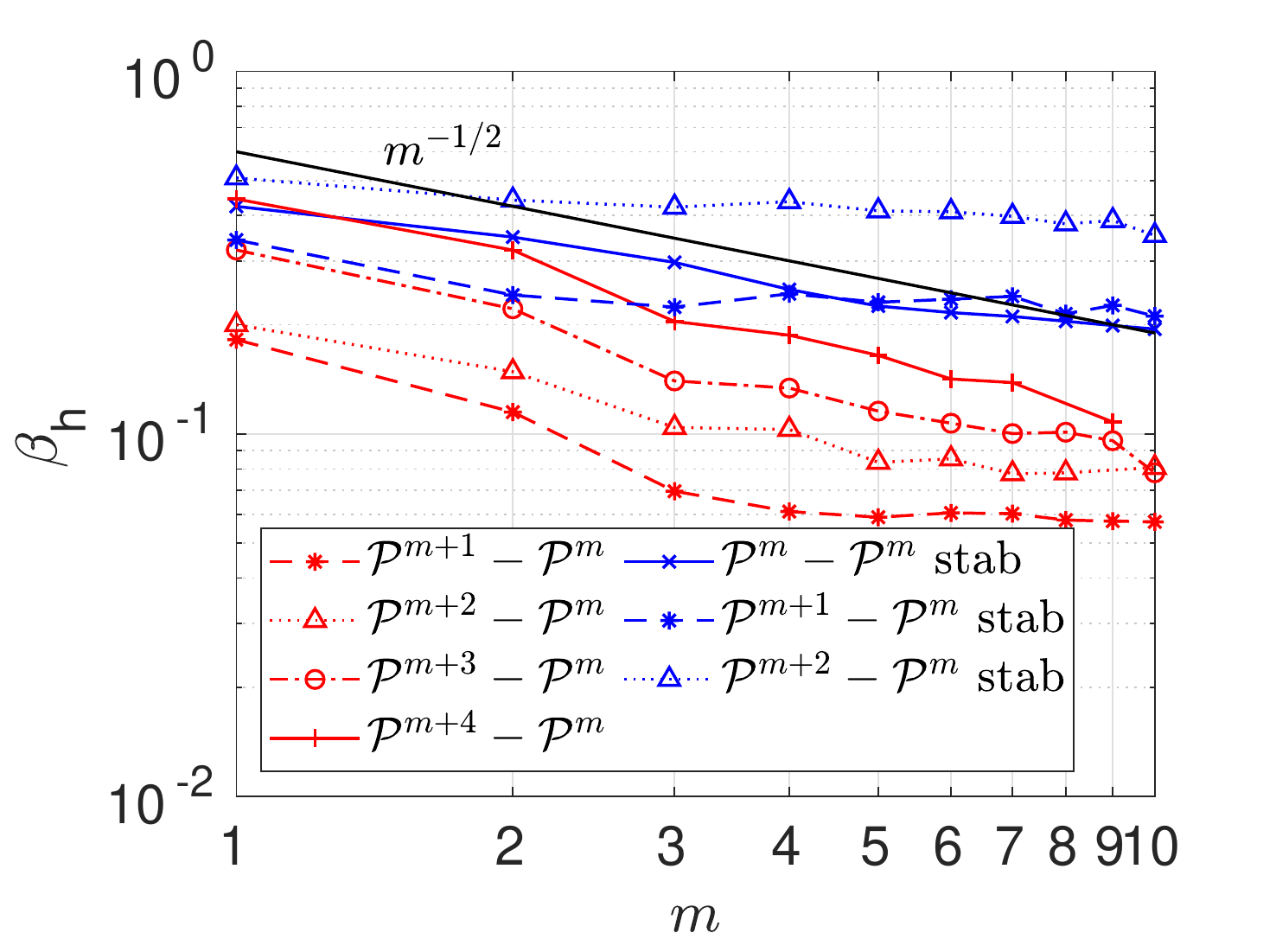}
\caption{Agglomerated polygonal meshes ($N_{el} = 8$)}
\label{fig:infsup_p_polyMGG}
\end{subfigure}
%%%%%%%%%%%%%%%%%
\caption{Values of~$\beta_h$ as function of the polynomial degree  for different
choices of the velocity and pressure spaces $\mathcal P^{m+k} - \mathcal P^m$,
$k=0,1,2,3,4$ computed solving the generalized eigenvalue
problem~\eqref{eq:GEP}.  The parameter $\eta=1$ if $k=0,1,2$ (blue lines), and
$\eta=0$ if $k=1,2,3,4$ (red lines).}
\label{fig:infsup_p_tria_all}
\end{figure}
%%%%%%%%%%%%%%%%%%%%%%%%%%%%%%%%%%%%%%%%%%%%%%%%%%%%%%%%%%%%%%%%%%%%

% DEGENERATING EDGES

Moreover, we study the behaviour of~$\beta_h$ in the case of degenerate edges, i.e., when the number of the edges of a polygon with fixed size increases and the size of the edges tends to zero.
In particular, we consider an initial triangular mesh with a uniform mesh size; see Figure~\ref{fig:meshesDegEd} (left).
Starting from this grid, we generate a sequence of meshes by halving recursively the edges of the element at the center of the mesh,
leading to a polygon with an increasing number of edges; see Figure~\ref{fig:meshesDegEd} (middle and right).
We indicate the number of the edges of the polygon with $\# edges$.
In Figure~\ref{fig:infsup_h_deged}, we report the values of~$\beta_h$ as a function of $\# edges$ for different choices of the discrete velocity
and pressure spaces $\mathcal P^{m+k} - \mathcal P^m$, $k=0,1,2$, with ($\eta = 1$) and without ($\eta = 0$) pressure stabilization.
The computed numerical \textit{inf-sup} constant $\beta_h$ seems to be independent of the size of the edges for any choice of $k$ and the pressure stabilization term.
This indicates that Assumption~\ref{ass:star} in Proposition~\ref{prop:gen_inf_sup} may be relaxed; see Remark~\ref{rem:betah}.

%% MESH
\begin{figure}[!htbp]
    \centering
    \includegraphics[width=0.36\textwidth]{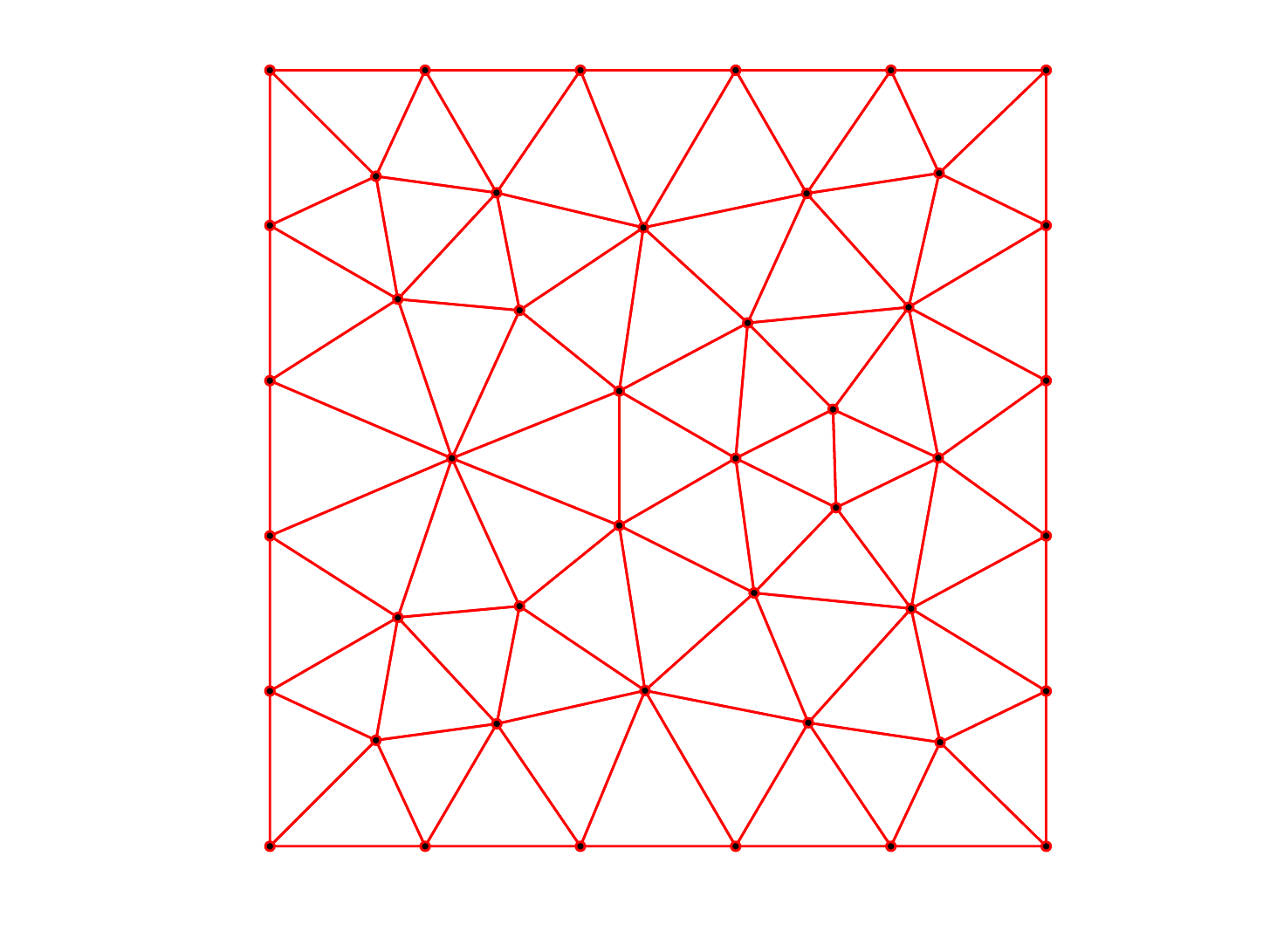}
    \hspace{-0.9cm}
    \includegraphics[width=0.36\textwidth]{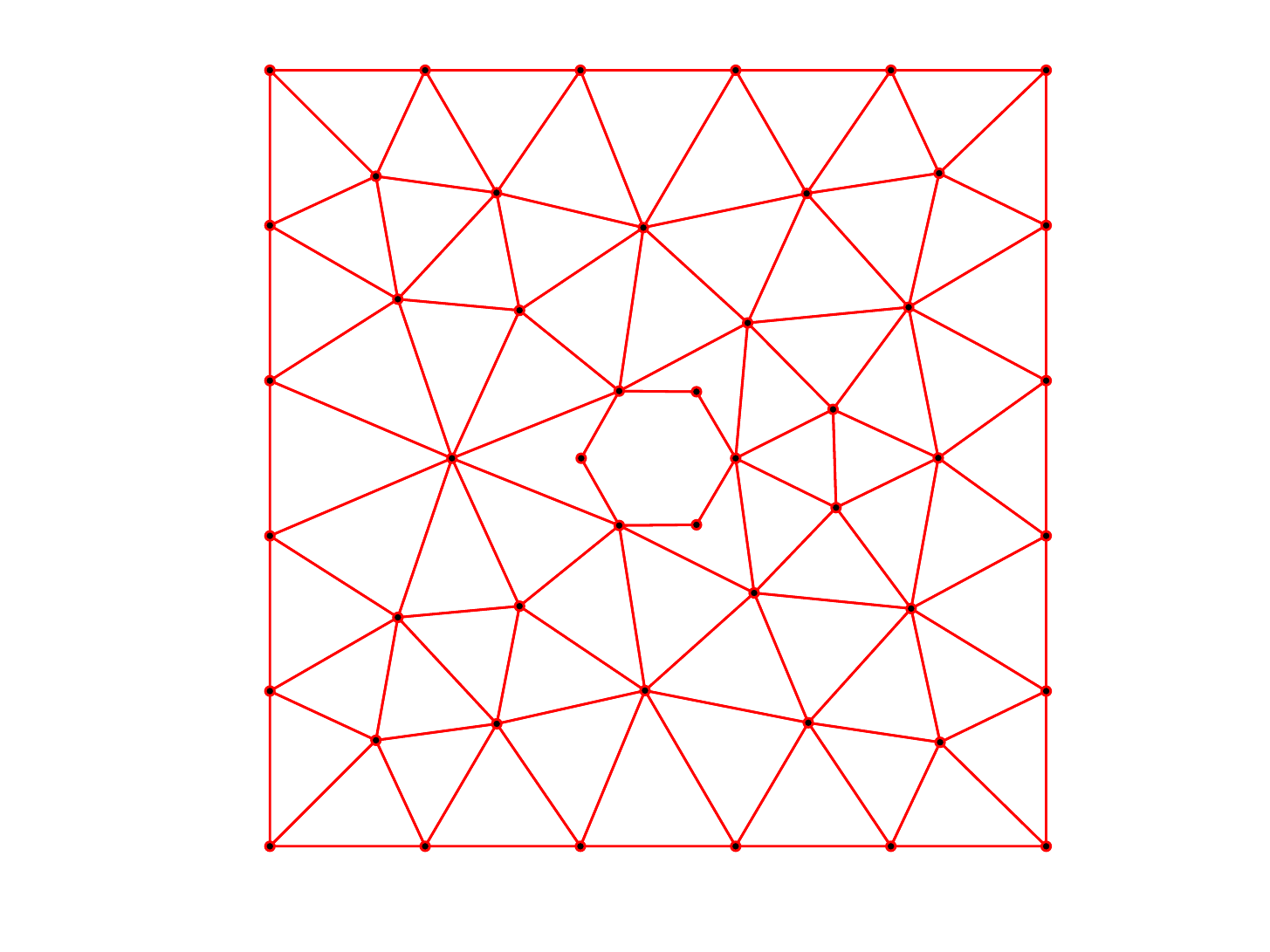}
    \hspace{-0.9cm}
    \includegraphics[width=0.36\textwidth]{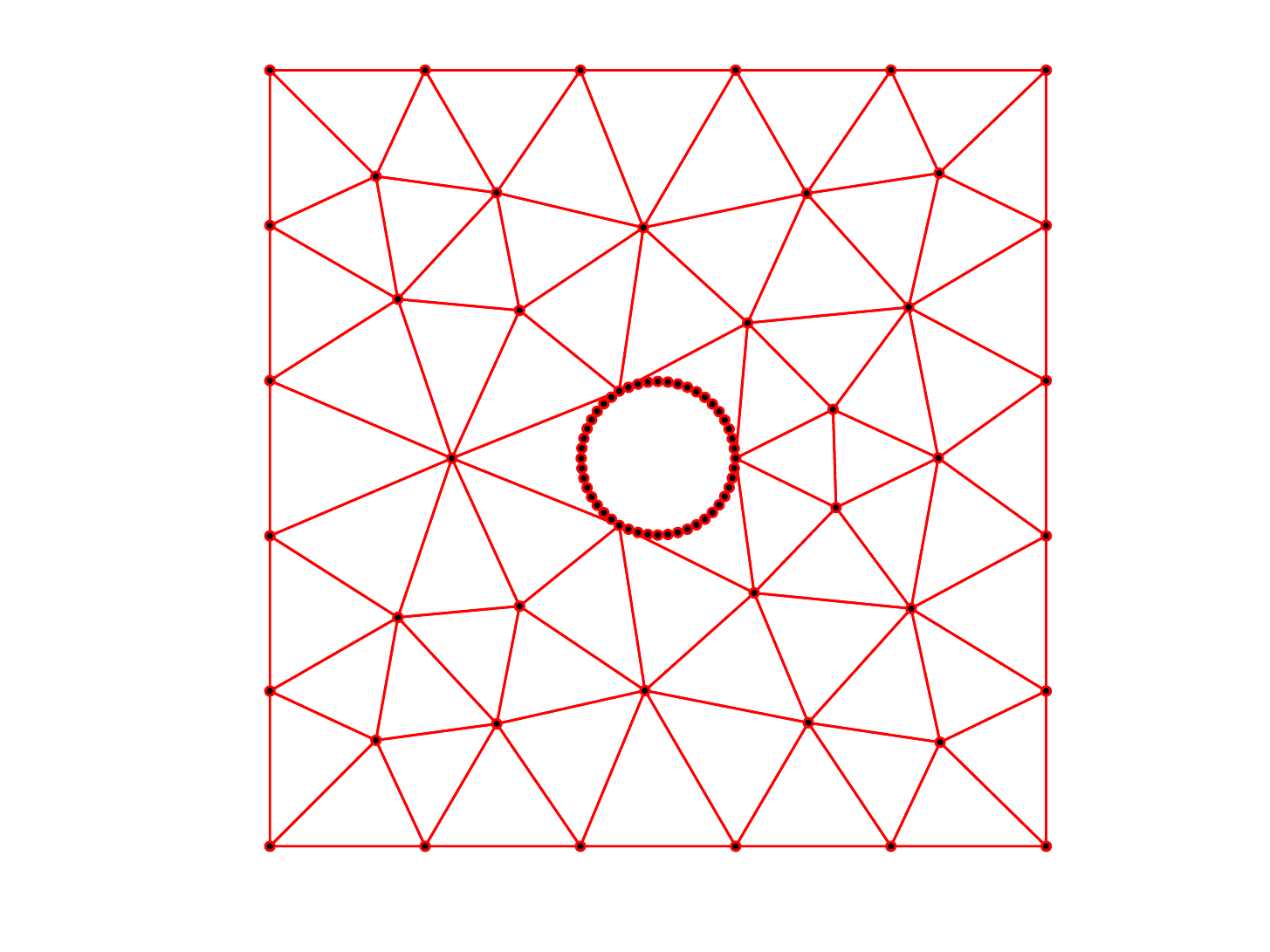}
    \caption{Meshes obtained with a recursive splitting of the edges of the
    element at the center. From left to right: initial triangular meshes with
    $\# edges = 3$; mesh at the first iteration with $\# edges = 6$; mesh at the
    fourth iteration with $\# edges = 48$.}
    \label{fig:meshesDegEd}
\end{figure}

\begin{figure}[!htbp]
    \centering
    \includegraphics[width=0.343\textwidth]{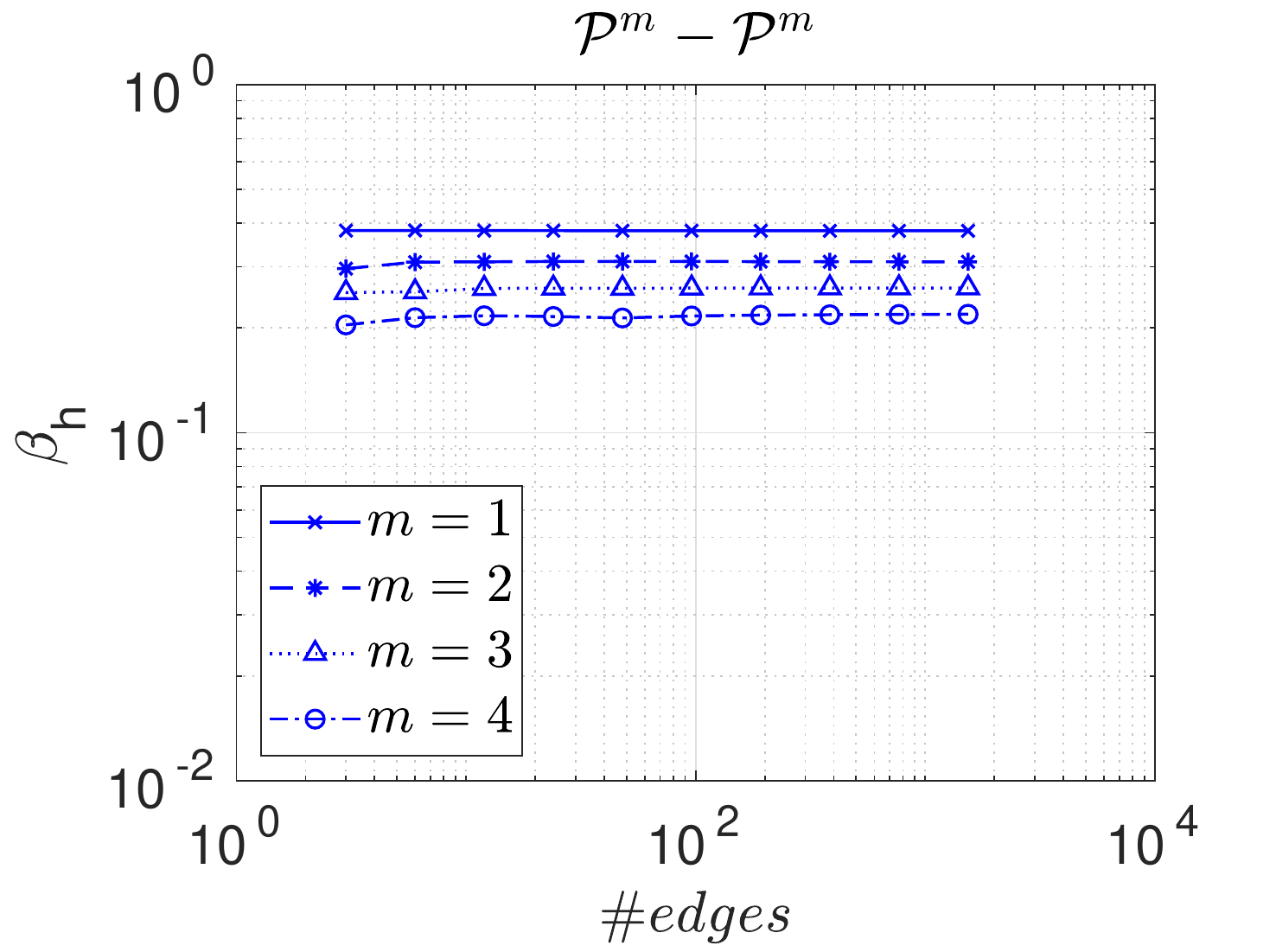}
    \hspace{-0.48cm}
    \includegraphics[width=0.343\textwidth]{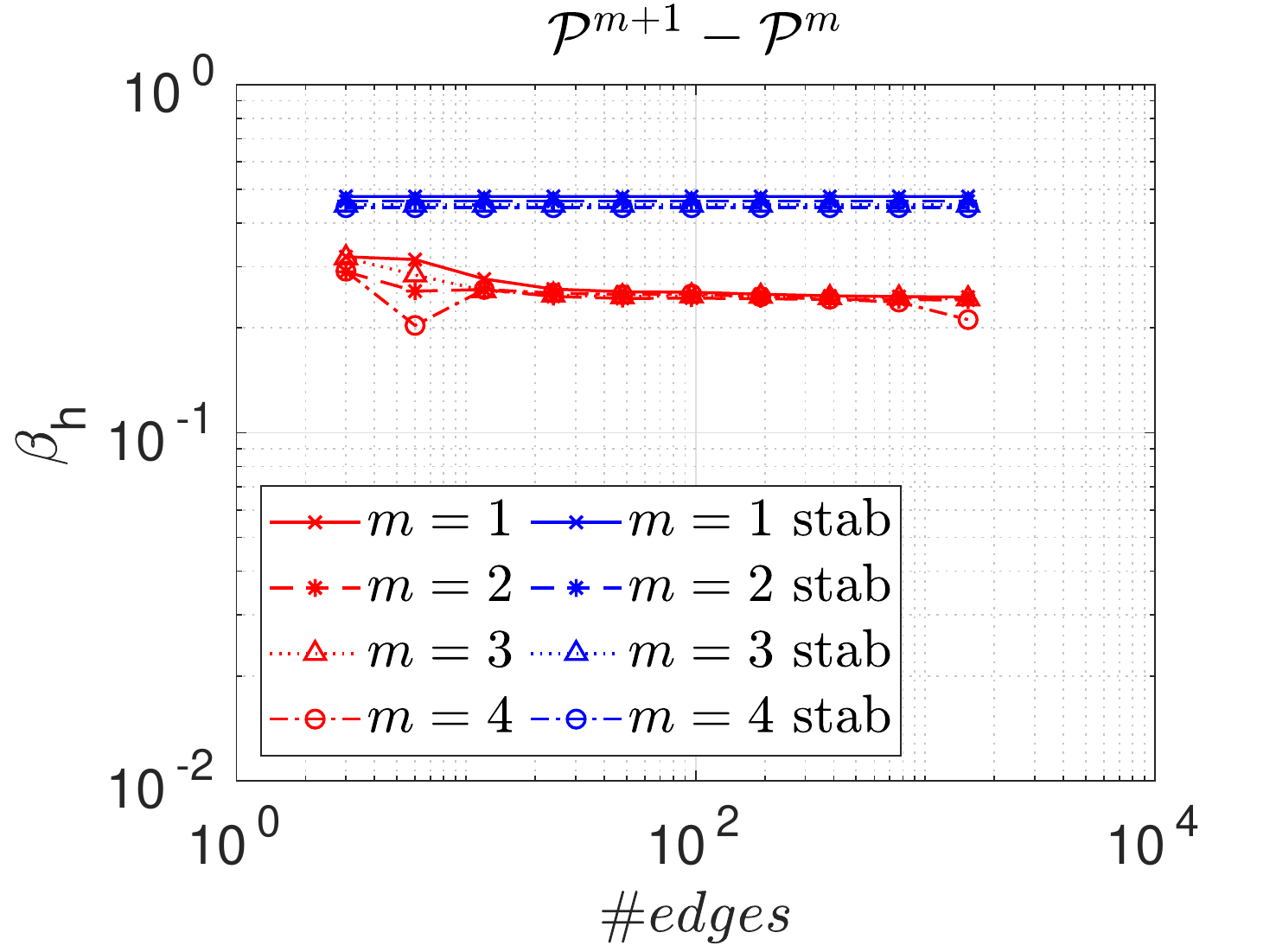}
    \hspace{-0.48cm}
    \includegraphics[width=0.343\textwidth]{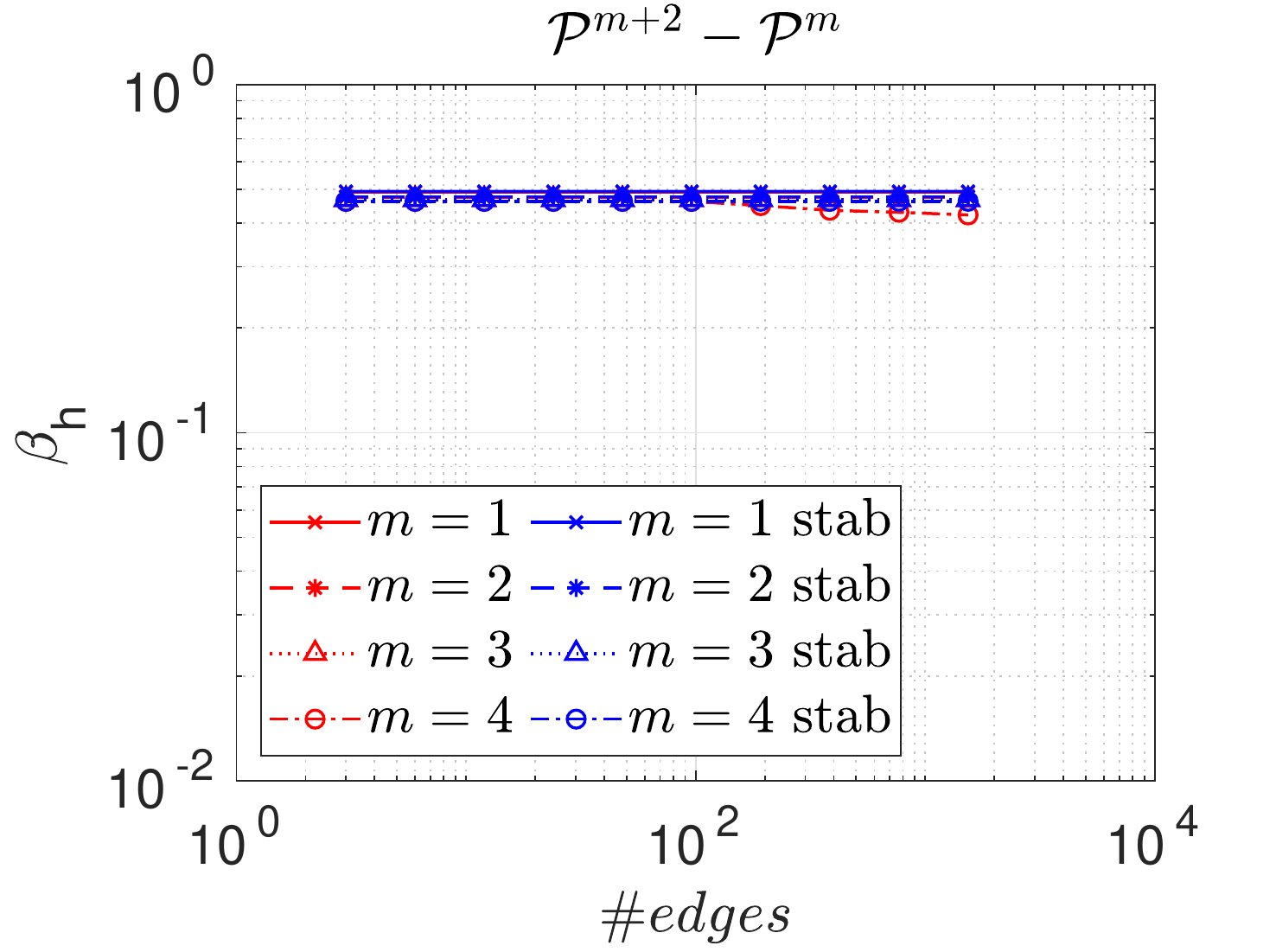}
    \caption{Values of~$\beta_h$ as function of the number of edges $\# edges$ for
    different choices of the polynomial degree for the discrete velocity and
    pressure spaces $\mathcal P^{m+k} - \mathcal P^m$, computed solving the generalized eigenvalue
    problem~\eqref{eq:GEP}. From left to right: $\mathcal P^{m} - \mathcal
    P^m$, $\mathcal P^{m+1} - \mathcal P^m$, $\mathcal P^{m+2} - \mathcal P^m$.
    The parameter $\eta=1$ for $k=0,1,2$ (blue lines), and  $\eta=0$ for $k=1,2$
    (red lines).}
    \label{fig:infsup_h_deged}
\end{figure}
%%%%%%%%%%%%%%%%%%%%%%%%%%%%%%%%%%%%%%%%%%%%%%%%%%%%%%%%%%%%%%%%%%%%

% FSI-TYPE MESH

Finally, we consider a sequence of grids that mimics fluid meshes typically appearing in fluid-structure interaction applications; see Section~\ref{sec:examplesFSI} below.
These grids are generated as follows: first, consider a uniform regular triangular mesh of a square domain; next, carve the domain out and get a hole inside it.
For example, this hole may represent a structure domain immersed in a fluid one.
In the proximity of the hole, the resulting mesh presents polygonal elements that may be non-convex, of arbitrary size and of anisotropic shape.

We consider a slender rectangular hole placed in the center of the
square domain that rotates around its center of mass, see
Figure~\ref{fig:meshesAni}, and we study the behaviour of the discrete
\textit{inf-sup} constant $\beta_h$ by varying the angle of rotation $\theta$ of
the hole.
In Figure~\ref{fig:infsup_h_ani}, we plot the value of the discrete
\textit{inf-sup} constant as a function of the angle $\theta$ for different
choices of the discrete velocity and pressure spaces $\mathcal P^{m+k} -
\mathcal P^m$, $k=0,1,2$, $m=1,2$, with and without the pressure stabilization term.
The presence of small or anisotropic elements only slightly deteriorates the constant~$\beta_h$ for the nonstabilized case ($\eta = 0$),
while they seem irrelevant for the stabilized case ($\eta = 1$).

%% MESH
\begin{figure}[!htbp]
    \centering
    \hspace{-0.9cm}
    \includegraphics[width=0.36\textwidth]{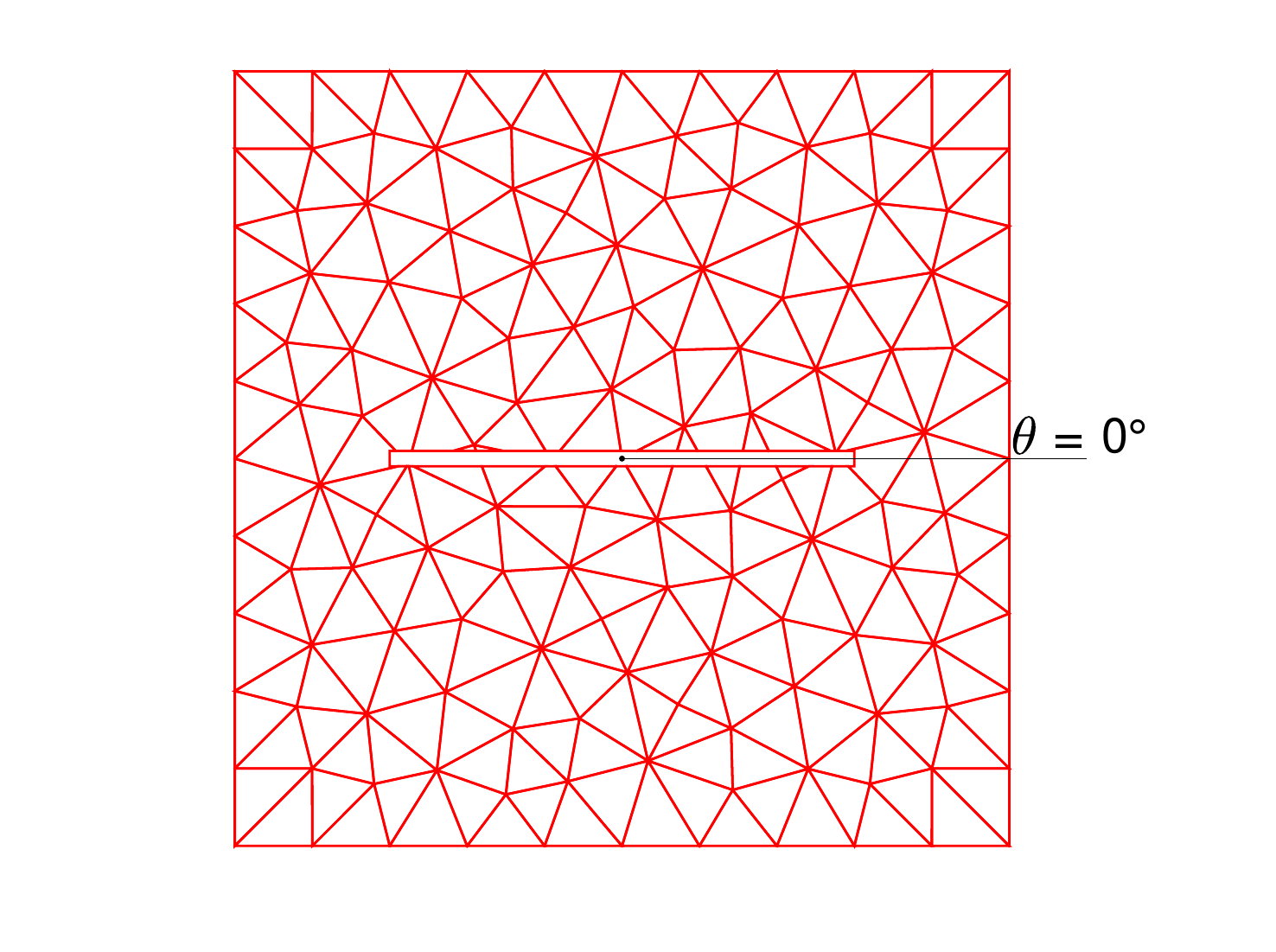}%
    \hspace{-0.4cm}
    \includegraphics[width=0.36\textwidth]{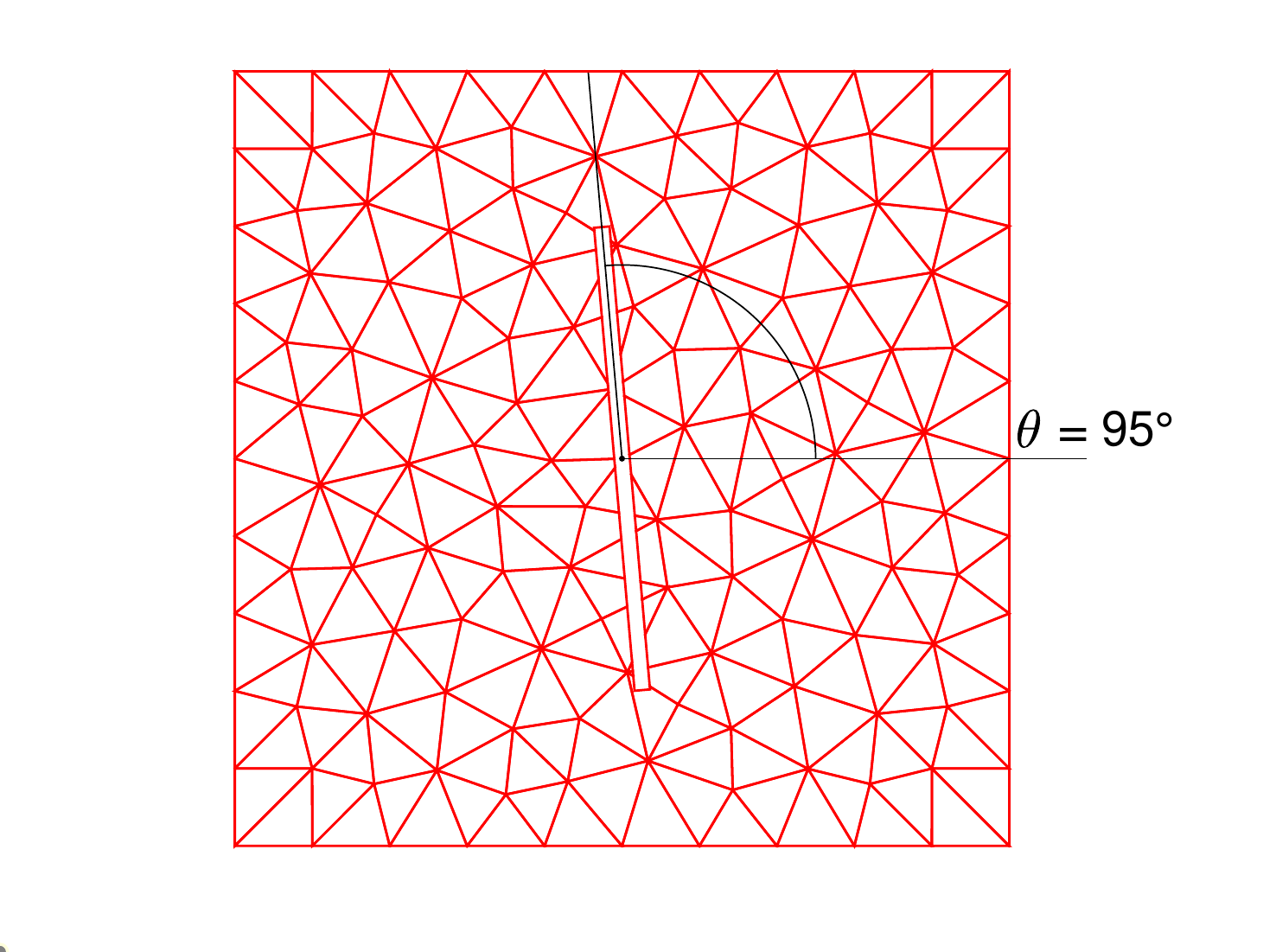}%
    \hspace{-0.4cm}
    \includegraphics[width=0.36\textwidth]{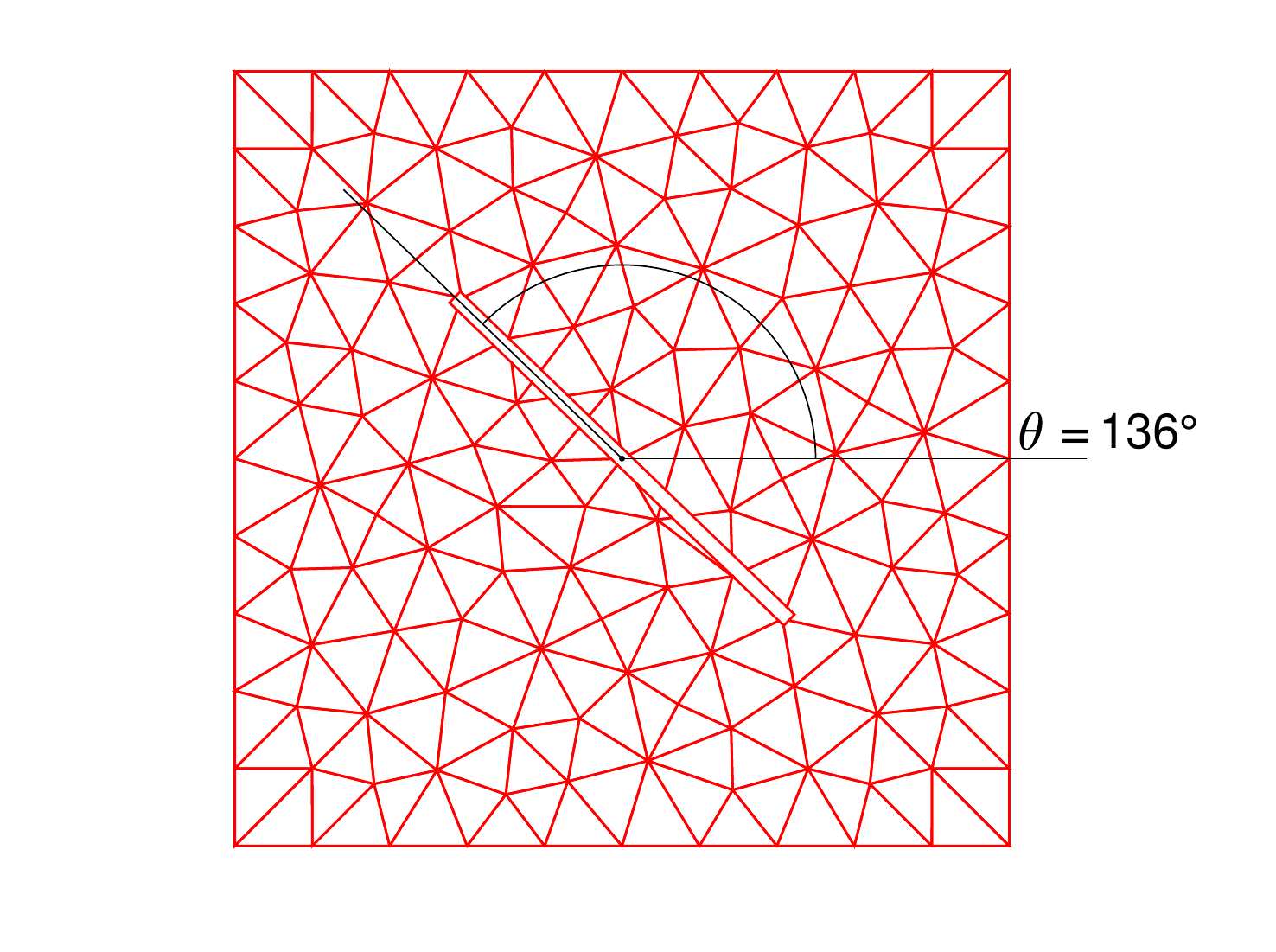}%
    \caption{Polygonal meshes obtained by rotating the hole placed in the center
    of the initial triangular mesh. Small or anisotropic elements appear.
From left to right: polygonal mesh obtained for $\theta = 0^\circ$; polygonal mesh obtained for $\theta = 95^\circ$; polygonal mesh obtained for
    $\theta = 136^\circ$.} \label{fig:meshesAni}
\end{figure}

\begin{figure}[!htbp]
    \centering
    \includegraphics[width=0.343\textwidth]{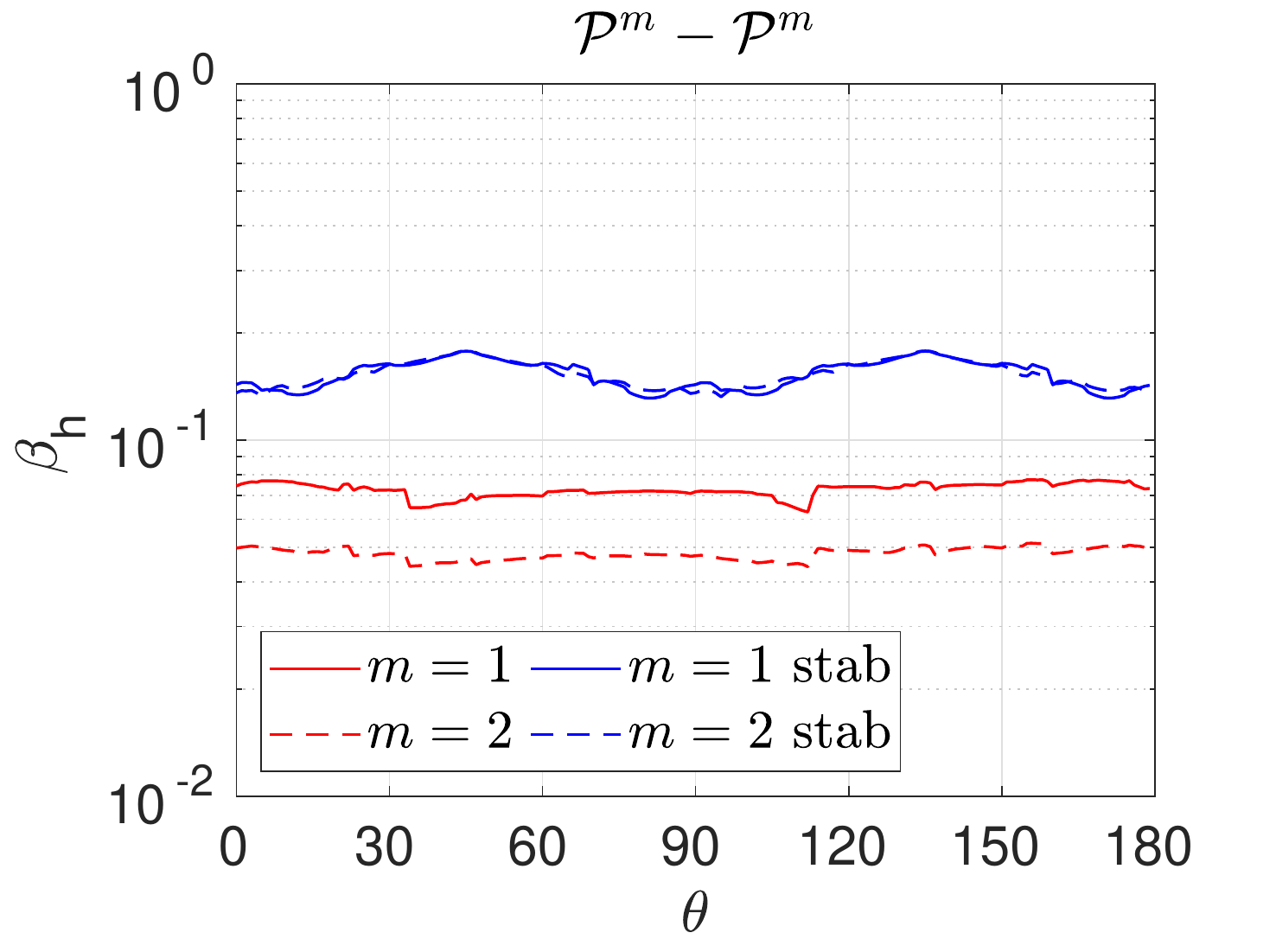}
    \hspace{-0.48cm}
    \includegraphics[width=0.343\textwidth]{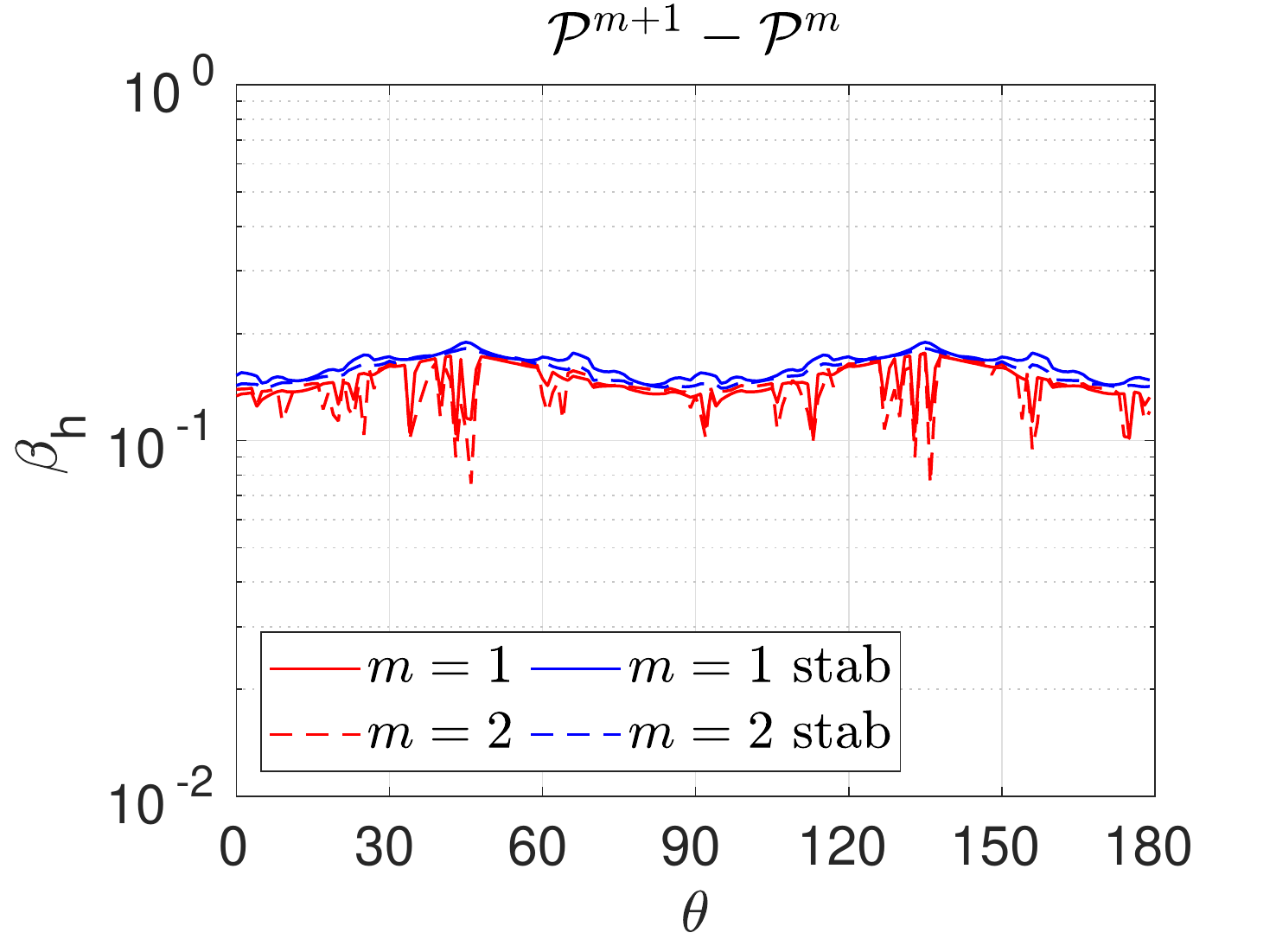}
    \hspace{-0.48cm}
    \includegraphics[width=0.343\textwidth]{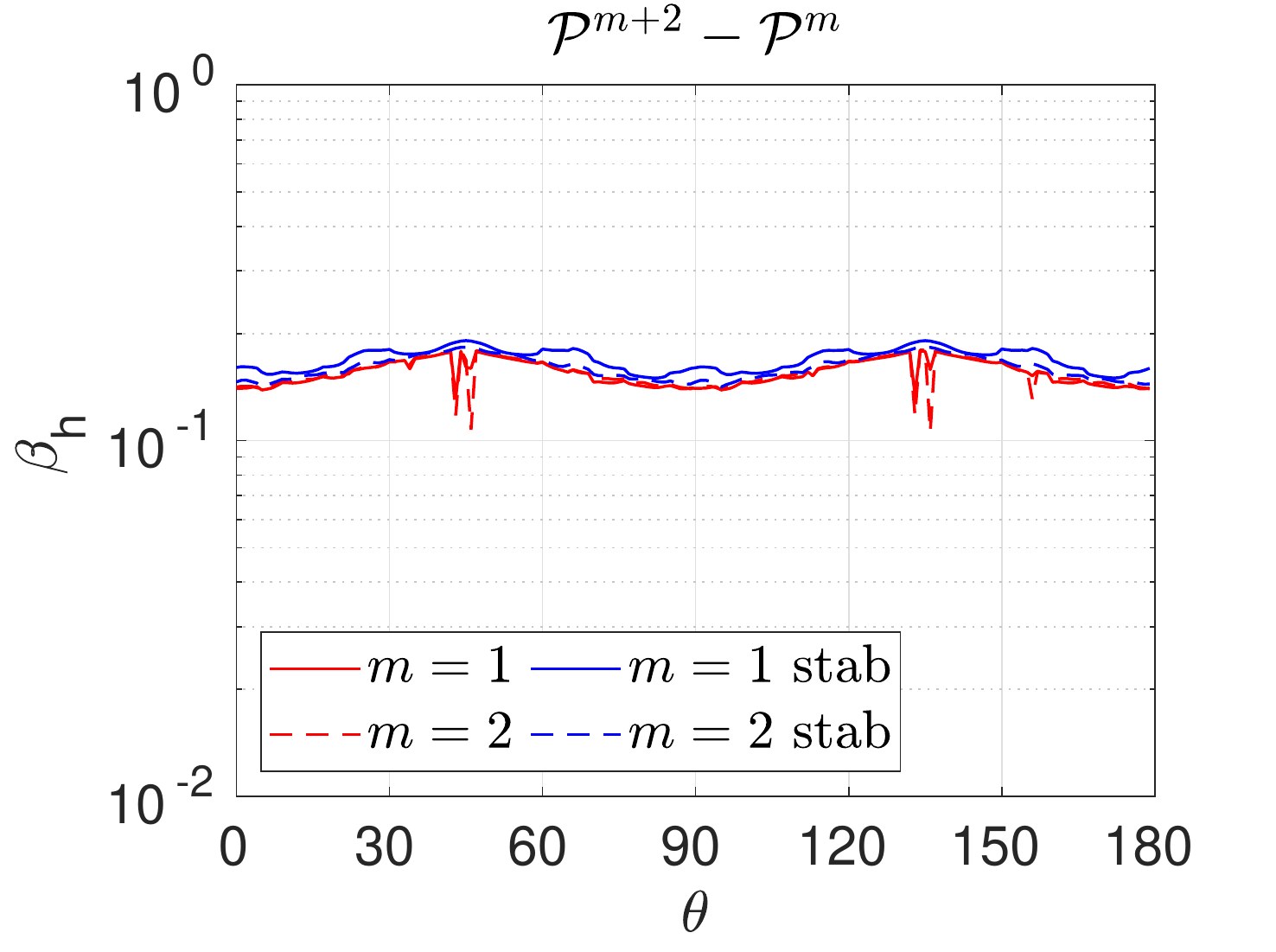}
    \caption{Values of~$\beta_h$ as function of the angle of rotation $\theta$
    for different choices of the polynomial degree for the discrete velocity and
    pressure spaces $\mathcal P^{m+k} - \mathcal P^m$, computed solving the
    generalized eigenvalue problem~\eqref{eq:GEP}. From left to right: $\mathcal
    P^{m} - \mathcal P^m$, $\mathcal P^{m+1} - \mathcal P^m$, $\mathcal P^{m+2}
    - \mathcal P^m$.  The parameter $\eta=1$ (blue lines) and $\eta=0$ (red
    lines).}
    \label{fig:infsup_h_ani}
\end{figure}
%%%%%%%%%%%%%%%%%%%%%%%%%%%%%%%%%%%%%%%%%%%%%%%%%%%%%%%%%%%%%%%%%%%%

\section{A priori error estimates for the stationary Stokes problem} \label{section:abstract-convergence}

Introduce the spaces $\mathcal X = \Vh \cap \lbrack H^2(\Omega) \rbrack^d$ and $\mathcal M = \Qh \cap H^1(\Omega)$ for the velocity and pressure, respectively.
For all $(\xx u, p), (\xx v, q) \in \mathcal X \times \mathcal M$, we consider the discrete bilinear form~$\mathcal B_h$ introduced in equation~\eqref{eq:abs_bil_form}.

Define~$\norm{(\cdot, \cdot)}_{\mathcal X \times \mathcal M}$ as the energy norm
defined on the pair of spaces $\mathcal X \times \mathcal M$. In particular, we
fix
\[
\norm{(\xx v_h, q_h)}^2_{\mathcal X \times \mathcal M} =
\normVh{\xx v_h}^2 + \normQh{q_h}^2.
\]

We now state the main result of the section.
\begin{theorem}[Abstract error estimate] \label{th:abs_err_est}
Let~$(\xx u , p) \in \mathcal X \times \mathcal M$ and~$(\xx u_h, p_h) \in \Vh \times \Qh$ be the solutions to~\eqref{eq:stokesWeak} and~\eqref{eq:abs_stokes}, respectively, and Assumptions~\ref{ass:mesh_all} and~\ref{ass:star} be valid.
Recalling that the bilinear form $\mathcal B_h(\cdot , \cdot)$ is stable and bounded with constants $C^*$ and $C_*$, the following error estimate is valid:
\begin{align}
\label{eq:abs_err_est}
    \norm{(\xx u - \xx u_h , p - p_h)}_{\mathcal X \times \mathcal M}     \le \left( 1 + \frac{C_*}{C^*} \right)\inf_{(\xx v_h, q_h) \in \Vh \times \Qh} \norm{(\xx u - \xx  v_h, p - q_h)}_{\mathcal X \times \mathcal M}.
\end{align}
\end{theorem}
\begin{proof}
For all~$(\xx v_h, q_h) \in \Vh \times \Qh$, we apply the triangle inequality and get
\begin{equation}
\label{eq:EE_main}
\norm{(\xx u - \xx u_h , p - p_h)}_{\mathcal X \times \mathcal M} \le \norm{(\xx
u - \xx v_h , p - q_h)}_{\mathcal X \times \mathcal M} + \norm{(\xx v_h - \xx
u_h , q_h - p_h)}_{\mathcal X \times \mathcal M}.
\end{equation}
We have the following Galerkin orthogonality property. Given~$({\xx u}, p)$ and~$({\xx u_h}, p_h)$ the solutions to the continuous and discrete Stokes problem, respectively, we can write
\begin{equation} \label{Galerkin-orthogonality}
\mathcal B_h((\xx{u}- \xx{u_h}, p-p_h), ({\xx w_h}, r_h))=0 \quad\quad \forall ({\xx w_h}, r_h) \in {\xx V}_h^\ell \times Q_h^m.
\end{equation}
The main tool used in proving~\eqref{Galerkin-orthogonality} is the extra smoothness required on the continuous pressure~$p$, whence the jump terms involving it disappear.

Using the coercivity, the Galerkin orthogonality~\eqref{Galerkin-orthogonality}, and the continuity of the form $\mathcal B_h$,
we can show an upper bound on the second term of the previous inequality
as follows:
\begin{align}
\label{eq:EE_2}
\begin{split}
    \norm{(\xx v_h - \xx u_h , q_h - p_h)}_{\mathcal X \times \mathcal M} &\le
    \frac{1}{C^*} \sup_{(\xx w_h, r_h) \in \Vh \times \Qh}
    \frac{ \mathcal B_h((\xx v_h - \xx u_h, q_h - p_h),(\xx w_h, r_h)) }
    {\norm{(\xx w_h, r_h)}_{\mathcal X \times \mathcal M}} \\
    &= \frac{1}{C^*}
    \sup_{(\xx w_h, r_h) \in \Vh \times \Qh}
    \frac{ \mathcal B_h((\xx v_h - \xx u, q_h - p),(\xx w_h, r_h)) }
    {\norm{(\xx w_h , r_h)}_{\mathcal X \times \mathcal M}} \\
    &\quad + \frac{1}{C^*} \sup_{(\xx w_h, r_h) \in \Vh \times \Qh}
    \frac{ \mathcal B_h((\xx u - \xx u_h, p - p_h),(\xx w_h , r_h)) }
    {\norm{(\xx w_h , r_h)}_{\mathcal X \times \mathcal M}} \\
    &\le \frac{C_*}{C^*} \norm{(\xx u - \xx v_h, p - q_h)}_{\mathcal X \times
    \mathcal M}
\end{split}
\end{align}
Inserting the inequality~\eqref{eq:EE_2} into \eqref{eq:EE_main}, the assertion
follows.
\end{proof}

Finally, by employing the approximation results reported in
Section~\ref{sec:wellposed} with Theorem~\ref{th:abs_err_est}, we show the
$hp$-version a priori error estimate for the discrete Stokes problem
in~\eqref{eq:abs_stokes}.

\begin{corollary}[Convergence rate in the energy norm] \label{cor:convergence}
Let $\mathcal T_h$ be a polytopic mesh and~$\mathcal T_h^\#$ be the corresponding covering satisfying Definition~\ref{def_mesh_polytopic_regular}.
Moreover, let Assumptions~\ref{ass:mesh_all} and~\ref{ass:star}, and the hypotheses of Theorem~\ref{th:abs_err_est} be valid. Finally assume that~$\ell \ge m-1$.
If, for any $K \in \mathcal T_h$, $(\xx u, p)|_K \in H^{r}(K) \times H^{r-1}(K)$, with $r > 1 + d/2$, such that for any $\mathcal K \in \mathcal T_h^\#$, $K \subset \mathcal K$, $(\mathcal E \xx u, \mathcal E p)|_{\mathcal K} \in H^{r}(\mathcal K) \times H^{r-1}(\mathcal K)$, then
$$
\norm{(\xx u - \xx u_h, p - p_h)}_{\mathcal X \times \mathcal M} \lesssim     \sum_{K \in \mathcal T_h} \frac{ h_K^{s-1} }{\ell^{r-3/2} }    \left( \norm{\mathcal E \xx u}_{H^r(\mathcal K)}    + \norm{\mathcal E p}_{H^{r-1} (\mathcal K)} \right),
$$
where $s = \min\{ \ell+1 , m+2 , r \}$ and the hidden positive constant is
independent of the discretization parameters.
\end{corollary}
\begin{proof}
By considering equation~\eqref{eq:abs_err_est}, we set
$$
\mathcal I = \inf_{(\xx v_h, q_h) \in \Vh \times \Qh} \norm{(\xx u - \xx v_h, p
- q_h)}_{\mathcal X \times \mathcal M}.
$$
Recall that~$\Pi^\ell$ denotes the best polynomial approximant introduced in Lemma~\ref{lm:localapproxTPi}.
With an abuse of notation, we shall use the same symbol for scalar, vector, and tensor approximants. We have
\begin{align*}
\begin{split}
\mathcal{I}^2 &=
\inf_{(\xx v_h, q_h) \in \Vh \times \Qh} \norm{(\xx u - \xx v_h, p -
q_h)}^2_{\mathcal X \times \mathcal M} \le \norm{(\xx u - \Pi^{\ell} \xx u, p -
\Pi^{m} p)}^2_{\mathcal X \times \mathcal M},\\
& \le \underbrace{\sum_{K \in \mathcal T_h} \norm{ \mu^{1/2} \nabla \left( \xx u -
\Pi^{\ell} \xx u \right) }^2_{L^2(K)} + \sum_{F \in \mathcal F_h} \norm{
\sigma_v^{1/2} \jump{\xx u - \Pi^{\ell} \xx u } }^2_{L^2(F)}}_{\circled{A}}\\
& \qquad + \underbrace{\sum_{K \in \mathcal T_h} \norm{ p - \Pi^{m} p }^2_{L^2(K)} +
\sum_{F \in \mathcal F_h} \norm{ \sigma_p^{1/2} \jump{p - \Pi^{m} p }
}^2_{L^2(F)}}_{\circled{B}}.
\end{split}
\end{align*}
By using Lemmata~\ref{lm:localapproxTPi} and~\ref{lm:discInvIn1},
Assumption~\ref{ass:local_unif}, and~$\ell \ge m-1$, we can prove
the following bounds on the terms $\circled{A}$ and $\circled{B}$:
\begin{align*}
\begin{split}
\circled{A}
&\lesssim \sum_{K \in \mathcal T_h} \mu \frac{h_K^{2(s_u -1)}}{\ell^{2(r-1)}}
\norm{ \mathcal E \xx u }^2_{H^r(\mathcal K)}
+ \sum_{K \in \mathcal T_h} (\max_{F \subset \partial K} \sigma_v )
\frac{h_K^{2(s_u-1/2)}}{\ell^{2(r-1/2)}} \norm{ \mathcal E \xx u }^2_{H^r(\mathcal
K)}\\
&\lesssim \sum_{K \in \mathcal T_h} \frac{h_K^{2(s_u -1)}}{\ell^{2(r-1)}} \norm{
\mathcal E \xx u }^2_{H^r(\mathcal K)}
+ \sum_{K \in \mathcal T_h} \frac{\ell^2}{h_K}
\frac{h_K^{2(s_u-1/2)}}{\ell^{2(r-1/2)}} \norm{ \mathcal E \xx u
}^2_{H^r(\mathcal K)}\\
&\lesssim \sum_{K \in \mathcal T_h} \frac{h_K^{2(s_u -1)}}{\ell^{2(r-3/2)}} \norm{
\mathcal E \xx u }^2_{H^r(\mathcal K)},
\end{split}
\end{align*}
and
\begin{align*}
\begin{split}
\circled{B}
&\lesssim \sum_{K \in \mathcal T_h} \frac{h_K^{2s_p}}{m^{2(r-1)}} \norm{\mathcal
E p}^2_{H^{r-1}(\mathcal K)}
+ \sum_{K \in \mathcal T_h} (\max_{F \subset \partial K} \sigma_p)
\frac{h_K^{2(s_p-1/2)}}{m^{2(r-3/2)}} \norm{ \mathcal E p}^2_{H^{r-1}(\mathcal
K)}\\
&\lesssim \sum_{K \in \mathcal T_h} \frac{h_K^{2s_p}}{m^{2(r-1)}} \norm{\mathcal
E p}^2_{H^{r-1}(\mathcal K)}
+ \sum_{K \in \mathcal T_h} \frac{h_K}{m} \frac{h_K^{2(s_p-1/2)}}{m^{2(r-3/2)}}
\norm{ \mathcal E p}^2_{H^{r-1}(\mathcal K)}\\
&\lesssim \sum_{K \in \mathcal T_h} \frac{h_K^{2s_p}}{m^{2(r-1)}} \norm{\mathcal
E p}^2_{H^{r-1}(\mathcal K)},
\end{split}
\end{align*}
where $s_u = \min\{ \ell+1, r \}$ and $s_p = \min\{m+1, r-1\}$.

Finally, we obtain, $s= \min \{ \ell+1 , m+2, r \}$,
$$
\mathcal{I}^2 \lesssim \sum_{K \in \mathcal T_h}\frac{h_K^{2(s-1)}}{\ell^{2(r-3/2)}}
\left( \norm{\mathcal E \xx u}_{H^r(\mathcal K)}
+ \norm{\mathcal E p}_{H^{r-1}(\mathcal K)} \right)^2.
$$

By inserting the bound on~$\mathcal I$ into~\eqref{eq:abs_err_est}, the assertion follows.
\end{proof}

\begin{remark} \label{remark:suboptimality-convergence}
The estimate of Corollary~\ref{cor:convergence} is suboptimal in terms of half a polynomial order also due to the presence of the coercivity constant~$C^* = C^* (\beta_h)$,
where $\beta_h$ is the discrete generalized \textit{inf-sup} constant introduced in Proposition~\ref{prop:gen_inf_sup}.
\end{remark}

\begin{remark}
By assuming $h \simeq h_K$, for any $K \in \mathcal T_h$, and uniform regularity
of the solution, the estimate in Corollary~\ref{cor:convergence} becomes:
$$
    \norm{(\xx u - \xx u_h, p - p_h)}_{\mathcal X \times \mathcal M} \lesssim
    \frac{ h^{s-1} }{\ell^{r-3/2} }
    \left( \norm{\mathcal E \xx u}_{H^r(\cup_{K \in \mathcal T_h} \mathcal K)}
    + \norm{\mathcal E p}_{H^{r-1} (\cup_{K \in \mathcal T_h}\mathcal
    K)} \right).
$$
where $s = \min\{ \ell+1 , m+2 , r \}$ and the hidden positive constant is
independent of the discretization parameters.
\end{remark}

\begin{remark}
We can also prove a priori error estimates by setting minimal regularity $\mathcal X = \Vh \cap H^1_0(\Omega)$ and $\mathcal M = \Qh$ for the velocity and pressure, respectively.
This requires to introduce a nonconsistent formulation, modify the bilinear form in equation~\eqref{eq:abs_bil_form}, and consider the residual term of the Strang's lemma in
Theorem~\ref{th:abs_err_est}.
\end{remark}

\section{An application: PolyDG for FSI problems}\label{sec:time}

In this section, we introduce a continuos FSI problem and its PolyDG
discretization, with the aim of further exploring the stability properties of
the PolyDG discretization of the Stokes problem and their impact on the
approximation of related differential problems; see
Section~\ref{sec:examplesFSI} below.

Let $\Omega \subset \mathbb{R}^d$ and $\Omega_s \subset \mathbb{R}^d$, $d =
2,3$, be two polygonal/polyhedral domains.
In~$\Omega$, we consider an incompressible Newtonian fluid with density $\rho$
and dynamic viscosity $\mu$, where $\xx{u}$ and $p$ are the fluid velocity and
pressure, while in $\Omega_s$ we consider a linear elastic material with density
$\rho_s$, Young's modulus $E$, and Poisson's ratio $\nu$, where $\xx{d}$ is the
solid displacement.

In what follows, we denote by $\Sigma$ the fluid-structure interface and by
$\xx{n}$ its normal vector pointing outwards of $\Omega_s$. We indicate with
$\partial \Omega$ and $\partial \Omega_s$ the outer boundary of the fluid and
solid domain, respectively. The domains may change in time.

The fluid-structure interaction problem reads as follows:
for any $t \in ( 0 , T ]$, with $T>0$, find the fluid velocity $\xx{u} = \xx{u}(t)$, the fluid pressure $p = p(t)$, and the solid displacement $\xx{d} = \xx{d}(t)$, such that
%\[
%\label{eq:fsiPb}
    \begin{align*}
    & \rho \partial_t \xx{u} - \nabla \cdot \xx{T}_f (\xx{u} , p) = \xx f     && \text{in} \; \Omega(t),\\%  \label{eq:fluidEq1} \\
    & \nabla \cdot \xx{u} = 0     && \text{in} \; \Omega(t), \\% \label{eq:fluidEq2}\\
    & \xx{u} = 0     && \text{on} \; \partial \Omega, \\% \label{eq:fluidBc}\\[0.3cm]
    & \xx{u} = \partial_t \xx{d}     && \text{on} \; \Sigma(t), \\% \label{eq:interfaceEq1} \\
    & \xx{T}_f (\xx{u},p )\xx{n} = \xx{T}_s (\xx{d}) \xx{n}  && \text{on} \; \Sigma(t), \\% \label{eq:interfaceEq2} \\[0.3cm]
    & \rho_s \partial_{tt} \www{\xx{d}} - \nabla \cdot \www{\xx{T}}_s (\xx{d}) = \xx f_s  && \text{in} \; \www \Omega_s, \\%\label{eq:solidEq} \\
    & \www{\xx{d}} = \xx{0}     && \text{on} \; \partial \www \Omega_s, %\label{eq:solidBc}
    \end{align*}
%\]
where $\xx{T}_f (\xx{u}, p) = 2 \mu \xx{D}(\xx{u}) - p\xx{I}$ is the fluid Cauchy
stress tensor and $\www{\xx{T}}_s (\xx{d}) = 2 \mu_s \xx{D}(\www{\xx{d}}) + \lambda_s \nabla \cdot
\www{\xx{d}} \xx{I}$ is the solid first Piola-Kirchhoff stress tensor, with
$\xx{D} (\xx{w}) = 1/2 \left( \nabla \xx {w} + \nabla^T \xx{w} \right)$ and
$\lambda_s = \frac{E \nu}{(1+\nu)(1-2\nu)}$, $\mu_s = \frac{E}{2(1+\nu)}$ are
the Lam\'e parameters.

The structure problem is written in the reference configuration
$\www \Omega_s = \Omega_s(t=0)$, and all the related quantities are indicated
with the $\www{\cdot}$ notation.

\medskip

Given the time discretization parameter~$\Delta t > 0$, we indicate with~$t^n =n \Delta t$, $n \ge 0$, the $n$-th time step and indicate the approximation of the unknown~$u$ at time~$t^n$ by~$u^n$.
We introduce the fluid and solid meshes $\mathcal T_{f,h}^n$ and $\mathcal T_{s,h}^n$, respectively, of the fluid and solid domains $\Omega(t^n)$ and~$\Omega_s (t^n)$, respectively.
We denote the $(d-1)$-dimensional faces at time $t^n$ of the fluid and solid meshes by~$\mathcal{F}^n_{f,h}$ and~$\mathcal{F}^n_{s,h}$, respectively, except the set of faces composing the fluid-structure interface~$\Sigma$ at time~$t^n$, which are denoted by~$\mathcal{F}^n_{\Sigma,h}$.
Finally, $\xx V^{\ell,n}_h$ and~$Q^{m,n}_h$ are the fluid velocity and pressure spaces evaluated at time~$t^n$, defined as
\[
\left.
\begin{array}{l}
    \xx V^{\ell,n}_h = \{\xx v \in [L^2(\Omega (t^n))]^d:\,\xx v|_{K}\in
    [\mathcal{P}^\ell(K)]^d \; \forall K\in\mathcal{T}_{f,h}^n\},\\[0.2cm]
    Q^{m,n}_h = \{q \in L^2_0(\Omega (t^n)):\, q|_{K}\in
    \mathcal{P}^m(K) \; \forall K\in\mathcal{T}_{f,h}^n\}.
\end{array}
\right.
\]
The solid displacement space $\xx W^{\ell}_h$ evaluated in the reference configuration is defined as
\[
\left.
\begin{array}{l}
    \xx W^{\ell}_h = \{\xx w \in [L^2(\www{\Omega}_s )]^d:\,\xx w|_{K}\in
    [\mathcal{P}^\ell(K)]^d \; \forall K\in \www{\mathcal{T}}_{s,h}\}.
\end{array}
\right.
\]
We have assumed that the spatial polynomial order~$\ell$ is the same for both
the fluid velocity and the solid displacement.

Given $r\in\mathbb{N}^+$, we apply a Backward Difference Formula (BDF) scheme~\cite{hairer1993solving} of order $r$ both for the fluid and the solid subproblems.
We indicate the coefficients appearing in the approximation of the first and second order time derivatives with~$\xi_i$ and~$\zeta_i$, $i=0,\ldots,r$, respectively.

Define
\begin{align}
\begin{split}
A_{f,h}^n \left( \xx{u}_h^n, p_h^n ; \xx{v}_h , q_h \right)
&= \rho \left( \cfrac{\xi_0}{\Delta t} \xx{u}_h^n , \xx{v}_h\right)_{\Omega^n} + a_{f,h}^n \left( \xx{u}_h^n, \xx{v}_h \right) + b_{h}^n \left( p_h^n , \xx{v}_h \right)  - b_{h}^n \left( q_h , \xx{u}_h^n \right)\\
&\quad + s_h^n \left( p_h^n , q_h \right); \label{eq:fsiFluid}
\end{split}
\end{align}
\begin{align}
\begin{split}
A_{s,h}^n \left( \www{\xx{d}}_h^n , \www{\xx{w}}_h \right)
&= \rho_s \left( \cfrac{\zeta_0}{\Delta t^2} \www{\xx{d}}_h^n ,\www{\xx{w}}_h\right)_{\www\Omega_s} + a_{s,h} \left( \www{\xx{d}}_h^n , \www{\xx{w}}_h \right); \label{eq:fsiSolid}
\end{split}
\end{align}
\begin{align}
\begin{split}
A_{\Sigma,h}^n (\xx u_h^n, p_h^n
&, {\xx d}_h^n;\xx v_h, q_h, \xx w_h) = - \left( \delta \xx {T}_f ( \xx{u}_h^n , p_h^n ) \xx{n} + ( 1 -  \delta )\xx{T}_s (\xx{d}_h^n ) \xx{n} , \xx{v}_h - \xx{w}_h\right)_{\mathcal  F_{\Sigma,h}^n}\\
&- \left( \xx{u}_h^n - \cfrac{\xi_0}{\Delta t} \xx{d}_h^n , \delta \xx {T}_f  ( \xx{v}_h , -q_h ) \xx{n} + ( 1 - \delta ) \xx{T}_s (\xx{w}_h) \xx{n} \right)_{\mathcal F_{\Sigma,h}^n}\\
&+ \left( \sigma_\Sigma(\xx{u}_h^n - \cfrac{\xi_0}{\Delta t} {\xx d}_h^n),\xx{v}_h - \xx{w}_h \right)_{\mathcal F_{\Sigma,h}^n}; \label{eq:fsiInterface}
\end{split}
\end{align}
\begin{align}
\begin{split}
    F_h^n(\xx v_h,{\xx w}_h)
    & =  \rho \left( \sum_{i=1}^r \cfrac{\xi_i}{\Delta t} \xx{u}_h^{n-i} , \xx{v}_h\right)_{\Omega^n} + \rho_s \left( \sum_{i=1}^r \cfrac{\zeta_i}{\Delta t^2} \www{\xx{d}}_h^{n-i} , \www{\xx{w}}_h  \right)_{\www\Omega_s}\\
    &+ \left( \sum_{i=1}^r \cfrac{\xi_i}{\Delta t} \xx{d}_h^{n-i} , \delta \xx {T}_f ( \xx{v}_h , -q_h ) \xx{n} + ( 1 - \delta) \xx{T}_s (\xx{w}_h) \xx{n} \right)_{\mathcal F_{\Sigma,h}^n}\\
    &- \left( \sigma_\Sigma \sum_{i=1}^r \cfrac{\xi_i}{\Delta t} \xx{d}_h^{n-i} ,\xx{v}_h -  \xx{w}_h \right)_{\mathcal F_{\Sigma,h}^n} + \left( \xx f , \xx{v}_h\right)_{\Omega^n} + \left( \www{\xx f}_s , \www{\xx{w}}_h\right)_{\www\Omega_s}.  \label{eq:fsiRHS}
\end{split}
\end{align}
The fully-discrete PolyDG approximation reads as follows: given $\delta\in [0,1]$, $\sigma_{v}\in L^{\infty}(\mathcal F_{f,h}^n)$,
$\sigma_{p}\in L^{\infty}(\mathcal F_{f,h}^n)$, $\www{\sigma}_s\in L^{\infty}(\www{\mathcal F}_{s,h})$, $\sigma_{\Sigma}\in L^{\infty}(\mathcal F_{\Sigma,h}^n)$, $\xx f \in [L^2(\Omega_f^n)]^2$ and $\www{\xx f}_s\in [L^2(\www{\Omega}_s)]^2$,
for $n >0$, find $( \xx{u}_h^n, p_h^n, \www{\xx{d}}_h^n ) \in \xx{V}^{\ell,n}_h \times Q^{m,n}_h \times \xx{W}^\ell_h$, such that
\[
\begin{split}
A_{f,h}^n \left( \xx{u}_h^n, p_h^n ; \xx{v}_h , q_h \right) + A_{s,h}^n \left( \www{\xx{d}}_h^n , \www{\xx{w}}_h \right) + A_{\Sigma,h}^n (\xx u_h^n, p_h^n , {\xx d}_h^n; \xx v_h, q_h , \xx w_h) = F_h^n (\xx v_h,{\xx w}_h),
\end{split}
\]
for all $( \xx{v}_h, q_h, \www{\xx{w}}_h ) \in \xx{V}^{\ell,n}_h \times Q^{m,n}_h \times \xx{W}^\ell_h$.

In \eqref{eq:fsiFluid}, the pressure stabilization term $s_h^n :L^2_0 \times L^2_0 \rightarrow\R$ is that given in Section~\ref{sec:polydg} evaluated on $\mathcal{F}_{f,h}^{n,i}$.
In \eqref{eq:fsiSolid}, we have introduced the bilinear forms $a_{f,h}^n: [H^1(\mathcal{T}_{f,h}^n)]^d\times [H^1(\mathcal{T}_{f,h}^n)]^d \rightarrow \R$, $b_h^n: L^2_0 \times [H^1(\mathcal{T}_{f,h}^n)]^d \rightarrow \R$ and
$a_{s,h} : [H^1(\www{\mathcal{T}}_{s,h})]^d \times [H^1(\www{\mathcal{T}}_{s,h})]^d \rightarrow \R$, which are defined as
\begin{align*}
\begin{split}
a_{f,h}^n ( \xx{u}_h^n, \xx{v}_h ) &= \int_\Omega^n 2 \mu \xx{D}_h ( \xx{u}_h^n
): \nabla_h \xx{v}_h  - \displaystyle \sum_{F \in \mathcal F_{f,h}^n} \int_F 2 \mu
\{ \xx{D}_h ( \xx{u}_h^n ) \} : \jump{\xx{v}_h}\\
& - \sum_{F \in \mathcal F_{f,h}^n} \int_F 2 \mu \jump{\xx{u}_h^n} : \{
\xx{D}_h ( \xx{v}_h ) \} + \sum_{F \in \mathcal F_{f,h}^n} \int_F \sigma_v
\jump{\xx{u}_h} : \jump{\xx{v}_h},
%\label{eq:form_a_h_f_FSI}
\end{split}
\end{align*}
\begin{equation*}
b_h^n ( p_h^n, \xx{v}_h^n ) = - \int_\Omega^n p_h^n \nabla_h \cdot \xx{v}_h^n +
\sum_{F \in \mathcal F_{f,h}^n} \int_F \{ p_h^n \xx{I} \} : \jump{ \xx{v}_h^n},
%\label{eq:form_b_h_FSI}
\end{equation*}
\begin{align*}
\begin{split}
a_{s,h} (\www{\xx{d}}^n_h, \www{\xx{w}}_h) &=
\int_{\www{\Omega}_s} 2 \mu_s \xx{D}_h (\www{\xx{d}}_h^n) : \nabla_h \www{\xx{w}}_h +
\int_{\www{\Omega}_s} \lambda_s \nabla_h \cdot \www{\xx{d}}_h^n \, \nabla_h \cdot
\www{\xx{w}}_h\\
& - \displaystyle \sum_{F \in \www{\mathcal{F}}_{s,h}} \int_F 2 \mu_s \{
\xx{D}_h ( \www{\xx{d}}_h^n ) \} : \jump{\www{\xx{w}}_h} - \sum_{F \in
\www{\mathcal{F}}_{s,h}} \int_F \lambda_s \{ \nabla_h \cdot \www{\xx{d}}_h^n
\xx{I}\} : \jump{ \www{\xx{w}}_h }\\
& - \displaystyle \sum_{F \in \www{\mathcal{F}}_{s,h}}
\int_F 2 \mu_s \jump{ \www{\xx{d}}_h^n } : \{ \xx{D}_h ( \www{\xx{w}}_h ) \}
- \displaystyle \sum_{F \in \www{\mathcal{F}}_{s,h}} \int_F \lambda_s \jump{
\www{\xx{d}}_h^n } : \{ \nabla_h \cdot \www{\xx{w}}_h \xx{I} \}\\
& + \sum_{F \in \www{\mathcal{F}}_{s,h}} \int_F \www{\sigma}_s \jump{
\www{\xx{d}}_h^n } : \jump{\www{\xx{w}}_h},
%\label{eq:form_a_h_s_FSI}
\end{split}
\end{align*}
where $\xx{D}_h (\xx{w}) = 1/2 (\nabla_h \xx{w} + \nabla_h^T \xx{w})$.

The functions $\sigma_v$ and $\sigma_p$ are given in Definition~\ref{def:sigmav_sigmap} on $\mathcal{F}_{f,h}^n$,
while $\www{\sigma}_s : \mathcal{F}_{s,h}^n \rightarrow \R$ and $\sigma_{\Sigma} : \mathcal{F}_{\Sigma,h}^n \rightarrow \R$ are defined as
\begin{align*}
\www{\sigma}_s|_F =
\begin{cases}
\yy \gamma_s \max_{K^+,K^-} \left\{ \frac{\ell^2
\overline{\mathcal{C}}_{s,K}}{h_K}\right\} &  F \in \mathcal F_{s,h}^{n,i},\\
\yy \gamma_s \frac{\ell^2 \overline{\mathcal{C}}_{s,K}}{h_K} &  F \in \mathcal
F_{s,h}^{n,b},
\end{cases}\\
\sigma_{\Sigma}|_F =
\yy \gamma_{\Sigma} \max_{K^+,K^-} \left\{ \frac{\ell^2}{h_K} \left( \delta
\mu + (1 -\delta) \overline{\mathcal{C}}_{s,K}
\right) \right\} & \quad F \in \mathcal F_{\Sigma,h}^{n},
\end{align*}
with $\gamma_s, \gamma_\Sigma$ positive constants,
$\overline{\mathcal{C}}_{s,K} = \norm{\mathcal{C}_s|K}_{l^2}$ and
$\mathcal{C}_{s,K}$ the linear elasticity fourth order tensor.

The fluid domain $\Omega^n$ and the interface $\Sigma^n$ in \eqref{eq:fsiFluid}, \eqref{eq:fsiInterface}, and~\eqref{eq:fsiRHS} are unknown. Thus, they are approximated with extrapolations of order~$r$ of the domains at the previous time steps.

\begin{remark}
\label{rem:delta}
For the numerical stability of the FSI problem, theoretical results show that the parameter $\delta$ appearing in the interface terms of equation~\eqref{eq:fsiInterface} and~\eqref{eq:fsiRHS} has to be set equal to~$1$; see \cite{antonietti2019numerical}.
\end{remark}

\section{Numerical results} \label{sec:examplesFSI}
In this section, we present some numerical experiments for the steady Stokes problem and the time-dependent FSI problem.
In Section~\ref{sec:conv_stokes}, we assess the order of accuracy of the method for the steady Stokes problem as the spatial discretization parameter tends to zero and the spatial polynomial degree increases.
In Section~\ref{sec:pres_fsi}, we consider a FSI problem and we numerically compare the pressure field for different choices of the velocity and pressure spaces, with and without the pressure stabilization term.
Finally, in Section~\ref{sec:ex_fsi}, we show that the proposed PolyDG method is able to reproduce the expected dynamics of a time-dependent FSI problem.

For all the proposed examples, the resulting linear system corresponding to the Stokes and
the FSI problems is solved in Matlab by means of a direct method.

\subsection{The steady Stokes problem: convergence results}
\label{sec:conv_stokes}

Here, we numerically estimate the order of convergence of the steady Stokes problem with respect to the spatial parameter $h$ when it tends to zero and the spatial polynomial degree increases.

We consider a square unit domain $\Omega = \lbrack 0 ,1 \rbrack^2$ and the exact solution
\begin{equation*}
    \xx u_{ex} =
    \left[
    \arraycolsep=8pt\def\arraystretch{2}
    \begin{array}{c}
        - cos(2 \pi x) sin(2 \pi y)\\
         sin(2 \pi x) cos(2 \pi y)\\
    \end{array} \right],
    \qquad
    p_{ex} = 1 - e^{-x(x-1)(x-0.5)^2 - y(y-1)(y-0.5)^2}.
\end{equation*}
The forcing term~$\xx f$ and the Dirichlet boundary conditions are computed accordingly.
We picked $\xx {u}_{ex}$ so that $\nabla \cdot \xx u_{ex} = 0$.
We set~$\mu = 1$, $\gamma_v = 10$, $\gamma_p = 10$ and~$m = \ell = 4$.
In Figure~\ref{fig:conv_hp} (left), we plot the error in the~$L^2$ and $DG$-norms of the velocity, in the $L^2$-norm of the pressure and in the pressure semi-norm $\seminormJ{\cdot}$ versus $h \frac{1}{\sqrt{N_{el}}}$.
The expected order of convergence are found. In Figure~\ref{fig:conv_hp}
(right), we show the errors with respect to the polynomial degree~$m$, with an
underlying uniform and regular polygonal mesh, generated via \texttt{PolyMesher}
\cite{talischi2012polymesher}, consisting of $N_{el} = 160$ elements.
We observe exponential convergence in terms of the polynomial degree.

\begin{figure}[h]
    \centering
    \includegraphics[width=0.51\textwidth]{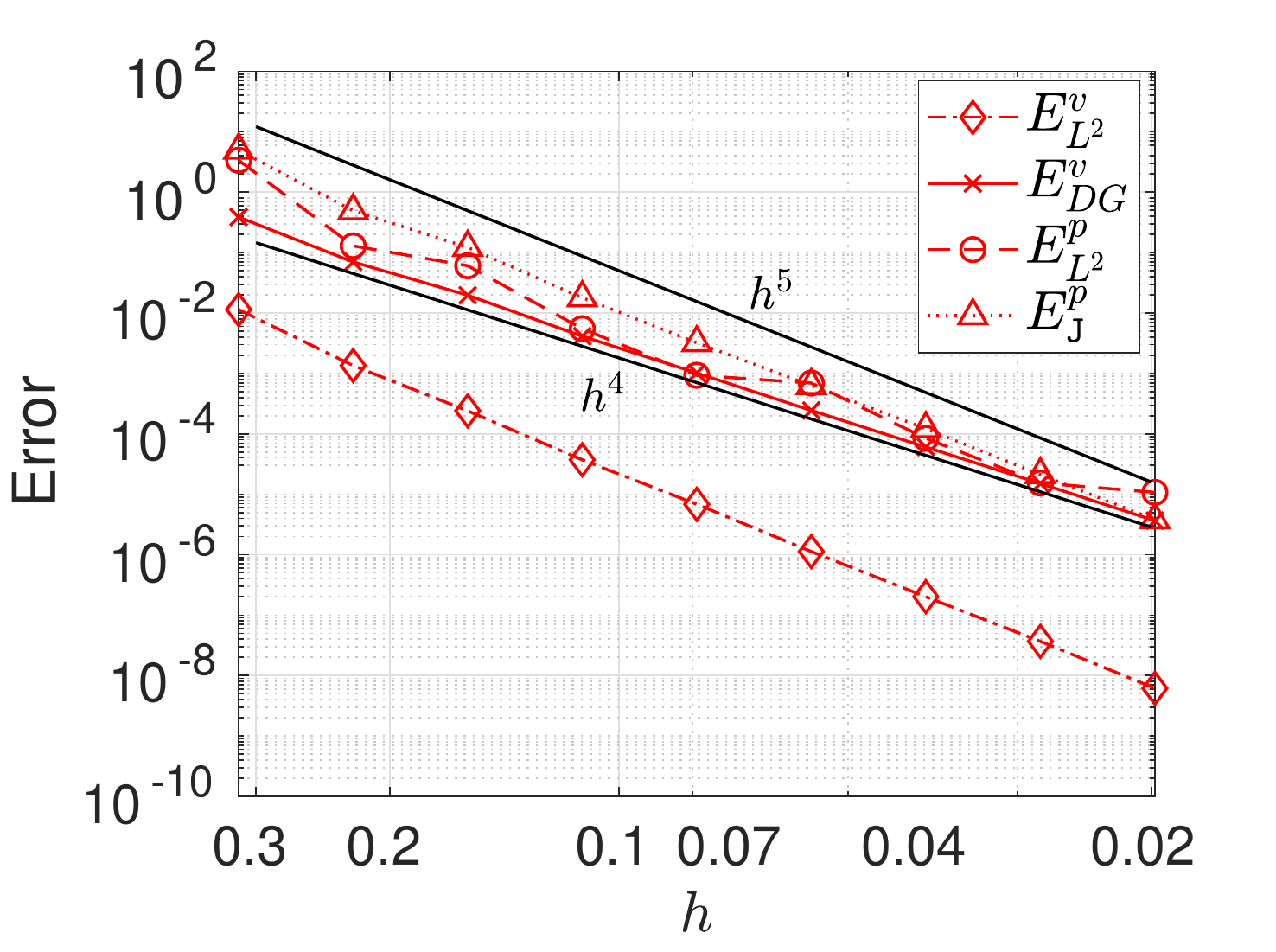}
    \hspace{-0.6cm}
    \includegraphics[width=0.51\textwidth]{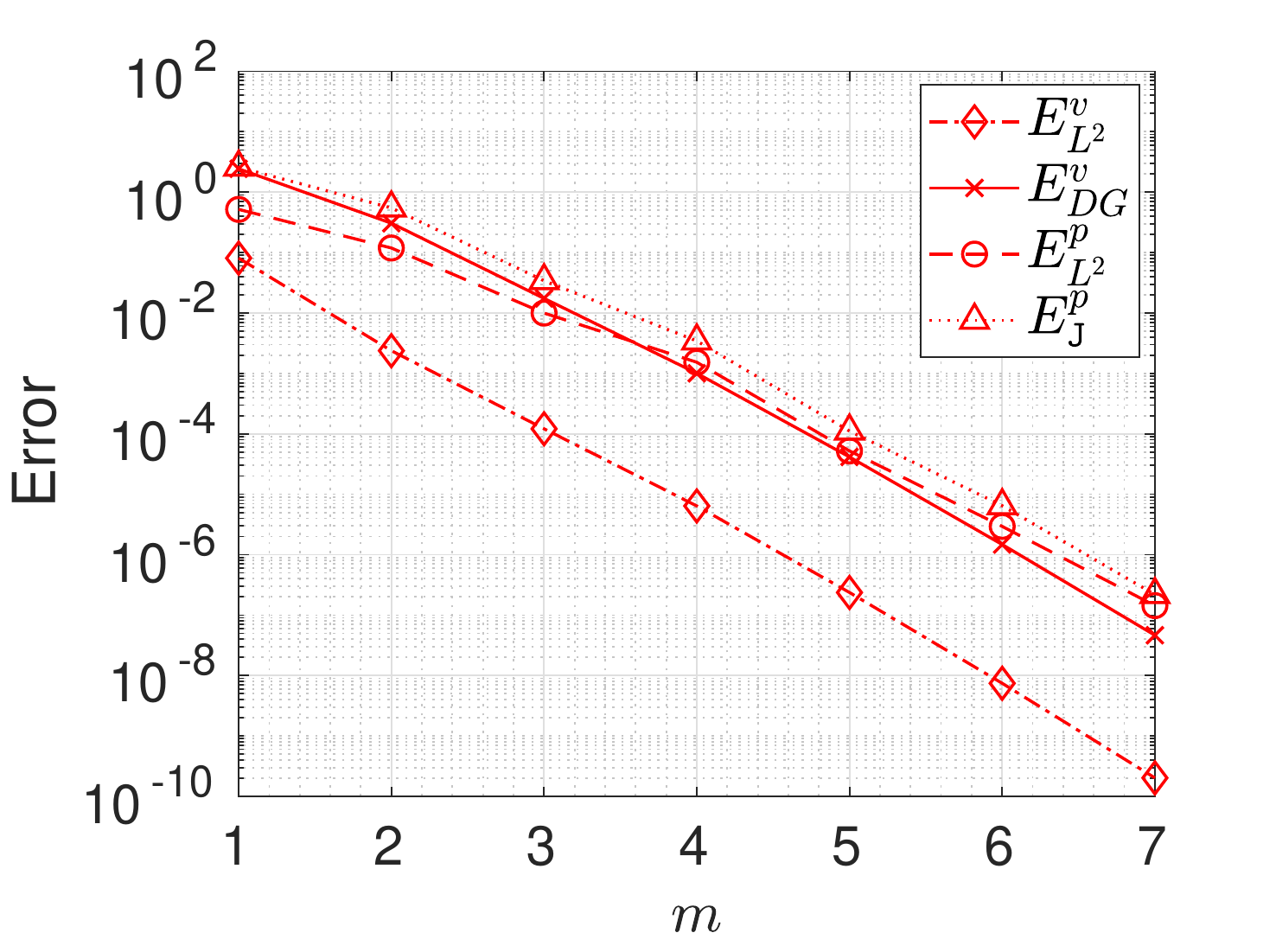}%
    \caption{Numerical estimates of the order of convergence with respect
    to the mesh size $h$ (left) and the polynomial degree $m$ (right).}
    \label{fig:conv_hp}
\end{figure}

\begin{remark}
Within the theoretical setting of the paper, we cannot prove the exponential convergence of the $p$-version of the method.
Notwithstanding, it is the expected behaviour in the standard Galerkin setting with simplicial and tensor product element meshes for analytic solutions; see, e.g., \cite{SchwabpandhpFEM} and the references therein.
The reason of this resides in the continuity property of the Stein extension operator~\eqref{continuity:Stein}, which is valid modulo a hidden constant depending on the involved Sobolev regularity~$s$.
In particular, when trying to recover exponential convergence, a term growing more than exponentially with respect to~$s$ appears.

A possible way to overcome this issue would be to resort to a different approach, where we assume that the solution is analytic over a slightly larger domain than~$\Omega$.
In particular, we should substitute the approximation result in Lemma~\ref{lm:localapproxTPi} with some approximation properties by means of tensor product Legendre polynomials on tensor product element and Koornwinder polynomials on simplicial elements;
see, e.g., \cite{SchwabpandhpFEM} and~\cite{braess2000approximation} for more details, respectively.
We avoid further details on this point, for it might render the understanding of the paper more cumbersome.

The suboptimality in terms of the polynomial degree
due to the nonrobustness of the \textit{inf-sup} condition, see
Remark~\ref{remark:suboptimality-convergence}, is eaten up by the expected
exponential convergence for analytic solutions.
\end{remark}

\subsection{The fluid-structure interaction problem: numerical comparison of the pressure fields}
\label{sec:pres_fsi}

In this first numerical test we compare the pressure field for different choices of the spatial polynomial degree of the discrete velocity and pressure spaces for a FSI problem.
The fluid domain $\Omega$ represents a viscous fluid with density $\rho = 1 \, \text{g/cm}^2$ and viscosity $\mu = 0.03 \, \text{g/s}$, while the structure domain $\Omega_s$ is a linear elastic barrier
that horizontally divides the fluid domain in two compartments; see Figure~\ref{fig:domainCompPres}.
For the structure we set the density $\rho_s = 1.2 \, \text{g/cm}^2$, the Young's modulus $E = 2 \cdot 10^4\, \text{dyne/cm}$, and the Poisson's ratio $\nu = 0.49$.

\begin{figure}[h]
    \centering
    \resizebox{0.7\textwidth}{!}{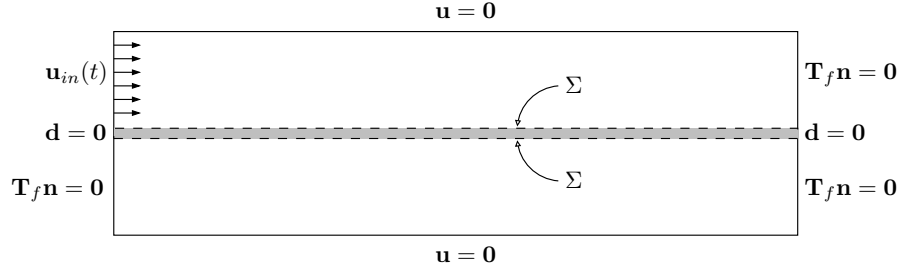}%
    \caption{Setting of the boundary conditions on the fluid (white) and
    structure (grey) domains. The moving fluid-structure interface is depicted
    by dashed lines.}
    \label{fig:domainCompPres}
\end{figure}

At the upper and lower boundaries of the fluid domain, we impose zero velocity,
i.e., $\xx u = \xx 0$, and the barrier is fixed on the left and right sides, i.e., $\xx d = \xx 0$.
To the system, initially at rest, is prescribed an inlet velocity $\xx u_{in} (t)
= ( u_{in,x} (t) , 0 ) \, \text{cm/s}$ to the upper
compartment from the left boundary, where
$$
u_{in, x} (t) =
\begin{cases}
    10 t & t \le 0.1, \\
    0 & \text{otherwise},
\end{cases}
$$
while we prescribe a homogeneous Neumann condition to the other three ends of
the two compartments. The fluid-structure interface $\Sigma$, namely the upper
and lower boundaries of the barrier, is free to move; see
Figure~\ref{fig:domainCompPres}.

In the discrete setting, we set $\Delta t = 10^{-3} \, \text{s}$, $T = 0.25 \, \text{s}$, $\gamma_v = \gamma_p = 10$ and $\delta = 1$; see Remark~\ref{rem:delta}.
The fluid and structure meshes initially consist of uniform and regular triangles consisting of $1100$ elements ($h = 0.025\, \text{cm}$) and $400$ elements ($h = 0.01 \, \text{cm}$), respectively.
Due to their intersection, polygonal elements appear. We employ the Backward Difference Formula (BDF) scheme of order~$3$ for the temporal discretization.

We pick the pairs of velocity and pressure spaces $\mathcal P^{\ell} - \mathcal P^m$, with $\ell = 3$ and $m = 1, 2, 3$, both with and without the pressure stabilization term~\eqref{eq:form_s}.
The spatial polynomial order of discrete displacement field is set equal to $\ell = 3$.

In Figure~\ref{fig:pressureFields}, we plot the pressure field at time $t = 0.1 \, \text{s}$ for all the considered configurations.
As expected, for a fixed pair of spaces $\mathcal P^3 - \mathcal P^m$, $m=1,2,3$, the stabilized pair yields a stable and more regular pressure field compared to the not stabilized one.
The not stabilized $\mathcal P^3 - \mathcal P^3$ pair, Figure~\ref{fig:pressureField_33}, leads to an oscillating pressure field near the inlet boundary and all along the fluid-structure interface, where elements of general shape appear.
This instabilities become less evident as the pressure polynomial order decreases; see Figure~\ref{fig:pressureField_32} and Figure~\ref{fig:pressureField_31}.
On the other hand, the stabilized $\mathcal P^3 - \mathcal P^3$ pair, Figure~\ref{fig:pressureField_33_stab}, shows some oscillations only at the corners of the inlet boundary, where we expect a lower regularity in the solution, a pressure peak and strong pressure gradients.
For the stabilized $\mathcal P^3 - \mathcal P^2$ and $\mathcal P^3 - \mathcal P^1$ cases, see Figure~\ref{fig:pressureField_32_stab} and Figure~\ref{fig:pressureField_31_stab}, the pressure field does not present any noticeable oscillation.
Moreover, there is no significant difference in the fluid velocity and structure displacement fields.

\begin{figure}[!htbp]
    \centering
    \begin{subfigure}{0.46\textwidth}
    \includegraphics[width=\textwidth]{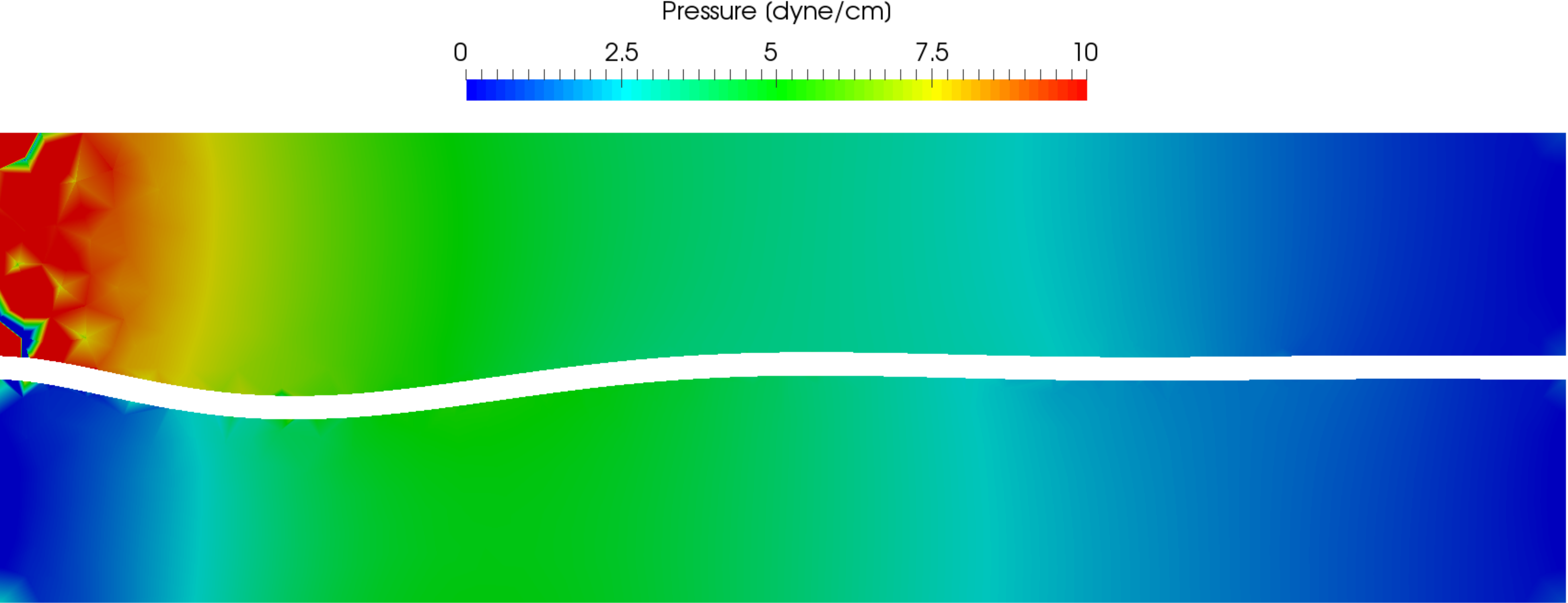}
    \caption{$\mathcal P^3 - \mathcal P^3$, stabilized}
    \label{fig:pressureField_33_stab}
    \end{subfigure}
%%%%%%%%%%%%%%%%%%%%
    \begin{subfigure}{0.46\textwidth}
    \includegraphics[width=\textwidth]{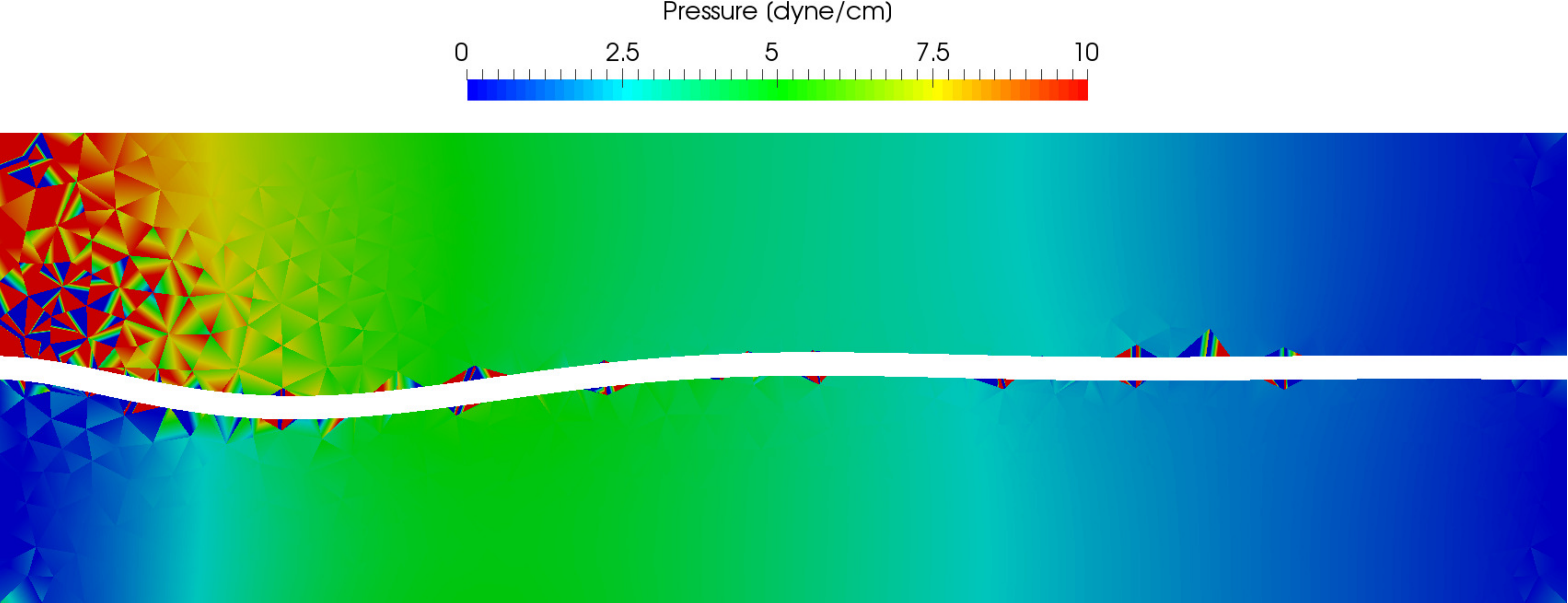}
    \caption{$\mathcal P^3 - \mathcal P^3$}
    \label{fig:pressureField_33}
    \end{subfigure}
%%%%%%%%%%%%%%%%%%%%
    \begin{subfigure}{0.46\textwidth}
    \includegraphics[width=\textwidth]{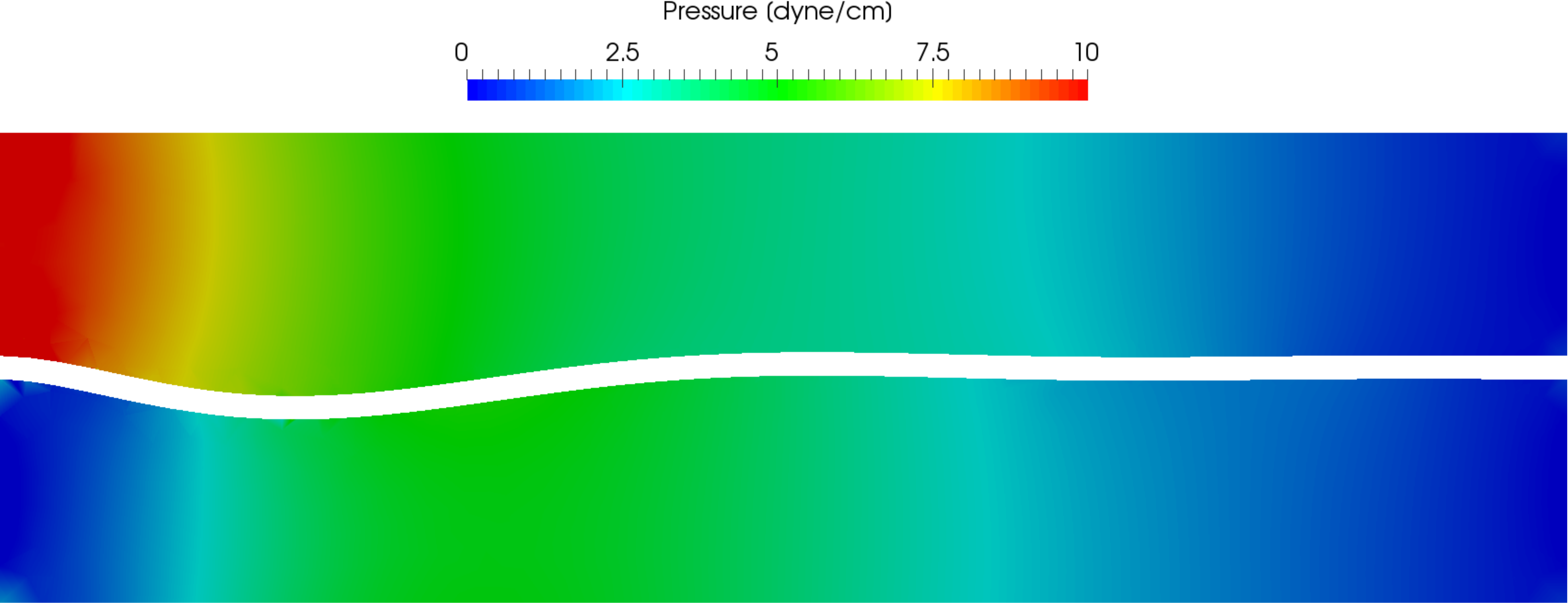}
    \caption{$\mathcal P^3 - \mathcal P^2$, stabilized}
    \label{fig:pressureField_32_stab}
    \end{subfigure}
%%%%%%%%%%%%%%%%%%%%
    \begin{subfigure}{0.46\textwidth}
    \includegraphics[width=\textwidth]{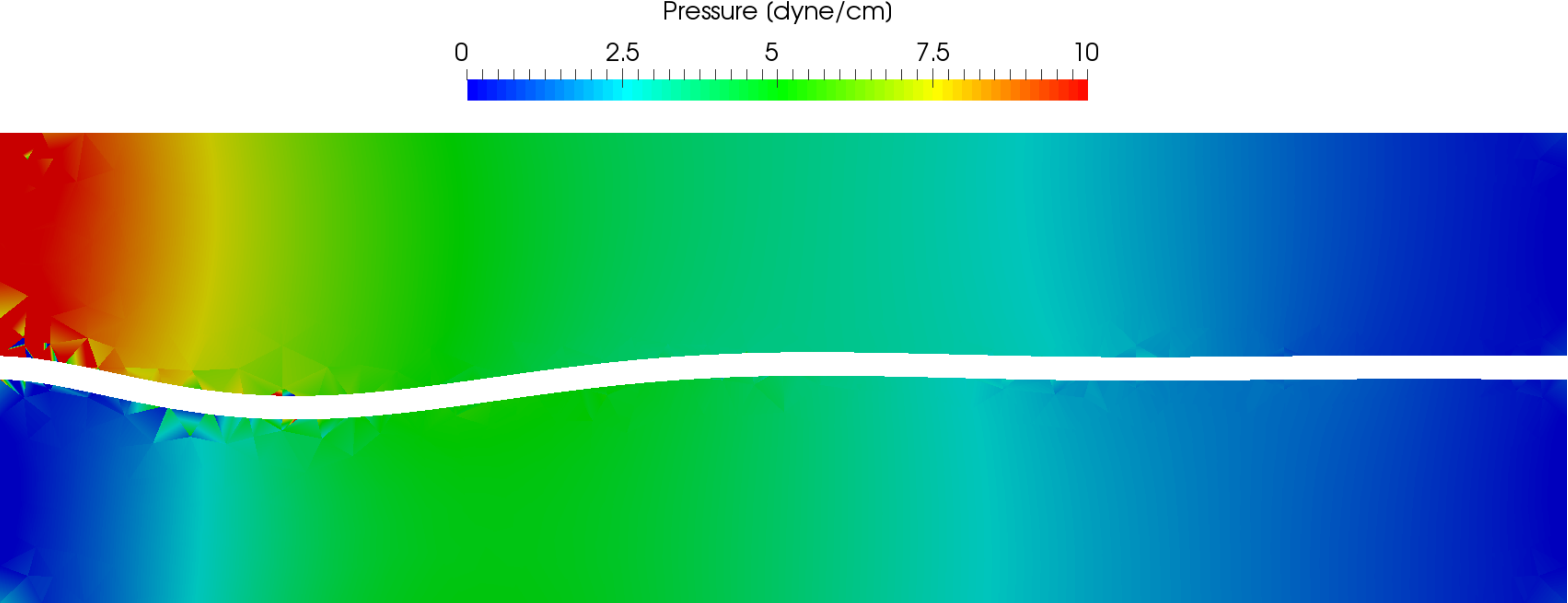}
    \caption{$\mathcal P^3 - \mathcal P^2$}
    \label{fig:pressureField_32}
    \end{subfigure}
%%%%%%%%%%%%%%%%%%%%
    \begin{subfigure}{0.46\textwidth}
    \includegraphics[width=\textwidth]{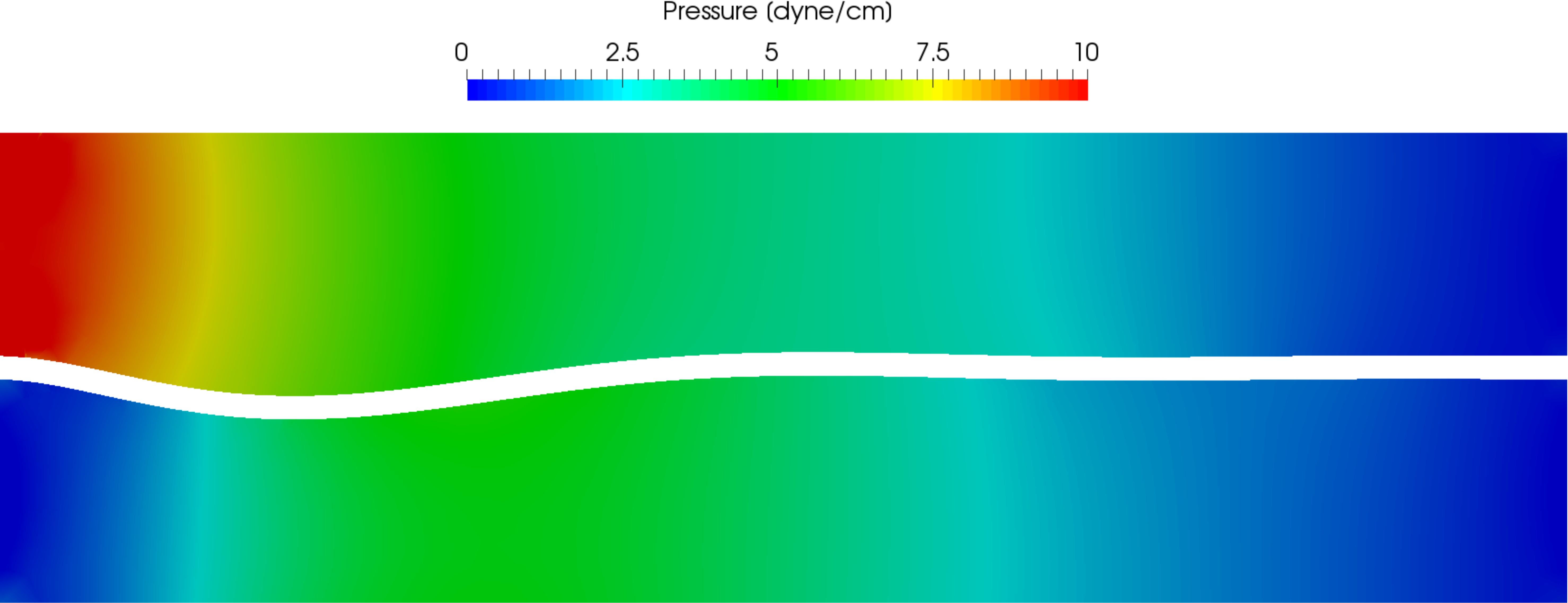}
    \caption{$\mathcal P^3 - \mathcal P^1$, stabilized}
    \label{fig:pressureField_31_stab}
    \end{subfigure}
%%%%%%%%%%%%%%%%%%%%
    \begin{subfigure}{0.46\textwidth}
    \includegraphics[width=\textwidth]{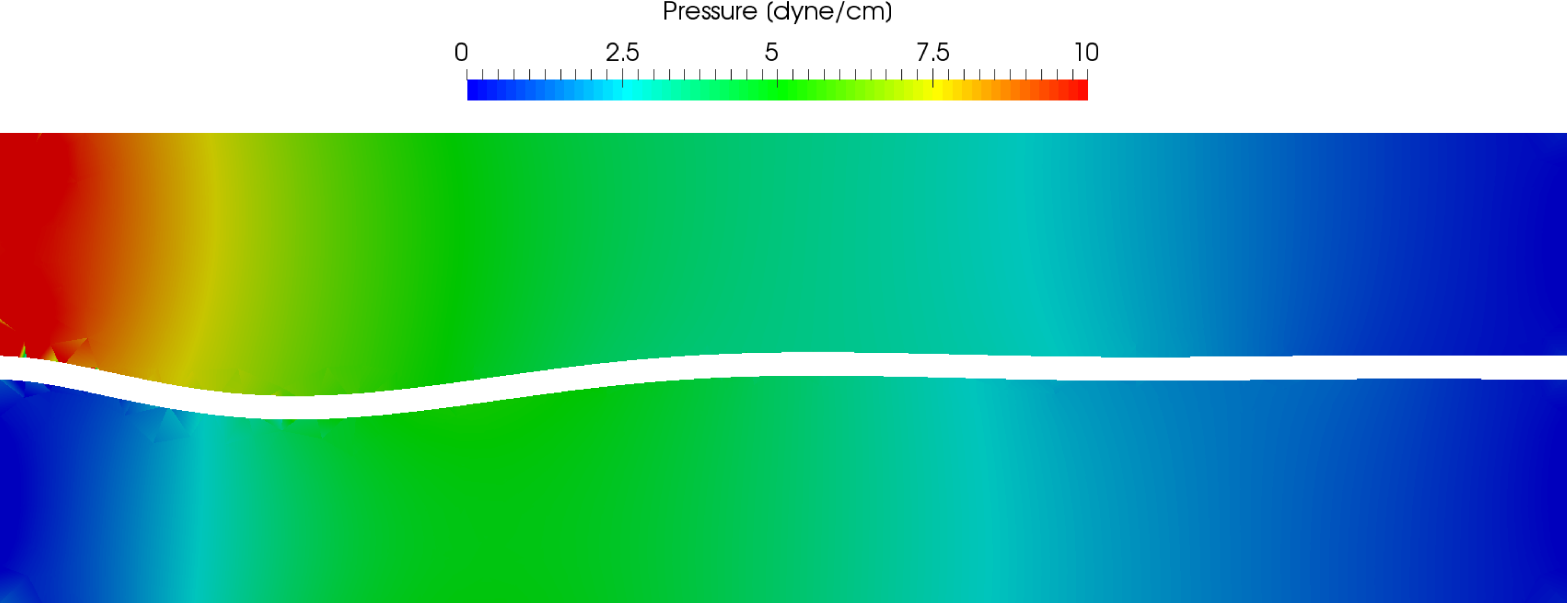}
    \caption{$\mathcal P^3 - \mathcal P^1$}
    \label{fig:pressureField_31}
    \end{subfigure}
%%%%%%%%%%%%%%%%%%%%
    \caption{Pressure field at time $t = 0.1 \, \text{s}$ for different choices
    of the pressure space $\mathcal P^m$, $m= 1, 2, 3$, with pressure
    stabilization (left) and without stabilization (right). The velocity space
    is fixed to $\mathcal P^3$.}
    \label{fig:pressureFields}
\end{figure}

\subsection{The fluid-structure interaction problem: an elastic membrane in a pipe} \label{sec:ex_fsi}

Here, we consider a second fluid-structure interaction problem aiming
at showing that the proposed PolyDG discretization method is able to reproduce the expected dynamics of the
system with a stable pressure field.
More precisely, we consider a pipe filled by a viscous fluid with an immersed linear elastic
membrane that blocks the flow.  The pipe is represented by a fluid domain
$\Omega$ of size $0.4 \, \text{cm} \, \times 0.2 \, \text{cm}$, while the
solid domain $\Omega_s$ represents the elastic membrane of size $0.01
\,\text{cm} \, \times 0.2 \, \text{cm}$ centred in the pipe; see Figure~\ref{fig:blocked_meshes}.
At initial time, the system is at rest. The membrane is clamped at the pipe,
i.e., $\xx d = \xx 0$ on the upper and lower sides of $\Omega_s$.  At the top
and bottom boundaries of the fluid domain, $\xx u = \xx 0$, while on the left
and right sides we prescribe a jump in the stress, namely, $\xx \sigma_f
\xx n = (-10 , 0) \, \text{dyne/cm}$ and $\xx \sigma_f \xx n = \xx 0$,
respectively.
This induces oscillations in the structure, which are subsequently dumped
by the viscous fluid until a steady state is reached. At the steady state, we
expect a uniform pressure inside each chamber of the pipe.  The fluid and structure
have the following material properties: $\rho = \rho_s = 1 \, \text{g/cm}^2$,
$\mu = 0.1 \, \text{g/s}$, $E = 10^4 \, \text{dyne/cm}$ and $\nu = 0.45$.

For the numerical simulation, we consider a fluid mesh consisting of $1400$ elements ($h= 0.0125 \, \text{cm}$) and a solid mesh consisting of $400$ elements ($h = 0.004 \, \text{cm}$); see Figure~\ref{fig:blocked_meshes}.
Although the meshes are initially made of regular triangles, their intersection generates elements of general shape.
We consider the following discrete parameters: $\Delta t= 0.002 \, \text{s}$,
$\gamma_v = \gamma_p = 10$, $\delta = 1$, see Remark~\ref{rem:delta}, and $\ell
= 3$, $m = 2$ with pressure stabilization.
For the time discretization, we employ the $3$-rd order BDF scheme.

\begin{figure}[h]
    \centering
    \includegraphics[width=0.60\textwidth]{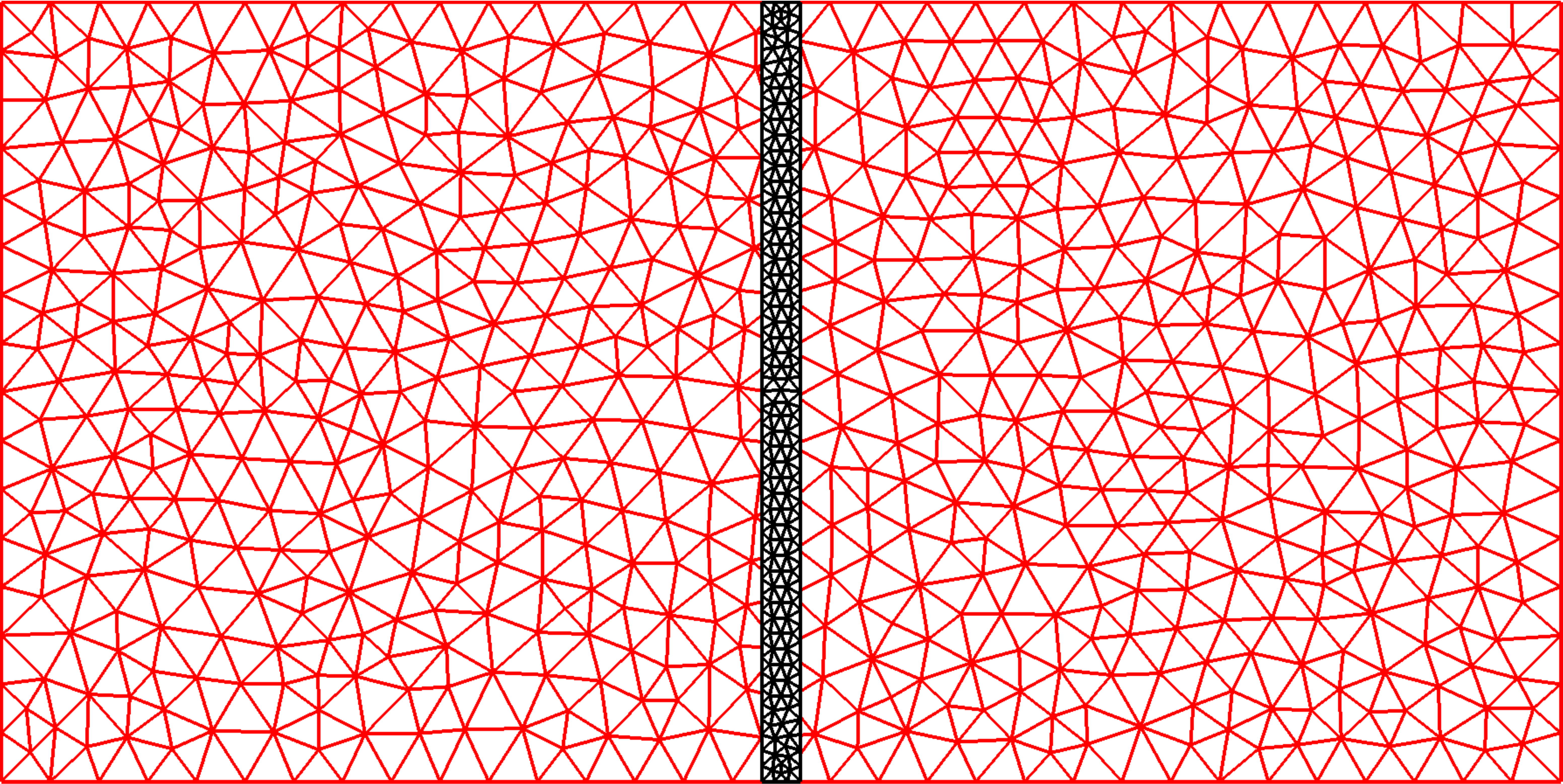}
    \caption{``Pipe'' test case. Fluid (red) and structure (black) meshes.}
    \label{fig:blocked_meshes}
\end{figure}

In Figure~\ref{fig:blocked_fig} (left), we show the configuration at the steady
state, $t = 1 \, s$. As expected, each of the two chambers of the pipe reach a
uniform value of the pressure. In Figure~\ref{fig:blocked_fig} (right), we plot
the $x$-displacement of the structure at its center of mass.

\begin{figure}[h]
    \centering
    \includegraphics[height=0.21\textheight]{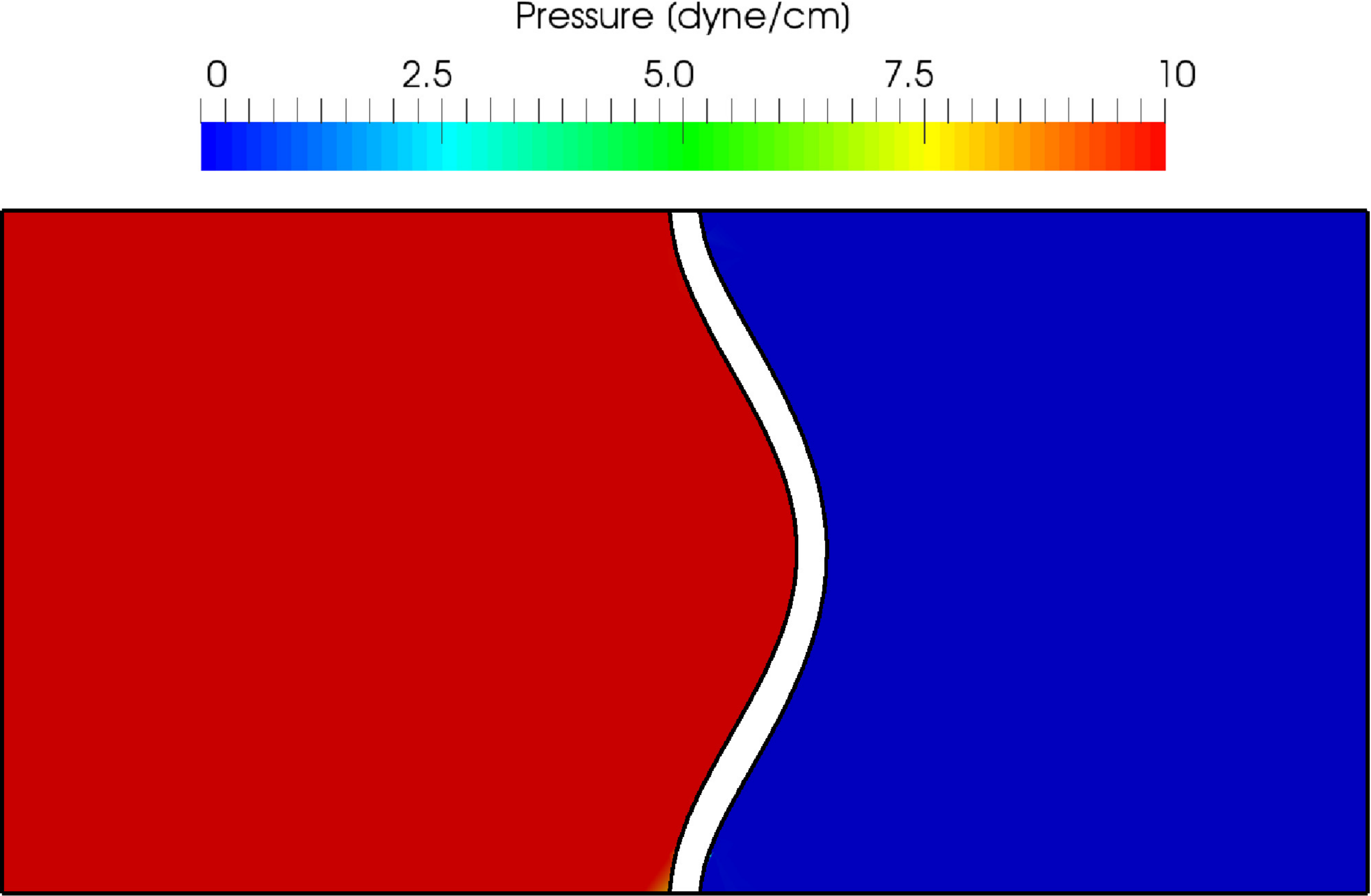}
    \includegraphics[height=0.21\textheight]{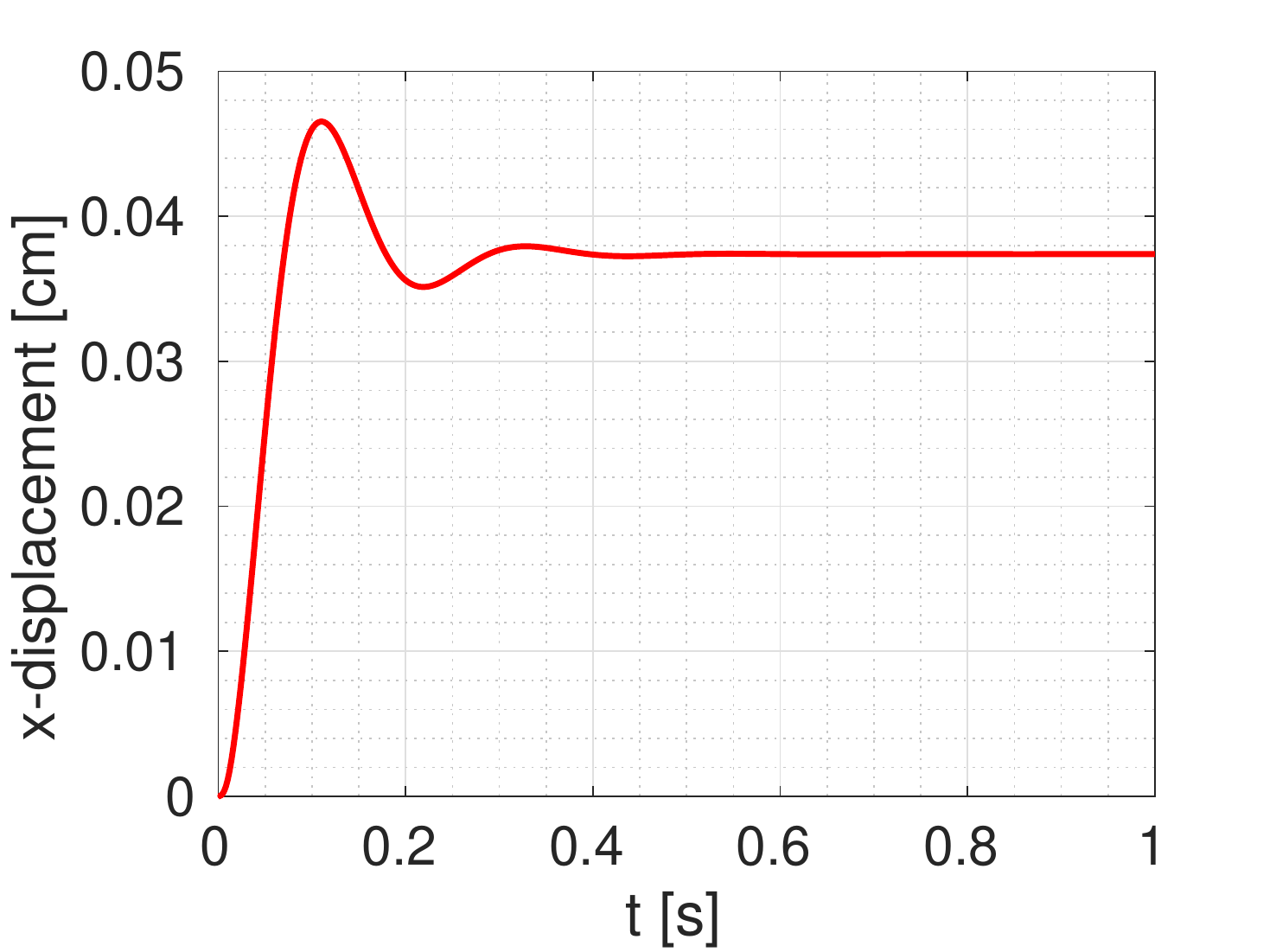}%
    \caption{``Pipe'' test case. Left: pressure field and position of the structure at the steady
    state ($t = 1 \, \text{s}$). Right: evolution in time of the
    $x$-displacement of the structure at the center of mass.}
    \label{fig:blocked_fig}
\end{figure}

\section{Conclusions}\label{sec:conclusions}

In this work, we showed the well-posedness of the discrete Stokes problem obtained via a Discontinuous Galerkin approximation for polygonal and polyhedral grids.
In particular, we proved a generalized \textit{inf-sup} condition that is valid for $m - \ell \le 1$, with $\ell$ and $m$ the spatial polynomial degrees for the velocity and pressure spaces, respectively.
Under suitable mesh assumptions, we proved that the discrete \textit{inf-sup} constant is uniform with respect to the mesh size and presents a mild dependence with respect to the spatial polynomial degree.
Moreover, from the numerical tests, the discrete \textit{inf-sup} constant seems to be independent of the size of the edges in much more general configurations than those addressed theoretically, indicating that the method is robust with respect to degenerating edges.
We also proved a priori error estimate in the energy norm for the Stokes problem that is suboptimal with respect to the polynomial degree, since it inherits the suboptimality of the discrete \textit{inf-sup} constant.
Finally, we presented numerical examples by considering a time-dependent fluid-structure interaction problem in the case of large displacement regime showing that the proposed PolyDG method is able to produce stable solutions.

%%%% Acknowledgments %%%%%%%%
\section*{Acknowledgments}
PFA, LM, MV and SZ are member of the INdAM Research group GNCS and this work is
partially funded by INDAM-GNCS. PFA, MV and SZ have been partially funded by the
PRIN Italian research grant n. 201744KLJL funded by MIUR.

\addcontentsline{toc}{chapter}{Bibliography}
\hyphenpenalty=10 %

\bibliographystyle{plain}
\bibliography{biblio}

\begin{thebibliography}{10}

\bibitem{ager2019poro}
C.~Ager, B.~Schott, A.-T. Vuong, A.~Popp, and W.~A. Wall.
\newblock A consistent approach for fluid-structure-contact interaction based
  on a porous flow model for rough surface contact.
\newblock {\em Internat. J. Numer. Methods Engrg.}, 2019.

\bibitem{aghili2015hybridization}
J.~Aghili, S.~Boyaval, and D.~A. Di~Pietro.
\newblock Hybridization of {M}ixed {H}igh-{O}rder {M}ethods on {G}eneral
  {M}eshes and {A}pplication to the {S}tokes {E}quations.
\newblock {\em Computational Methods in Applied Mathematics}, 15(2):111--134,
  2015.

\bibitem{Alauzet16}
F.~Alauzet, B.~Fabr\`eges, M.~A. Fern\'andez, and M.~Landajuela.
\newblock {Nitsche-XFEM} for the coupling of an incompressible fluid with
  immersed thin-walled structures.
\newblock {\em Comput. Methods Appl. Mech. Engrg.}, 301:300--335, 2016.

\bibitem{antonietti2016review}
P.~F. Antonietti, A.~Cangiani, J.~Collis, Z.~Dong, E.~H. Georgoulis, S.~Giani,
  and P.~Houston.
\newblock {R}eview of {D}iscontinuous {G}alerkin {F}inite {E}lement {M}ethods
  for {P}artial {D}ifferential {E}quations on {C}omplicated {D}omains.
\newblock In {\em Building Bridges: Connections and Challenges in Modern
  Approaches to Numerical Partial Differential Equations}, pages 281--310.
  Springer, 2016.

\bibitem{antoniettiF2020}
P.~F. Antonietti, C.~Facciol\'a, P.~Houston, I.~Mazzieri, G.~Pennesi, and
  M.~Verani.
\newblock High-{O}rder {D}iscontinuous {G}alerkin {M}ethods on {P}olyhedral
  {G}rids for {G}eophysical {A}pplications: {S}eismic {W}ave {P}ropagation and
  {F}ractured {R}eservoir {S}imulations.
\newblock In {\em SEMA-SIMAI Springer series on Polyhedral Methods in
  Geosciences}. Springer, 2020.

\bibitem{antonietti2013hp}
P.~F. Antonietti, S.~Giani, and P.~Houston.
\newblock $hp$-version composite {D}iscontinuous {G}alerkin methods for
  elliptic problems on complicated domains.
\newblock {\em SIAM J. Sci. Comput.}, 35(3):A1417--A1439, 2013.

\bibitem{antoniettim1}
P.~F. Antonietti and I.~Mazzieri.
\newblock High-order {D}iscontinuous {G}alerkin methods for the elastodynamics
  problem on polygonal and polyhedral meshes.
\newblock {\em Comput. Methods Appl. Mech. Engrg.}, 342:414--437, 2018.

\bibitem{antonietti2019numerical}
P.~F. Antonietti, M.~Verani, C.~Vergara, and S.~Zonca.
\newblock Numerical solution of fluid-structure interaction problems by means
  of a high order {D}iscontinuous {G}alerkin method on polygonal grids.
\newblock {\em Finite Elem. Anal. Des.}, 159:1--14, 2019.

\bibitem{babuskaHP87}
I.~Babu{\v{s}}ka and M.~Suri.
\newblock The $h$-$p$ version of the finite element method with quasi-uniform
  meshes.
\newblock {\em ESAIM Math. Model. Numer. Anal.}, 21(2):199--238, 1987.

\bibitem{babuvska1987optimal}
I.~Babu{\v{s}}ka and M.~Suri.
\newblock The optimal convergence rate of the $p$-version of the finite element
  method.
\newblock {\em SIAM J. Numer. Anal.}, 24(4):750--776, 1987.

\bibitem{Basetal12}
F.~Bassi, L.~Botti, A.~Colombo, D.~A. {Di Pietro}, and P.~Tesini.
\newblock On the flexibility of agglomeration based physical space
  discontinuous {G}alerkin discretizations.
\newblock {\em J. Comput. Phys.}, 231(1):45--65, 2012.

\bibitem{da2020equilibrium}
L.~Beir{\~a}o Da~Veiga, C.~Canuto, R.~H. Nochetto, and G.~Vacca.
\newblock Equilibrium analysis of an immersed rigid leaflet by the virtual
  element method.
\newblock {\em arXiv preprint arXiv:2007.09130}, 2020.

\bibitem{da2017divergence}
L.~Beir{\~a}o~da Veiga, C.~Lovadina, and G.~Vacca.
\newblock Divergence free virtual elements for the {S}tokes problem on
  polygonal meshes.
\newblock {\em ESAIM Math. Model. Numer. Anal.}, 51(2):509--535, 2017.

\bibitem{BrezziFortin}
D.~Boffi, F.~Brezzi, and M.~Fortin.
\newblock {\em Mixed {F}inite {E}lement {M}ethods and {A}pplications},
  volume~44.
\newblock Springer Series in Computational Mathematics, 2013.

\bibitem{boffi2017fictitious}
D.~Boffi and L.~Gastaldi.
\newblock A fictitious domain approach with {L}agrange multiplier for
  fluid-structure interactions.
\newblock {\em Numer. Math.}, 135(3):711--732, 2017.

\bibitem{Borazjani13}
I.~Borazjani.
\newblock Fluid--structure interaction, immersed boundary-finite element method
  simulations of bio-prosthetic heart valves.
\newblock {\em Comput. Methods Appl. Mech. Engrg.}, 257:103--116, 2013.

\bibitem{borker2019mesh}
R.~Borker, D.~Huang, S.~Grimberg, C.~Farhat, P.~Avery, and J.~Rabinovitch.
\newblock Mesh adaptation framework for embedded boundary methods for
  computational fluid dynamics and fluid-structure interaction.
\newblock {\em Internat. J. Numer. Methods Fluids}, 90(8):389--424, 2019.

\bibitem{bouaanani2014effects}
N.~Bouaanani and S.~Renaud.
\newblock Effects of fluid--structure interaction modeling assumptions on
  seismic floor acceleration demands within gravity dams.
\newblock {\em Eng. Struct.}, 67:1--18, 2014.

\bibitem{braess2000approximation}
D.~Braess and Ch. Schwab.
\newblock Approximation on simplices with respect to weighted {S}obolev norms.
\newblock {\em J. Approx. Theory}, 103(2):329--337, 2000.

\bibitem{brezzi1974existence}
F.~Brezzi.
\newblock On the existence, uniqueness and approximation of saddle-point
  problems arising from {L}agrangian multipliers.
\newblock {\em Publications math{\'e}matiques et informatique de Rennes},
  S4:1--26, 1974.

\bibitem{burman:hal-02519896}
E.~Burman, G.~Delay, and A.~Ern.
\newblock {An unfitted hybrid high-order method for the Stokes interface
  problem}.
\newblock HAL Id: hal-02519896, July 2020.

\bibitem{burmanFF20}
E.~Burman, M.~A. Fern{\'a}ndez, and S.~Frei.
\newblock A {N}itsche-based formulation for fluid-structure interactions with
  contact.
\newblock {\em ESAIM Math. Model. Numer. Anal.}, 54(2):531--564, 2020.

\bibitem{cangiani17Spacetime}
A.~Cangiani, Z.~Dong, and E.~H. Georgoulis.
\newblock $hp$-version space-time discontinuous {G}alerkin methods for
  parabolic problems on prismatic meshes.
\newblock {\em SIAM J. Sci. Comput.}, 39(4):A1251--A1279, 2017.

\bibitem{CaDoGe2019}
A.~Cangiani, Z.~Dong, and E.~H. Georgoulis.
\newblock $hp$-version discontinuous {G}alerkin methods on essentially
  arbitrarily-shaped elements, 2019.

\bibitem{cangiani2016hp}
A.~Cangiani, Z.~Dong, E.~H. Georgoulis, and P.~Houston.
\newblock $hp$-{V}ersion discontinuous {G}alerkin methods for
  advection-diffusion-reaction problems on polytopic meshes.
\newblock {\em ESAIM Math. Model. Numer. Anal.}, 50(3):699--725, 2016.

\bibitem{cangiani2017hp}
A.~Cangiani, Z.~Dong, E.~H. Georgoulis, and P.~Houston.
\newblock {\em hp-{V}ersion {D}iscontinuous {G}alerkin {M}ethods on {P}olygonal
  and {P}olyhedral {M}eshes}.
\newblock Springer, 2017.

\bibitem{CangianiGeorgoulisHouston_2014}
A.~Cangiani, E.~H. Georgoulis, and P.~Houston.
\newblock {$hp$}-version discontinuous {G}alerkin methods on polygonal and
  polyhedral meshes.
\newblock {\em Math. Models Methods Appl. Sci.}, 24(10):2009--2041, 2014.

\bibitem{cockburn2002local}
B.~Cockburn, G.~Kanschat, D.~Sch{\"o}tzau, and C.~Schwab.
\newblock Local discontinuous {G}alerkin methods for the {S}tokes system.
\newblock {\em SIAM J. Numer. Anal.}, 40(1):319--343, 2002.

\bibitem{Court15}
S.~Court and M.~Fourni\'e.
\newblock A fictitious domain finite element method for simulations of
  fluid–structure interactions: The {N}avier-{S}tokes equations coupled with
  a moving solid.
\newblock {\em J. Fluid. Struct.}, 55:398--408, 2015.

\bibitem{di2010discrete}
D.~A. Di~{P}ietro and A.~Ern.
\newblock Discrete functional analysis tools for discontinuous {G}alerkin
  methods with application to the incompressible {N}avier-{S}tokes equations.
\newblock {\em Math. Comp.}, 79(271):1303--1330, 2010.

\bibitem{dipietro_ern_2010}
D.~A. Di~Pietro and A.~Ern.
\newblock {\em Mathematical aspects of discontinuous {G}alerkin methods},
  volume~69 of {\em Math\'{e}matiques \& Applications}.
\newblock Springer, Heidelberg, 2012.

\bibitem{di2016discontinuous}
D.~A. Di~Pietro, A.~Ern, A.~Linke, and F.~Schieweck.
\newblock A discontinuous skeletal method for the viscosity-dependent {S}tokes
  problem.
\newblock {\em Computer Methods in Applied Mechanics and Engineering},
  306:175--195, 2016.

\bibitem{donea2}
J.~Donea.
\newblock An arbitrary {Lagrangian-Eulerian} finite element method for
  transient dynamic fluid-structure interaction.
\newblock {\em Comput. Methods Appl. Mech. Engrg.}, 33:689--723, 1982.

\bibitem{dumbser2018staggered}
M.~Dumbser, F.~Fambri, I.~Furci, M.~Mazza, S.~Serra-Capizzano, and M.~Tavelli.
\newblock Staggered discontinuous {G}alerkin methods for the incompressible
  {N}avier--{S}tokes equations: {S}pectral analysis and computational results.
\newblock {\em Numerical Linear Algebra with Applications}, 25(5):e2151, 2018.

\bibitem{ern2013theory}
A.~Ern and J.-L. Guermond.
\newblock {\em Theory and practice of finite elements}, volume 159.
\newblock Springer Science \& Business Media, 2013.

\bibitem{fedele2017patient}
M.~Fedele, E.~Faggiano, L.~Dede, and A.~Quarteroni.
\newblock A patient-specific aortic valve model based on moving resistive
  immersed implicit surfaces.
\newblock {\em Biomech. Model. Mechan.}, 16(5):1779--1803, 2017.

\bibitem{gerstenberger2008extended}
A.~Gerstenberger and W.~A. Wall.
\newblock An extended finite element method/{L}agrange multiplier based
  approach for fluid--structure interaction.
\newblock {\em Comput. Methods Appl. Mech. Engrg.}, 197(19):1699--1714, 2008.

\bibitem{ghosh2020numerical}
R.~P. Ghosh, G.~Marom, M.~Bianchi, K.~D'souza, W.~Zietak, and D.~Bluestein.
\newblock Numerical evaluation of transcatheter aortic valve performance during
  heart beating and its post-deployment fluid--structure interaction analysis.
\newblock {\em Biomech. Model. in Mechan.}, pages 1--16, 2020.

\bibitem{girault2005discontinuous}
V.~Girault, B.~Rivi{\`e}re, and M.~Wheeler.
\newblock A discontinuous {G}alerkin method with nonoverlapping domain
  decomposition for the {S}tokes and {N}avier-{S}tokes problems.
\newblock {\em Math. Comp.}, 74(249):53--84, 2005.

\bibitem{glowinski2001fictitious}
R.~Glowinski, T.-W. Pan, T.~I. Hesla, D.~D. Joseph, and J.~Periaux.
\newblock A fictitious domain approach to the direct numerical simulation of
  incompressible viscous flow past moving rigid bodies: application to
  particulate flow.
\newblock {\em J. Comput. Phys.}, 169(2):363--426, 2001.

\bibitem{Griffith12}
B.~E. Griffith.
\newblock Immersed boundary model of aortic heart valve dynamics with
  physiological driving and loading conditions.
\newblock {\em Int. J. Numer. Methods Biomed. Eng.}, 28(3):317--345, 2012.

\bibitem{hairer1993solving}
E.~Hairer, S.~P. N{\o}rsett, and G.~Wanner.
\newblock {\em Solving Ordinary Differential Equations I. Nonstiff Problems}.
\newblock Springer-Verlag Berlin Heidelberg, 1993.

\bibitem{hansbo2002discontinuous}
P.~Hansbo and M.~G. Larson.
\newblock Discontinuous {G}alerkin methods for incompressible and nearly
  incompressible elasticity by {N}itsche's method.
\newblock {\em Comput. Methods Appl. Mech. Engrg.}, 191(17-18):1895--1908,
  2002.

\bibitem{hron2006monolithic}
J.~Hron and S.~Turek.
\newblock A monolithic {FEM}/multigrid solver for an {ALE} formulation of
  fluid-structure interaction with applications in biomechanics.
\newblock In {\em Fluid-structure interaction}, pages 146--170. Springer, 2006.

\bibitem{kamakoti2004fluid}
R.~Kamakoti and W.~Shyy.
\newblock Fluid--structure interaction for aeroelastic applications.
\newblock {\em Progress in Aerospace Sciences}, 40(8):535--558, 2004.

\bibitem{massing2015nitsche}
A.~Massing, M.~G. Larson, A.~Logg, and M.~E. Rognes.
\newblock A {N}itsche-based cut finite element method for a fluid-structure
  interaction problem.
\newblock {\em Commun. Appl. Math. Comput. Sci.}, 10(2):97--120, 2015.

\bibitem{Mittal05}
R.~Mittal and G.~Iaccarino.
\newblock Immersed boundary methods.
\newblock {\em Annu. Rev. Fluid Mech.}, 37(1):239--261, 2005.

\bibitem{picelli2020topology}
R.~Picelli, S.~Ranjbarzadeh, R.~Sivapuram, R.S. Gioria, and E.C.N. Silva.
\newblock Topology optimization of binary structures under design-dependent
  fluid-structure interaction loads.
\newblock {\em Struct. Multidiscip. Optim.}, pages 1--16, 2020.

\bibitem{schotzau1999mixed}
D.~Sch{\"o}tzau, C.~Schwab, and R.~Stenberg.
\newblock Mixed $hp$-{FEM} on anisotropic meshes {II}: Hanging nodes and tensor
  products of boundary layer meshes.
\newblock {\em Numer. Math.}, 83(4):667--697, 1999.

\bibitem{schotzau2002mixed}
D.~Sch{\"o}tzau, C.~Schwab, and A.~Toselli.
\newblock Mixed $hp$-{DGFEM} for incompressible flows.
\newblock {\em SIAM J. Numer. Anal.}, 40(6):2171--2194, 2002.

\bibitem{schotzau2003stabilized}
D.~Sch{\"o}tzau, C.~Schwab, and A.~Toselli.
\newblock Stabilized $hp$-{DGFEM} for incompressible flow.
\newblock {\em Math. Models Methods Appl. Sci.}, 13(10):1413--1436, 2003.

\bibitem{SchwabpandhpFEM}
Ch. Schwab.
\newblock {\em $p$- and $hp$- {F}inite {E}lement {M}ethods: {T}heory and
  {A}pplications in {S}olid and {F}luid {M}echanics}.
\newblock Clarendon Press Oxford, 1998.

\bibitem{Stein70}
E.~M. Stein.
\newblock {\em Singular Integrals and Differentiability Properties of
  Functions}.
\newblock Princeton University Press, Princeton, NY, 1970.

\bibitem{stenberg1996mixed}
R.~Stenberg and M.~Suri.
\newblock Mixed $hp$ finite element methods for problems in elasticity and
  {S}tokes flow.
\newblock {\em Numer. Math.}, 72(3):367--389, 1996.

\bibitem{talischi2012polymesher}
C.~Talischi, G.~H. Paulino, A.~Pereira, and I.~F.~M. Menezes.
\newblock Polymesher: a general-purpose mesh generator for polygonal elements
  written in {M}atlab.
\newblock {\em Struct. Multidiscip. Optim.}, 45(3):309--328, 2012.

\bibitem{tavelli2018arbitrary}
M.~Tavelli and M.~Dumbser.
\newblock Arbitrary high order accurate space--time discontinuous {G}alerkin
  finite element schemes on staggered unstructured meshes for linear
  elasticity.
\newblock {\em Journal of Computational Physics}, 366:386--414, 2018.

\bibitem{tello2020fluid}
A.~Tello, R.~Codina, and J.~Baiges.
\newblock Fluid structure interaction by means of variational multiscale
  reduced order models.
\newblock {\em Internat. J. Numer. Methods Engrg.}, 121(12):2601--2625, 2020.

\bibitem{terahara2020heart}
T.~Terahara, K.~Takizawa, T.~E. Tezduyar, Y.~Bazilevs, and M.-C. Hsu.
\newblock Heart valve isogeometric sequentially-coupled {FSI} analysis with the
  space--time topology change method.
\newblock {\em Comput. Mech.}, pages 1--21, 2020.

\bibitem{tezduyar2007modelling}
T.~E. Tezduyar and S.~Sathe.
\newblock Modelling of fluid--structure interactions with the space--time
  finite elements: {S}olution techniques.
\newblock {\em Internat. J. Numer. Methods Fluids}, 54(6-8):855--900, 2007.

\bibitem{toselli2002hp}
A.~Toselli.
\newblock $hp$ discontinuous {G}alerkin approximations for the {S}tokes
  problem.
\newblock {\em Math. Models Methods Appl. Sci.}, 12(11):1565--1597, 2002.

\bibitem{toselli2003mixed}
A.~Toselli and C.~Schwab.
\newblock Mixed $hp$-finite element approximations on geometric edge and
  boundary layer meshes in three dimensions.
\newblock {\em Numer. Math.}, 94(4):771--801, 2003.

\bibitem{WiKuMiTaWeDa2013}
D.~Wiresaet, E.~J. Kubatko, C.~E. Michoski, S.~Tanaka, J.~J. Westerink, and
  C.~Dawson.
\newblock {D}iscontinuous {G}alerkin methods with nodal and hybrid modal/nodal
  triangular, quadrilateral, and polygonal elements for nonlinear shallow water
  flow.
\newblock {\em Comput. Methods Appl. Mech. Engrg.}, 270:113--149, 2014.

\bibitem{xu2013large}
D.~Xu, E.~Kaliviotis, A.~Munjiza, E.~Avital, C.~Ji, and J.~Williams.
\newblock Large scale simulation of red blood cell aggregation in shear flows.
\newblock {\em J. Biomech.}, 46(11):1810--1817, 2013.

\bibitem{zhang2007immersed}
L.~T. Zhang and M.~Gay.
\newblock Immersed finite element method for fluid-structure interactions.
\newblock {\em J. Fluid. Struct.}, 23(6):839--857, 2007.

\end{thebibliography}

\end{document}